\documentclass{article}

\usepackage{amsmath, amsfonts, amssymb, amsthm}
\usepackage{geometry}
\usepackage{color}
\usepackage{hyperref}
\hypersetup{colorlinks=true,linkcolor=blue, citecolor=red}

\usepackage{graphicx}
\usepackage{caption, subcaption}
\usepackage{booktabs}
\usepackage{makecell}
\usepackage{multirow}
\usepackage{tikz}
\usetikzlibrary{patterns}

\newcommand{\RR}{\mathbb{R}}
\newcommand{\GG}{\mathbb{G}_S}
\newcommand{\GGb}{\mathbb{G}_{1/2}}
\newcommand{\bbS}{\mathbb{S}}
\newcommand{\bS}{\boldsymbol{S}}
\newcommand{\cN}{\mathcal{N}}
\newcommand{\cR}{\mathcal{R}}
\newcommand{\tr}{\textrm{Tr}}
\newcommand{\conv}{\textrm{conv}}
\newcommand{\xc}{\textnormal{xc}}
\newcommand{\rank}{\textrm{rank}}

\newcommand{\enet}{\epsilon_{\textrm{net}}}
\newcommand{\fs}{f^{\sharp}}
\newcommand{\ff}{f^{\flat}}
\newcommand{\supp}{\textrm{supp}~}
\newcommand{\Trho}{T_{\rho}}

\newcommand{\COR}{\textnormal{COR}}

\newcommand{\cK}{ \mathcal{K}}

\newcommand{\cV}{ \mathcal{V}}

\newcommand{\proj}{\textrm{proj}}
\newcommand{\ext}{\textrm{ext}}

\newcommand{\bOne}{\boldsymbol{1}}
\newcommand{\vol}{\textrm{vol}}
\newcommand{\vrad}{\textrm{vrad}}
\newcommand{\GOE}{\textrm{GOE}}
\newcommand{\Arg}[1]{\text{Arg}\left( #1 \right)}

\newcommand{\Snk}{\bS_+^{n,k}}

\newcommand{\epsbasic}{\epsilon^*}
\newcommand{\epsavg}{\epsilon^*_{\textrm{avg}}}
\newcommand{\epsdavg}{\epsilon^*_{\textrm{dual-avg}}}

\newcommand{\base}[1]{ B_H \left( #1 \right) }
\newcommand{\based}[1]{ B_H^* \left( #1 \right) }
\newcommand{\normop}[1]{ \left\| #1 \right\|_{op}}
\newcommand{\Ind}[1]{ \boldsymbol{1}_{\left\{#1\right\}}}
\newcommand{\inner}[2]{ \left\langle #1,~ #2 \right\rangle}

\DeclareMathOperator*{\cl}{\textrm{cl}}
\DeclareMathOperator*{\cone}{\textrm{cone}}

\newcommand{\bbE}{\mathbb{E}}

\newtheorem{definition}{Definition}[section]
\newtheorem{theorem}{Theorem}
\newtheorem{lemma}{Lemma}[section]
\newtheorem{proposition}{Proposition}
\newtheorem{example}{Example}[section]
\newtheorem{corollary}{Corollary}

\theoremstyle{remark}
\newtheorem{remark}{Remark}

\title{
On Approximations of the PSD Cone by\\a Polynomial Number of Smaller-sized PSD Cones
}
\author{Dogyoon Song \and Pablo A. Parrilo}

\begin{document}
\maketitle

\begin{abstract}
We study the problem of approximating the cone of positive semidefinite (PSD) matrices with a cone that can be 
described by smaller-sized PSD constraints. Specifically, we ask the question: ``how closely can we approximate 
the set of unit-trace $n \times n$ PSD matrices, denoted by $D$, using at most $N$ number of $k \times k$ PSD constraints?'' 
In this paper, we prove lower bounds on $N$ to achieve a good approximation of $D$ by considering two constructions 
of an approximating set.
First, we consider the unit-trace $n \times n$ symmetric matrices that are PSD when restricted to a fixed set of 
$k$-dimensional subspaces in $\RR^n$. We prove that if this set is a good approximation of $D$, then the number 
of subspaces must be at least exponentially large in $n$ for any $k = o(n)$.
Second, we show that any set $S$ that approximates $D$ within a constant approximation ratio must have 
superpolynomial $\bS_+^k$-extension complexity. 
To be more precise, if $S$ is a constant factor approximation of $D$, then $S$ must have $\bS_+^k$-extension complexity 
at least $\exp( C \cdot \min \{ \sqrt{n}, n/k  \})$ where $C$ is some absolute constant. 
In addition, we show that any set $S$ such that $D \subseteq S$ and the Gaussian width of $D$ is at most a constant 
times larger than the Gaussian width of $D$ must have $\bS_+^k$-extension complexity at least $\exp( C \cdot 
\min \{ n^{1/3}, \sqrt{n/k}  \})$.
These results imply that the cone of $n \times n$ PSD matrices cannot be approximated by a polynomial number 
of $k \times k$ PSD constraints for any $k = o(n / \log^2 n)$.
These results generalize the recent work of Fawzi \cite{fawzi2018polyhedral} on the hardness of polyhedral approximations 
of $\bS_+^n$, which corresponds to the special case with $k=1$.
\end{abstract}

\section{Introduction}\label{sec:intro}
Semidefinite programming (SDP) is a branch of convex optimization that considers problems of the form
\begin{equation}\label{eqn:sdp}
	\begin{aligned}
		\text{maximize}	\quad&\langle C, X\rangle\\
		\text{subject to}\quad&\langle A_i, X \rangle = b_i, \quad i = 1, \ldots, m,\\
						&X \in \bS_+^n,
	\end{aligned}
\end{equation}
where $C, A_i$'s are $n \times n$ symmetric matrices and $\bS_+^n$ denotes the cone of $n \times n$ positive semidefinite (PSD) matrices.
SDP has attracted great interest in many fields as a powerful tool to provide theoretical guarantees as well as practical algorithms. 
Although current SDP solvers utilizing interior-point methods can solve an SDP up to arbitrary accuracy, they suffer from large computational cost 
and memory requirement when $n$ is large. Indeed, such scalability issues remain as major challenges for researchers in the field.

To deal with these problems, one can seek an alternative formulation of the problem \eqref{eqn:sdp} that is computationally more tractable. 
For instance, one can try to replace the PSD cone with a computationally more tractable convex cone $\cK$ to approximate the feasible set. 
 %
If $\cK$ is a polyhedral cone, we obtain a linear programming (LP) approximation of \eqref{eqn:sdp}, and if $\cK$ is a second-order cone, 
then we get a second-order cone programming (SOCP) approximation, etc. These approximate conic programs can be solved potentially 
much faster than the original SDP, but possibly at the expense of the quality of the solution.

There arises an inevitable question: ``how much error is incurred in the optimal value of \eqref{eqn:sdp} when $\bS_+^n$ 
is replaced by $\cK$?'' In this work, we study this problem by asking the following question:
\begin{quote}\emph{
	``How closely can we approximate $\bS_+^n$ with a cone $\cK$ 
	that can be described using at most $N$ number of $k \times k$ PSD constraints?''}
\end{quote} 
We remark that we consider in this work global, non-adaptive approximations of $\bS_+^n$ that do not make use of the problem data $C, (A_i, b_i)_{i=1}^m$.

\paragraph{Contributions}
To formally state the question above, we need to specify the notion of approximation as well as 
what `a cone $\cK$ that can be described using at most $N$ number of $k \times k$ PSD constraints' means.

\bigskip\noindent
\emph{Notions of Approximation. }
First, we specify the notions of approximation for cones as follows. Let $H = \{ X \in \bS^n: ~ \tr ~X = 1 \}$ and let 
$\base{\cK} = (\cK \cap H) - \frac{1}{n}I_n$ for any cone $\cK$, where $I_n$ is the $n \times n$ identity matrix. 
That is, $\base{\cK}$ is the unit-trace affine section of $\cK$ translated by $- \frac{1}{n} I_n$; note that $0 \in \base{\bS_+^n}$. 
For $\epsilon > 0$, we say $\cK$ is an \emph{$\epsilon$-approximation} of $\bS_+^n$ if $\base{\bS_+^n} \subseteq 
\base{\cK} \subseteq (1+\epsilon) \base{\bS_+^n}$. 

This notion of approximation is natural and closely related to quantifying the difference in the optimal value (optimality gap) 
induced by relaxing $\bS_+^n$ to $\cK$. Suppose that we are given a problem of the form \eqref{eqn:sdp} with $m = 1$, 
$A_1 = I_n$, and $b_1 = 1$, and we relax the problem by replacing $\bS_+^n$ with a cone $\cK \supseteq \bS_+^n$. 
If $\cK$ is an $\epsilon$-approximation of $\bS_+^n$, then the relative optimality gap is at most $\epsilon$ for all $C \in \bS^n$.


We also define two auxiliary notions of approximation for the convenience of our analysis. 
Observe that the notion of $\epsilon$-approximation requires $\base{\cK}$ to approximate $\base{ \bS_+^n }$ well  in all directions in the ambient space. 
We introduce more lenient notions of approximation by requiring the relative optimality gap to be small only on average for randomized $C$ with standard 
Gaussian distribution. Specifically, $\cK$ is called an $\epsilon$-approximation of $\bS_+^k$ in the \emph{average sense} 
if $\base{\bS_+^n} \subseteq \base{\cK}$ and $w_G\big( \base{\cK} \big) \leq (1+\epsilon) \cdot w_G\big( \base{\bS_+^n} \big)$ 
where $w_G(S) = \bbE_g   \left[ \sup_{x \in S} \langle g, x \rangle \right]$ denotes the Gaussian width of $S$.
Likewise, $\cK$ is called an $\epsilon$-approximation of $\bS_+^k$ in the \emph{dual-average sense} if $\base{\bS_+^n} \subseteq 
\base{\cK}$ and $w_G\big( \base{\cK}^{\circ} \big) \geq \frac{1}{1+\epsilon}\cdot  w_G\big( \base{\bS_+^n}^{\circ} \big)$. 
More details about these notions can be found in Section \ref{sec:notion_approx}.

\bigskip\noindent
\emph{$k$-PSD Approximations of $\bS_+^n$. }
In Section \ref{sec:results.1}, we consider approximating $\bS_+^n$ by enforcing PSD constraints on certain $k$-dimensional subspaces in $\RR^n$. 
We begin by formally defining the $k$-PSD approximation of $\bS_+^n$.
\begin{definition}[$k$-PSD approximation]\label{def:restr_kPSD}
	Let $\cV = \{ V_1, \dots, V_N \}$ be a set of $k$-dimensional subspaces of $\RR^n$. The $k$-PSD approximation of $\bS_+^n$
	induced by $\cV$ is the convex cone
	\[
		\Snk( \cV ) := \big\{ X \in \bS^n: v^T X v \geq 0, ~\forall v \in V_i, ~\forall i = 1, \dots, N\}.
	\]
\end{definition}
\noindent
Note that $\bS_+^n \subseteq \Snk( \cV )$ and that $\Snk( \cV ) $ can be characterized using at most $N = | \cV|$ number of 
$k \times k$ PSD constraints.
A prominent example is the so-called sparse $k$-PSD approximation, denoted by $\Snk$, which is a $k$-PSD approximation of $\bS_+^n$ 
induced by the collection of $N = {n \choose k}$ subspaces of $k$-sparse vectors in $\RR^n$. 


Our first main results (Theorem \ref{thm:main.0} and Corollary \ref{coro:main.0}) state that if $\Snk( \cV )$ is a dual-average $\epsilon$-approximation 
of $\bS_+^n$, then $N \geq \exp\big( n \cdot \max\{ 1/(1+\epsilon) - \sqrt{k/n}, ~ 0 \}^2\big)$ is necessary, 
regardless of the choice of the subspaces $V_1, \dots, V_N$; see Remark \ref{rem:genkPSD_asymp} in Section \ref{sec:restr_kPSD}.
For instance, Corollary \ref{coro:main.0} implies that for any $\epsilon > 0$, $\Snk$ cannot be a dual-average $\epsilon$-approximation 
of $\bS_+^n$ unless $k = \Omega_n(n)$.

We remark that the conclusion of Theorem \ref{thm:main.0} (and Corollary \ref{coro:main.0}) is possibly too conservative, 
especially when the subspaces have overlaps. 
It is because the proof of Theorem \ref{thm:main.0} only takes the number of subspaces into consideration, 
and is oblivious to the configuration of the subspaces in $\cV$.
In Section \ref{sec:coord_kPSD}, we elaborate on this point with an example of the sparse $k$-PSD approximation. 
Although Corollary \ref{coro:main.0} already suggests that $k$ must scale at least linearly as $n$ in order for $\Snk$ to approximate $\bS_+^n$, 
it becomes uninformative once $k/n$ exceeds a certain threshold (approximately $0.137$); see Section \ref{sec:k_coordinate_generic} and Figure \ref{fig:generic_kPSD_delta}. 
%

In Section \ref{eqn:eps_approx}, a tailored analysis for the sparse $k$-PSD approximation is provided. 
To be specific, we consider a carefully designed matrix in $\Snk \setminus \bS_+^n$ to show $\epsbasic(\bS_+^n, \Snk) \geq \frac{n-k}{k-1}$ 
(Proposition \ref{prop:e-approx}) where $\epsbasic(\bS_+^n, \Snk) := \inf\{ \epsilon > 0: \Snk \text{ is an }\epsilon \text{-approximation of }\bS_+^n \}$. 
Furthermore, we prove a sharper lower bound for $\epsdavg(\bS_+^n, \Snk)$ that is strictly positive for all $1 \leq k < n$, 
using the duality between $\Snk$ and the cone of factor width at most $k$ (Proposition \ref{prop:upper_affine}). 
See Figure \ref{fig:sparse_kPSD_overview} for comparison between these tailored results and the weak bound obtained from Corollary \ref{coro:main.0}.

\bigskip\noindent
\emph{Approximate Extended Formulations of $\bS_+^n$. }
Recall that a $k$-PSD approximation of $\bS_+^n$ is the intersection of sets in $\bS^n$, each of which is 
described with a $k \times k$ PSD constraint.
Instead of directly intersecting sets in $\bS^n$, we may introduce additional variables in pursuit of 
a more compact description. To be precise, we can lift $\bS^n$ to a higher-dimensional space by embedding, 
intersect the lifted space with $k \times k$ PSD constraints, and then project the intersection back 
to describe a set in $\bS^n$. 
The resulting description is called an extended formulation of the set, and the preimage of the projection is 
called the lifted representation (or PSD lift) of the set.
The $\bS_+^k$-extension complexity of a set $S$, denoted by $\xc_{\bS_+^k}(S)$, counts the minimum number of 
$k \times k$ PSD constraints required to describe $S$ using extended formulation (i.e., with an arbitrary number of 
additional variables allowed in the description).

In Section \ref{sec:result.extended}, we argue that any set that well approximates $\base{\bS_+^n}$ must have 
$\bS_+^k$-extension complexity at least superpolynomially large in $n$ for all $k$ much smaller than $n$. 
That is, it is impossible to approximate $\base{\bS_+^n}$ using only polynomially many $k \times k$ PSD constraints, 
for any construction of the approximating set.
To be precise, if $S$ is an $\epsilon$-approximation of $\base{\bS_+^n}$, then $ \xc_{\bS_+^k}(S) \geq
	\exp \big( C \cdot \min \big\{ ( \frac{ n }{ 1+\epsilon} )^{1/2}, ~ \frac{1}{1+\epsilon}~ \frac{n}{k} \big\} \big)$ (Theorem \ref{thm:main.1}); 
and if $S$ is an average $\epsilon$-approximation of $\base{\bS_+^n}$, then $ \xc_{\bS_+^k}(S) \geq
	\exp \big( C \cdot \min \big\{ ( \frac{ n }{ (1+\epsilon)^2} )^{1/3} , ~ \frac{1}{1+\epsilon} (\frac{n}{k})^{1/2} \big\} \big)$ (Theorem \ref{thm:main.2}).
These results are visually illustrated in Figure \ref{fig:generic_kPSD_delta_overview}. 
We remark that these results extend \cite[Theorems 1 \& 2]{fawzi2018polyhedral} beyond the special case $k=1$. 

Nevertheless, we do not know whether our extension complexity lower bounds are tight. It might be possible to 
achieve stronger extension complexity lower bounds (i.e., move the curves upward) by means of a more 
sophisticated analysis. We leave it as an interesting open problem.

\bigskip\noindent
\emph{Summary of Results.}
Table \ref{tab:comparison} summarizes the results in this paper. The lower bounds in the table imply the hardness of 
approximating $\bS_+^n$ by using only a small number of $k \times k$ PSD constraints when $k = o(n)$.
%

\begin{table}[t]
	\caption{Overview of our results about the hardness of approximating $\bS_+^n$ with $\bS_+^k$, presented in terms of 
	the number $N$ of the $k \times k$ PSD constraints needed to construct an $\epsilon$-approximation of $\bS_+^n$. 
	Here, $C_1, C_2 > 0$ are some universal constants and $\gtrsim$ indicates that the inequality holds asymptotically in the limit $n \to \infty$.
	}
	\label{tab:comparison}
	\centering
	\resizebox{\columnwidth}{!}{
	\begin{tabular}{ l c c }
		\toprule
			Notion of Approx.			&	$k$-PSD Approximations of $\bS_+^n$	&	Approximate Extended Formulations of $\bS_+^n$\\
		\midrule
			\makecell[l]{$\epsilon$-approx.\\(Definition \ref{defn:approx})}	&	\makecell[c]{$N \gtrsim \exp\Big( n \cdot \max \Big\{ \frac{1}{1+\epsilon} - \sqrt{\frac{k}{n}}, ~ 0\Big\}^2\Big)$\\(Theorem \ref{thm:main.0} \& Corollary \ref{coro:main.0})}								
				&	\makecell[c]{$N \gtrsim \exp \Big( C_1 \cdot \min \big\{ \sqrt{ \frac{ n }{ 1+\epsilon} }, ~ \frac{1}{1+\epsilon}~ \frac{n}{k} \big\} \Big) $\\(Theorem \ref{thm:main.1})}				\\
		\midrule
			\makecell[l]{avg. $\epsilon$-approx.\\(Definition \ref{defn:average_approx})}	&	\makecell[c]{Same lower bound\\as on the right}							
				&	\makecell[c]{$N \gtrsim \exp \Big( C_2 \cdot \min \big\{ \big( \frac{ n }{ (1+\epsilon)^2} \big)^{1/3} , ~ \frac{1}{1+\epsilon} \sqrt{\frac{n}{k}} \big\} \Big) $	\\(Theorem \ref{thm:main.2})}				\\
		\bottomrule
	\end{tabular}
	}
\end{table}

\begin{figure}[h!]
	\centering
	\begin{subfigure}[b]{0.47\textwidth}
		\centering
		\includegraphics[width=\textwidth]{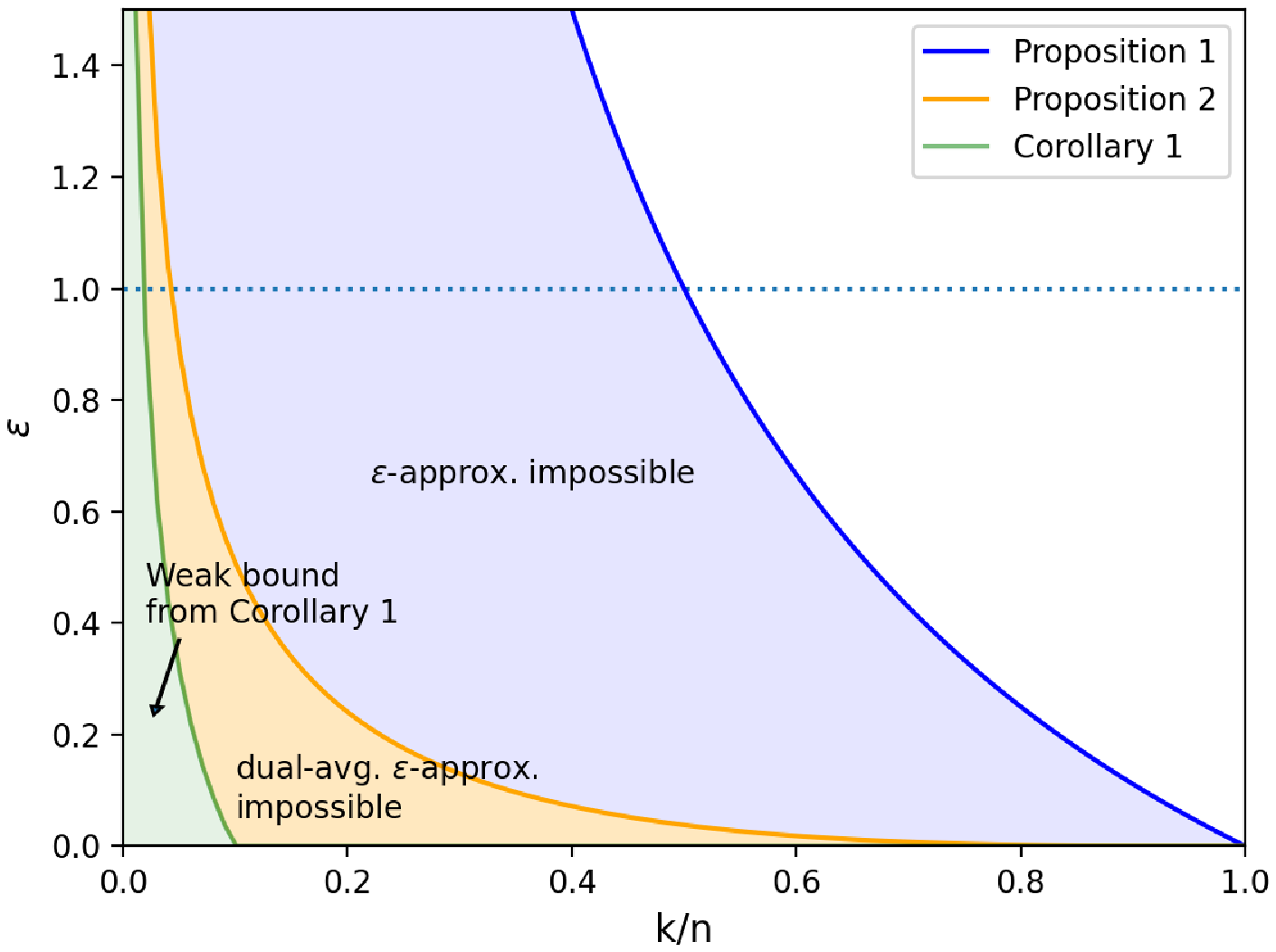}
		\caption{Hardness results for sparse $k$-PSD approximations. Generic (Cor. 1) and tailored bounds (Props. 1 \& 2).}
		\label{fig:sparse_kPSD_overview}
	\end{subfigure}
	\hfill
	\begin{subfigure}[b]{0.47\textwidth}
		\centering
		\includegraphics[width=\textwidth]{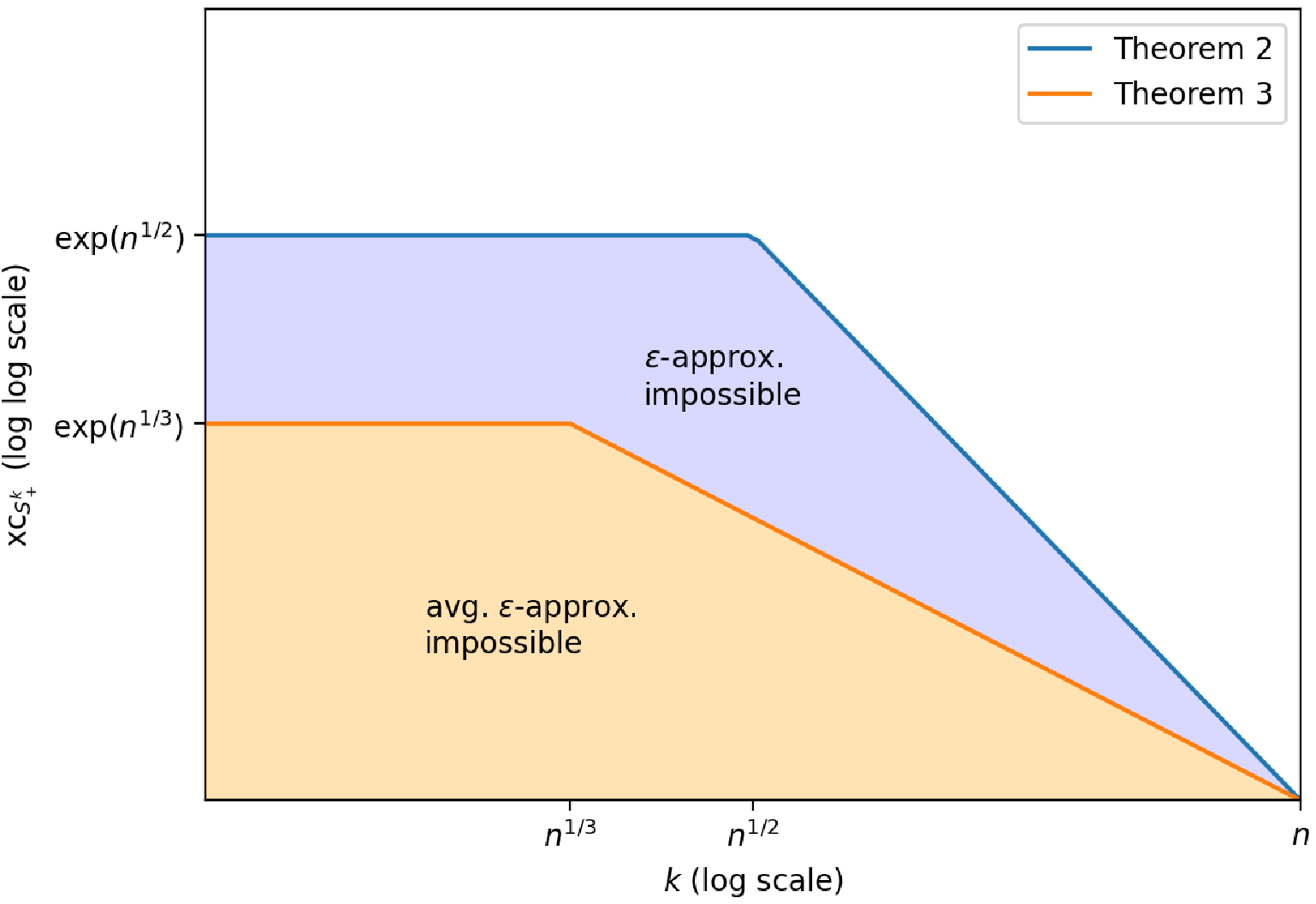}
		\caption{Impossibility of approximating $B_H(\bS_+^n)$ with a polynomial number of $k \times k$ PSD constraints.}
		\label{fig:generic_kPSD_delta_overview}
	\end{subfigure}
	\caption{Summary of our results about the hardness of approximating $\bS_+^n$.}
	\label{fig:summary}
\end{figure}

\newpage
\paragraph{Discussion and Related Work}
Here we make a few comments on our results and some related work.
\begin{itemize}
	\item
	Blekherman et al. \cite{blekherman2020sparse} also investigated the question of how well $\Snk$ approximates $\bS_+^n$. 
	They use the quantity $\overline{ \textrm{dist}}_F(\Snk, \bS_+^n) := \sup_{X \in \Snk, \| X \|_F = 1} \inf_{Y \in \bS_+^n} \| X - Y \|_F$ 
	to measure the quality of approximation, and thus, their result has a connection with our result on $\epsilon$-approximation.
	In this work, we extend the scope of the question in two directions: first, we consider the `average' distance with the notion of 
	average $\epsilon$-approximation as well as the maximal distance; second, our result (Theorem \ref{thm:main.0}) 
	applies to not only $\Snk$, but also $\Snk(\cV)$ with an arbitrary collection of $k$-dimensional subspaces $\cV$.

	\item
	Fawzi \cite{fawzi2018polyhedral} showed that any polytope that well approximates $\base{\bS_+^n}$ must have 
	LP extension complexity at least exponentially large in $n$. Our Theorems \ref{thm:main.1} and \ref{thm:main.2} 
	generalize their results beyond the special case of $k = 1$. 
	Our proof refines and adapts the ideas from \cite{fawzi2018polyhedral} to prove a lower bound for arbitrary $k$. 
	Specifically, we devise a different way of decomposing the component functions of the $\bS_+^k$-factorization 
	of the slack matrix into their sharp and flat parts, which enable us to apply Fawzi's argument even when $k > 1$. 
	In addition, we compare the variance of two representations of the slack matrix instead of their tail probabilities 
	to obtain a nontrivial $\bS_+^k$-extension complexity lower bound even when $k = \Omega_n(\sqrt{n})$. 
	See the proof of Theorem \ref{thm:main.1} in Section \ref{sec:proof_thm.1} for details.

	\item
	Here we compare our Theorem \ref{thm:main.1} (extension complexity lower bound for an $\epsilon$-approximation of $\base{\bS_+^n}$) 
	with a back-of-the-envelope calculation based on known results about the LP extension complexity 
	of $\base{\bS_+^n}$. Assume that there is a set $S$ such that $\xc_{\bS_+^k}(S) = N$ and $S$ is an $\epsilon$-approximation of $\base{\bS_+^n}$. 
	On the one hand, each of the $N$ cones of $k \times k$ PSD matrices can be approximated by $\exp(ck)$ facets of linear inequalities, where 
	$c > 0$ is an absolute constant; see Aubrun and Szarek \cite[Proposition 10]{aubrun2017dvoretzky}. Thus, the LP extension complexity of $S$ 
	is at most $N \exp(ck)$. On the other hand, the LP extension complexity of $S$ is at least $\exp( c' \sqrt{n} )$; see Fawzi \cite[Theorem 1]{fawzi2018polyhedral}. 
	Therefore, we get $N \geq \exp( c' \sqrt{n} - c k )$. This lower bound becomes trivial when $k = \Omega_n(\sqrt{n})$. In contrast, the 
	lower bound from Theorem \ref{thm:main.1} remains superpolynomial as long as $k = o_n( n / \log n)$.

	\item
	We also remark some works that studied lower bounds on the semidefinite extension complexity of polytopes associated with 
	NP-Hard combinatorial problems. Fawzi and Parrilo \cite{fawzi2013exponential} showed that the $\bS_+^k$-extension complexity 
	of the correlation polytope, $\COR(n) := \conv \big\{ vv^T: v \in \{0, 1\}^n \big\}$, is exponentially large in $n$ for any fixed constant $k$. 
	Their proof relies on a combinatorial argument that counts possible sparsity patterns of certain matrices with small PSD rank. 
	Lee, Raghavendra, and Steurer \cite{lee2015lower} proved a stronger lower bound on the semidefinite extension complexity of 
	the correlation polytope, 
	based on the notion of low-degree sum-of-squares proof. While these works consider a similar problem to ours, the object of study is 
	different; in this work, we are interested in the approximate semidefinite extension complexity of the \emph{spectrahedron} $B_H( \bS_+^n )$.

	\item
	Let $D_k = (\Snk)^* \cap H$. In our analysis, $w_G(D_k)$ turns out to be the expectation of the 
	largest $k$-sparse eigenvalue of the Gaussian Orthogonal Ensemble (GOE) (divided by $\sqrt{2}$).
	In this work, we only provide an asymptotic upper bound for $w_G( D_k )$ using Slepian's lemma (Proposition \ref{prop:upper_affine}), 
	however, it might be possible to prove a lower bound for $w_G( D_k )$ with tools from random matrix theory.
	

	\item
	We do not know whether our lower bounds in Theorems \ref{thm:main.1} and \ref{thm:main.2} are tight. 
	We remark that our proof techniques utilize information from the slack matrix only up to degree 2, i.e., up to 
	the second moment. It may be possible to achieve a stronger lower bound by exploiting higher-order moments.

	\item
	In this work, we consider the question of approximating $\bS_+^n \cap H$ and show that at least superpolynomially many 
	$k \times k$ PSD constraints are needed when $k \ll n$. However, if one is allowed to exploit the problem data -- $C, A_i, b_i$ -- 
	it could be still possible to construct a good approximation $F'$ of the feasible set $F = \{ X \in \bS_+^n : \langle A_i, X \rangle = b_i \}$ 
	with a smaller number of $\bS_+^k$ so that the optimality gap $\sup_{X \in F'} \langle C, X \rangle - \sup_{X \in F} \langle C, X \rangle$ 
	is small as empirically evidenced in \cite{ahmadi2019dsos}.
	
\end{itemize}

\paragraph{Organization}
In Section \ref{sec:background}, we review some background materials. In Section \ref{sec:notion_approx}, 
we define the notions of approximation that will be used in this paper. In Section \ref{sec:results.1}, we consider 
the $k$-PSD approximations of $\bS_+^n$. Specifically, Section \ref{sec:restr_kPSD} discusses 
a generic lower bound on the number of subspaces required to approximate $\bS_+^n$, and Section \ref{sec:coord_kPSD} 
provides a more refined analysis tailored to the so-called sparse $k$-PSD approximation of $\bS_+^n$.
In Section \ref{sec:result.extended}, we consider the approximate extended formulations of $\bS_+^n$. 
In Section \ref{sec:main_thms}, we present two main theorems about 
the hardness of approximating $\bS_+^n$. Section \ref{sec:proof_thm.1} and Section \ref{sec:proof_thm.2} 
are dedicated to their proofs.

\paragraph{Notation}
For $x \in \RR$, $[x]_+ := \max\{ x, 0 \}$. For a positive integer $n$, we let $[n] := \{ 1, 2, \dots, n \}$. 
$\RR^n$ denotes the $n$-dimensional real Euclidean space and $\bbS^{n-1}$ is the unit sphere in $\RR^n$. 
We also let $\bS^n$ denote the set of $n \times n$ real symmetric matrices. Given $X \in \bS^n$ and 
$I \subset [n]$, let $X_I \in \bS^{|I|}$ denote the principal submatrix of $X$ with row/column indices in $I$. 
For a matrix $X \in \bS^n$, $\lambda_1(X) \geq \dots \geq \lambda_n(X)$ are the eigenvalues of $X$ 
in descending order. A matrix $X \in \bS^n$ is positive semidefinite, denoted by $X \succeq 0$, 
if $v^T X v \geq 0$ for all $v \in \RR^n$. We let $\bS_+^n := \{ X \in \bS^n: X \succeq 0 \}$. 
The letter $H$ is reserved to indicate the subspace of unit trace: $H = \{ X \in \bS^n: \tr~ X = 1 \}$, 
and $I_n$ denotes the $n \times n$ identity matrix. 
For a cone $\cK \subseteq \bS^n$, its base (translated by $- \frac{1}{n} I_n$) is the compact set defined to be 
$\base{\cK} := \big( \cK \cap H \big) - \frac{1}{n} I_n = \{ X - \frac{1}{n} I_n \in \bS^n: X \in \cK \cap H \}$, 
and we define $\based{\cK} := \base{ \cK^* }$ for notational convenience. 
Given a set $S$, we let $\cl(S)$, $\conv(S)$, and $\cone(S)$ denote the closure, the convex hull, and 
the conical hull of $S$, respectively. Lastly, we let $N(\mu, \Sigma)$ denote the multivariate 
Gaussian distribution with mean $\mu$ and covariance $\Sigma$.

\section{Background}\label{sec:background}
In this section, we review some mathematical preliminaries that are used in our proof of the main theorems. 
Expert readers may want to skip this section and continue reading from Section \ref{sec:notion_approx}.

\subsection{Primer on Convex Geometry}\label{sec:dual}
We recall some basic concepts and results in convex analysis. 
The materials in this section are standard and can be found in classic references; 
we refer the interested readers to \cite{rockafellar1970convex} and \cite{aubrun2017alice} for more details.

\paragraph{Duality}
If $S \subseteq \RR^d$, the polar of $S$ (in $\RR^d$) is the closed convex set
\begin{equation}\label{eqn:polar}
	S^{\circ} := \big\{ y \in \RR^d: \langle x, y \rangle \leq 1 \text{ for all } x \in S \big\}.
\end{equation}
We observe a few properties involving polars. First of all, if $S \subseteq T$, then $S^{\circ} \supseteq T^{\circ}$. 
Next, it is useful to note that $(S \cup T)^{\circ} = S^{\circ} \cap T^{\circ}$ for any $S, T \subseteq \RR^d$, 
and that $(S \cap T)^{\circ} = \cl \conv( S^{\circ} \cup T^{\circ}) $ if $S, T$ are closed, convex, 
and contain the origin. 
Lastly, if $S$ is a closed, convex set that contains the origin, then $(S^{\circ})^{\circ} = S$; this is known as 
the bipolar theorem in convex analysis (see Lemma \ref{lem:mink_supp}).

A nonempty closed convex set $\cK \subset \RR^d$ is called a cone if $\cK$ is invariant under positive scaling, 
i.e., whenever $x \in \cK$ and $t \geq 0$, then $tx \in \cK$. Given a cone $\cK$, its dual cone $\cK^*$ (in $\RR^d$) 
is defined via
\begin{equation}\label{eqn:dual_cone}
	\cK^* := \big\{ y \in \RR^d : \langle x, y \rangle \geq 0 \text{ for all } x \in \cK \big\}.
\end{equation}
Note that $\cK^* = - \cK^{\circ}$. Thus, it follows from the properties of polars that (i) $(\cK^*)^* = \cK$; 
(ii) if $\cK_1 \subseteq \cK_2$, then $\cK_1^* \supseteq \cK_2^*$; and (iii) $\big( \cK_1 \cap \cK_2 \big)^* 
= \cl \cone ( \cK_1 \cup \cK_2 ) $ for two cones $\cK_1, \cK_2$. 

The notion of cone duality is closely related to that of set polarity. 
To clarify the link, we first define a base of a closed convex cone $\cK$. 
Fix a nonzero vector $e \in \RR^d$ and the corresponding affine hyperplane
\[
	H_e := \{ x \in \RR^d: \langle e, x \rangle = \langle e , e \rangle \}.
\]
If $e \in \cK^* \setminus \cK^{\perp}$ where $\cK^{\perp} = \{ v \in \RR^d: \langle v, x \rangle = 0, \forall x \in \cK \}$, 
then we call the set $\cK_e^b := \cK \cap H_e$ as the base of $\cK$ with respect to $e$. The duality of cones 
carries over to a duality of bases as follows.

\begin{lemma}[\cite{aubrun2017alice}, Lemma 1.6]\label{lem:base}
	Let $\cK \subset \RR^d$ be a closed convex cone and $e \in \cK \cap \cK^*$ be a nonzero vector. Then
	\[
		(\cK^*)_e^b = \big\{ y \in H_e : \langle -(y-e), x-e \rangle \leq \langle e, e \rangle \text{ for all }x \in \cK_e^b \big\}.
	\]
	In other words, if we translate $H_e$ so that $e$ becomes the origin, and consider $\cK_e^b$ and $(\cK^*)_e^b$ 
	as subsets of that vector space, then $(\cK^*)_e^b = - \langle e, e \rangle (\cK_e^b)^{\circ}$.
\end{lemma}

\begin{remark}\label{rem:BHK}
	In this paper, we are concerned with cones $\cK$ such that $\bS_+^n \subseteq \cK \subseteq \bS^n$ and 
	the unit-trace subspace $H$. Note that $H = H_e$ with $e = \frac{1}{n} I_n$. We let $\base{\cK} = \cK_e^b - \frac{1}{n} I_n$ 
	denote the base of $\cK$ with respect to $e = \frac{1}{n} I_n$, translated by $- \frac{1}{n} I_n$ to contain $0$.
\end{remark}

\paragraph{Minkowski Functional and support function}
Let $S$ be a nonempty subset of $\RR^d$. The Minkowski functional (or gauge function) of $S$ is defined 
to be the function $p_S: \RR^d \to [0, \infty]$ valued in the extended real numbers such that
\begin{equation}\label{eqn:minkowski}
	p_S(x) := \inf \{ \lambda \in \RR : \lambda > 0  \text{ and } x \in \lambda S \}.
\end{equation}
We follow the convention that the infimum of the empty set is positive infinity $\infty$. 
The support function of $S$ is defined as $h_S: \RR^d \to [0, \infty]$ such that
\begin{equation}\label{eqn:support}
	h_S(x) := \sup_{z \in S} \langle x, z \rangle.
\end{equation}

There is a duality between the gauge function and the support function. 
In words, the gauge function of a convex set is the support function of its polar, and vice versa. 
\begin{lemma}[\cite{rockafellar1970convex}, Theorem 14.5]\label{lem:mink_supp}
	Let $S$ be a closed convex set containing the origin. 
	Then the polar $S^{\circ}$ is another closed convex set containing the origin, and $(S^{\circ})^{\circ} = S$. 
	Moreover, 
	\[
		p_{S}(x) = h_{S^{\circ}}(x)		\qquad\text{and}\qquad		p_{S^{\circ}}(x) = h_{S}(x).
	\]
\end{lemma}

\paragraph{Mean Width}
Given a nonempty, bounded set $S \subset \RR^d$, we define the mean width of $S$ as 
the average of $h_S( u )$ with $u$ distributed uniformly over the unit sphere in the ambient space:
\[
	w(S) := \int_{\bbS^{d-1}} h_S(u) d\sigma(u),
\]
where $\bbS^{d-1}$ is the unit sphere in $\RR^d$ and $\sigma$ is the normalized Haar measure on $\bbS^{d-1}$ 
(uniform probability measure on $\bbS^{d-1}$).
It is often convenient to consider the Gaussian variant of the mean width because its value does not depend on the ambient dimension. 

\begin{definition}[Gaussian width]\label{defn:gaussian_width}
	For any nonempty bounded set $S \subset \RR^d$, the Gaussian (mean) width 
	of $S$ is defined as
	\[
		w_G(S):= \bbE_g  h_S(g) 
			= \bbE_g   \left[ \sup_{x \in S} \inner{g}{x} \right]
			= \frac{1}{(2\pi)^{d/2}} \int_{\RR^d} \sup_{x \in S} \inner{z}{x} \exp(-\|z\|^2 / 2) dz.
	\]
	where $g$ is a standard Gaussian random vector in $\RR^d$. 
\end{definition}
\noindent
It is easy to verify that $w_G(S) = \kappa_d w(S)$ where $\kappa_d := \bbE_g \| g \|_2 
= \frac{\sqrt{2}\Gamma\left((d+1)/2\right)}{\Gamma(d/2)}$. Note that $\kappa_d$ depends only on $d$ 
and is of order $\sqrt{d}$ (it is known that $\sqrt{d - 1/2} \leq \kappa_d \leq \sqrt{d - d/(2d+1)}$).

The Gaussian width has many nice properties. Here we list a few of them that we use in later sections.
\begin{enumerate}
	\item
	The Gaussian width does not depend on the ambient dimension.
	\item
	The Gaussian width is invariant under translation and rotation.
	\item
	If $S \subseteq S'$, then $w_G(S) \leq w_G(S')$.
\end{enumerate}

\paragraph{Urysohn's Inequality}
Given a bounded measurable set $S \subset \RR^d$, its volume radius is defined as
\[
	\vrad(S) := \bigg( \frac{\vol(S)}{\vol(B_d^2)} \bigg)^{1/d}
\]
where $B_d^2$ is the unit $d$-dimensional Euclidean ball. The volume radius of $S$ is the radius of the Euclidean ball 
that has the same volume as $S$. 

A set $K \subset \RR^d$ is a convex body if it is a convex, compact set with nonempty interior. 
The following inequality, known as Urysohn's inequality, states that the mean width 
is minimized for Euclidean balls, among the sets that have the same volume.

\begin{lemma}[Urysohn's inequality; \cite{aubrun2017alice}, Propositions 4.15 \& 4.16]\label{lem:urysohn}
	Let $K \subset \RR^d$ be a convex body containing the origin in its interior. Then
	\[
		\frac{1}{w(K^{\circ})} \leq \vrad(K) \leq w(K).
	\]
\end{lemma}

\subsection{Lifts, Extension Complexity and Slack Operator}\label{sec:xc}

Here we briefly review the $\cK$-extension complexity of a convex body and its connection to the $\cK$-rank of its slack operator. 
We refer interested readers to \cite{gouveia2013lifts} and \cite{fawzi2020lifting} for more details.

Let $\cK$ be a closed convex cone. Given a positive integer $r$, let $\cK^r = \cK \times \dots \times \cK$ ($r$ times) denote 
the Cartesian product of $r$ copies of $\cK$. We say that a set $S \subset \RR^d$ admits a $\cK^r $-lift if $S$ can be expressed as
\[
	S = \pi \big( \cK^r \cap L \big)
\]
where $\pi$ is a linear map and $L$ is an affine subspace. The convex set $\cK^r \cap L$ is called a $\cK^r$-lift of $S$. 
The $\cK$-extension complexity of $S$, denoted by $\xc_{\cK}(S)$, is defined as the smallest $r$ such that $S$ admits a $\cK^r $-lift.

Let $P, Q$ be two convex bodies such that $P \subseteq Q \subseteq \RR^d$ and the origin is contained in the interior of $P$. 
Let $Q^{\circ}$ be the polar of $Q$; see \eqref{eqn:polar}. 
We let $\ext(P)$ denote the set of extreme points of $P$ and define the slack operator $s_{P,Q}$ for $(P,Q)$ as follows.

\begin{definition}[slack operator]
	For a pair of convex bodies $P \subseteq Q$ with $0$ in the interior of $P$, the map $s_{P, Q}: \ext(P) \times \ext(Q^{\circ}) \to \RR$ 
	such that $s_{P, Q}(x,y) = 1 - \langle x, y \rangle$ is called its associated slack operator. 
	The slack operator $s_{P,Q}$ admits a $\cK$-factorization if there exists a pair of maps $A: \ext(P) \to \cK$ and 
	$B: \ext(Q^{\circ}) \to \cK^*$ such that $s_{P,Q}(x,y) = \inner{A(x)}{B(y)}$ for all $x \in \ext(P)$ and $y \in \ext(Q^{\circ})$. 
\end{definition}
\noindent
Note that $s_{P, Q}(x,y) \geq 0$ for all $(x, y) \in \ext(P) \times \ext(Q^{\circ})$ because $P \subseteq Q$ and therefore 
$\langle x, y \rangle \leq 1$ for all $(x,y) \in P \times Q^{\circ}$ by definition of the polar.

The existence of a $\cK$-lift of a convex body $S$ is closely related to that of a $\cK$-factorization of $s_{P,Q}$ for 
some convex bodies $P,Q$ such that $P \subseteq S \subseteq Q$. This connection is originally established by Yannakakis \cite{yannakakis1991expressing} 
for the special case with $\cK = \RR_+$, motivated by computational considerations about linear programming (LP). 
This special case of the $\RR_+$-extension complexity is widely known as the LP extension complexity (or the extension complexity of polytopes), 
which counts the minimum number of linear inequalities required to describe $S$.  If $ \xc_{\RR_+}(S) = N$, then one can optimize a linear function 
on $S$ by solving a linear program with $N$ inequality constraints. 
Note that a polytope is generated by a finite number of extreme points, and thus its slack operator is a nonnegative matrix (so called slack matrix). 
Yannakakis' theorem states that the LP extension complexity of a polytope is equal to the nonnegative rank of its slack matrix.

The Yannakakis' theorem is later generalized in \cite{gouveia2013lifts}. We state a generalized version of Yannakakis theorem in the next 
lemma (cf. \cite[Proposition 3.12]{fawzi2020lifting}), which immediately follows from the proof of \cite[Theorem 3]{gouveia2013lifts}.

\begin{lemma}\label{lem:gen_yannakakis}
	Let $P, Q$ be a pair of convex bodies such that $P \subseteq Q$. If there is a convex body $S$ such that 
	$S$ admits a proper $\cK$-lift and $P \subseteq S \subseteq Q$, then $s_{P,Q}$ has a $\cK$-factorization. 
	Conversely, if $s_{P,Q}$ admits a $\cK$-factorization, then there exists a convex set $S$ such that 
	$S$ has a $\cK$-lift and $P \subseteq S \subseteq Q$.
\end{lemma}

In this paper, we are interested in the case where $\cK$ is a Cartesian product of small PSD cones, $\bS_+^k$ where $k \geq 1$ is a fixed constant. 
We define the $\bS_+^k$-extension complexity of $S$, denoted by $\xc_{\bS_+^k}(S)$, to be the smallest integer $r$ such that $S$ admits a $(\bS_+^k)^r$-lift. 
Given a nonnegative operator $s$, we define $\rank_{\bS_+^k}(s)$ to be the least $r$ such that $s$ admits a $(\bS_+^k)^r$-factorization.
As a consequence of Lemma \ref{lem:gen_yannakakis}, we obtain
\begin{equation}\label{eqn:yannakakis}
	\inf_{S: P \subseteq S \subseteq Q} \xc_{\bS_+^k}(S) = \rank_{\bS_+^k}(s_{P,Q}).
\end{equation}
We remark that if $\xc_{\bS_+^k}(S) = N$, then one can optimize a linear function on $S$ by solving an SDP involving $N$ variables in $\bS_+^k$.

\subsection{Fourier Analysis on the Hypercube and Hypercontractivity}
Later in the proof of Theorems \ref{thm:main.1} and \ref{thm:main.2}, we consider a certain slack operator restricted on the $n$-dimensional hypercube 
and use its degree-2 Fourier component to prove extension complexity lower bounds. Specifically, we will need to control the norm of the degree-2 
Fourier component. We review the necessary notions here and refer the interested readers to a more comprehensive reference, 
e.g., \cite{o2014analysis}.

Let $H_n = \{ -1, 1 \}^n$ denote the vertex set of the $n$-dimensional hypercube. Every function $f: H_n \to \RR$ has a unique Fourier expansion
\[
	f = f_0 + f_1 + \cdots + f_n
\]
where each $f_k$ is a homogeneous multilinear polynomial of degree $k$. We call $f_k$ the $k$-th harmonic component of $f$ and 
let $\proj_k : f \mapsto f_k$ denote the projection onto the degree-$k$ harmonic subspace (the subspace of homogeneous polynomials of degree $k$).

Given $\rho \in [0,1]$, the noise operator $\Trho$ smooths $f: H_n \to \RR$, by attenuating its high-frequency modes. To be precise, 
$\Trho$ acts on $f$ multiplying the $k$-th Fourier coefficient of $f$ by $\rho^k$, i.e.,
\[
	\Trho f = \sum_{k=0}^n \rho^k f_k.
\]
For $\rho < 1$, $\Trho f$ is `smoother' than $f$ as the high-frequency terms of $f$ are diminished. In one extreme, $\Trho f$ is constant equal to 
$\bbE f$ when $\rho = 0$; in the other extreme where $\rho = 1$, there is no smoothing and $\Trho f = f$.

Next, we recall that the $p$-norm ($p \geq 1$) of $f: H_n \to \RR$ is defined as 
\[
	\| f \|_p = \Big( \bbE_{x \sim \mu(H_n)}\big[ |f(x)|^p \big] \Big)^{\frac{1}{p}}.
\]
where $\mu(H_n)$ denotes the uniform probability measure over $H_n$. Note that $\| f \|_p \leq \| f \|_q$ for $p \leq q$. When $ p < q$, 
there is no general way to control $\| f \|_q$ with $\| f \|_p$, and the ratio $\|f\|_q / \| f \|_p$ can be arbitrarily large; the ratio becomes larger 
as $f$ fluctuates more wildly. 

The hypercontractive inequality for $\Trho$ due to Bonami and Beckner \cite{bonami1970etude, beckner1975inequalities} 
provides an upper bound on $\| \Trho f \|_q$ in terms of $\| f \|_p$ with $p < q$, thereby giving an estimate for how much 
smoother $\Trho f$ is, when compared to $f$. It can be stated as follows.

\begin{lemma}[Hypercontractivity]\label{lemma:hypercontractive}
	Given $f: H_n \to \RR$, for any $0 < \rho \leq 1$ and $p \geq 1$, we have
	$\| \Trho f \|_q \leq \| f \|_p$ where $q = 1 + \frac{1}{\rho^2} (p-1)$.
\end{lemma}

We use Lemma \ref{lemma:hypercontractive} to control the norm of the degree-2 harmonic component of a bounded nonnegative function 
as stated below in Lemma \ref{lemma:harmonic}, following \cite[Lemma 2.3]{regev2011quantum} and \cite[Lemma 3]{fawzi2018polyhedral}. 
Its proof is included in Appendix \ref{sec:proof_lemma} for completeness.
\begin{lemma}\label{lemma:harmonic}
	Let $f: H_n \to \RR$ satisfy (i) $0 \leq f(x) \leq \Lambda$ for all $x \in H_n$; and (ii) $\bbE_{x \sim \mu(H_n)} [ f(x) ] \leq 1$. 
	Then 
	\[
		\| \proj_2 f \|_2 \leq 
			\begin{cases}
				\Lambda			& \text{if }\Lambda < e,\\
				e \log( \Lambda )	& \text{if }\Lambda \geq e.
			\end{cases}
	\]
\end{lemma}


\subsection{Some Useful Facts about (Sub-)Gaussians}
Here we collect a few facts about Gaussians that are useful to control the fluctuation of Gaussian processes. 
These are standard results and more details can be found in references such as \cite{boucheron2013concentration}, 
\cite{vershynin2018high} and \cite{aubrun2017alice}.

\subsubsection{Gaussian Random Matrices and Sub-gaussian Random Variables}
\paragraph{Standard Gaussian Distribution in $\bS^n$}
Recall that the space $\bS^n$ of real symmetric $n \times n$ matrices can be viewed as real Euclidean space 
of dimension ${n+1 \choose 2}$ equipped with the trace inner product $\langle A, B \rangle = \tr (AB)$. 
We define the standard Gaussian distribution in $\bS^n$ via the natural isomorphism between $\bS^n$ and $\RR^{n+1 \choose 2}$.

\begin{definition}\label{def:std_Gauss}
	A random matrix $A \in \bS^n$ has the standard Gaussian distribution 
	if the random variables $(a_{ij})_{1 \leq i \leq j \leq n}$ are independent, with $a_{ii} \sim N(0,1)$ 
	and $a_{ij} \sim N(0,1/2)$ for $i < j$.
\end{definition}

Note that $A$ is a standard Gaussian vector in the space $\bS^n$ if and only if $\sqrt{2}A$ is a $\GOE(n)$ 
(Gaussian Orthogonal Ensemble) matrix, cf. \cite[Section 6.2]{aubrun2017alice}. 
The GOE has the property of orthogonal invariance, i.e., if $A \in \bS^n$ is a $\GOE(n)$ matrix, then 
for any fixed orthogonal matrix $U \in O(n)$, the random matrix $U A U^T$ is also a $\GOE(n)$ matrix.

\paragraph{Sub-Gaussian and Sub-exponential Random Variables}
Many interesting properties of Gaussian random variables are due to the fast decaying tail probabilities. 
Such properties are shared by some of non-Gaussian random variables, so called the class of sub-Gaussian 
random variables. This notion can be formalized based on the moment-generating function $ \bbE[e^{\lambda X}]$:
\begin{definition}\label{def:subG}
	A random variable $X$ is sub-Gaussian with parameter $v > 0$ if $\bbE[X] = 0$ and 
	\[
		\bbE[e^{\lambda X}] \leq \exp\Big(\frac{\lambda^2}{2} v\Big),	\quad\forall \lambda \in \RR.
	\]
\end{definition}
\begin{definition}\label{def:subE}
	A random variable $X$ is sub-exponential with parameters $v, c > 0$ if $\bbE[X] = 0$ and 
	\[
		\bbE[e^{\lambda X}] \leq \exp\Big(\frac{\lambda^2}{2} v\Big),	\quad\forall \lambda \text{ such that } |\lambda| \leq \frac{1}{c}.
	\]
\end{definition}
\noindent
For example, exponential and chi-squared random variables (with centering) are sub-exponential.
Informally, a sub-gaussian random variable can be viewed as a sub-exponential random variable with $c$ tending to $0$.

A sub-exponential random variable exhibits sub-Gaussian tail behavior around its center and have exponentially decaying 
tail probabilities far away from $0$. More precisely, the following tail probability bounds can be obtained by the Cram\'er-Chernoff method: 
if $X$ is a sub-exponential random variable with parameters $(v,c)$, then for every $t > 0$,
\[
	\max \big\{ \Pr[ X > t ],~ \Pr[ X < -t ] \big\} \leq 
		\begin{cases}
			e^{-t^2/2v}		&	\text{if } 0 \leq t \leq \frac{v}{c},	\\
			e^{-{t/2c}}		&	\text{if } t > \frac{v}{c}.
		\end{cases}
\]

\subsubsection{Useful Inequalities}
\paragraph{Gaussian Comparison Inequality}
The following fundamental inequality, known as Slepian's lemma, expresses that a Gaussian process can get farther 
(i.e., has a larger supremum) when it has weaker correlations. We refer the interested readers to 
\cite[Theorem 7.2.1]{vershynin2018high} and \cite[Proposition 6.6]{aubrun2017alice} for more details.

\begin{definition}
	A random process $(X_t)_{t \in T}$ is a Gaussian process if the random vector $(X_t)_{t \in T_0}$ has normal distribution 
	for all finite subsets $T_0 \subset T$. 
\end{definition}

\begin{lemma}[Slepian's lemma]\label{lem:slepian}
	Let $(X_t)_{t \in T}$ and $(Y_t)_{t \in T}$ be Gaussian processes. Suppose that for all $t, s \in T$, the following three conditions hold: 
	(i) $\bbE X_t = \bbE Y_t = 0$; 
	(ii) $\bbE X_t^2 = \bbE Y_t^2$; and
	(iii) $\bbE X_t X_s \geq \bbE Y_t Y_s$. 
	Then for every $\tau \in \RR$,
	\[
		\Pr\bigg[ \sup_{t \in T} X_t \geq \tau \bigg]
			\leq
			\Pr\bigg[ \sup_{t \in T} Y_t \geq \tau \bigg].
	\]
\end{lemma}

There is a well known upper bound for the expectation of the largest eigenvalue of a standard Gaussian random matrix in $\bS^n$. 
Its proof is based on the Slepian's lemma and standard; see Appendix \ref{sec:proof_lemma} for the proof.

\begin{lemma}\label{lem:largest_GOE}
	If a random matrix $G \in \bS^n$ has the standard Gaussian distribution, then
	\[
		\bbE_G \big[ \lambda_1(G) \big]
			= \bbE_G \Big[ \sup_{v \in \bbS^{n-1}} \langle v, Gv \rangle \Big] 
			\leq \sqrt{2n}.
	\]
\end{lemma}

\begin{remark}\label{rem:tracy_widom}
	It is known that $\lim_{n \to \infty} \bbE_G \big[ \lambda_1(G) \big] / \sqrt{2n} = 1$. 
	Indeed, not only its expected value, but also its limiting distribution is known in the literature. 
	The quantity $\lambda_1(G) - \sqrt{2n}$ is of order $n^{-1/6}$ and its distribution converges to 
	the Tracy-Widom distribution after normalization.
\end{remark}

\paragraph{Gaussian Concentration}
A smooth function of independent Gaussian random variables is sub-Gaussian. 
The following result is widely known as the Gaussian concentration inequality; see \cite[Theorem 5.5]{boucheron2013concentration} for example. 
Note that the sub-Gaussian parameter $L^2$ depends only on the smoothness of the function, and not on the number of Gaussian random variables.
\begin{lemma}[Gaussian concentration]\label{lem:gauss_conc}
	Let $X = (X_1, \dots, X_n)$ be a vector of $n$ independent standard Gaussian random variables. 
	If $f: \RR^d \to \RR$ is $L$-Lipschitz (with respect to the $\ell_2$ norm), then $f(X) - \bbE f(X)$ is sub-Gaussian with sub-Gaussian parameter $L^2$. 
\end{lemma}

The following lemma states that the support function of a convex set concentrates around its mean. 
It can be proved applying Lemma \ref{lem:gauss_conc} to the support function, which is Lipschitz 
with the Lipschitz constant being the diameter of the set. We provide a proof in Appendix \ref{sec:proof_lemma} 
for readers' convenience.
\begin{lemma}\label{lemma:gaussian_conc}
	Let $K \subset \RR^d$ be a convex set containing $0$. Let $w_G(K)$ denote the Gaussian width of $K$. Then for any $\alpha \geq 0$,
	\[
		\max\bigg\{
			\Pr_{g \sim N(0,I_d)} \Big[  \max_{x \in K} \langle g, x \rangle < (1 - \alpha) w_G(K) \Big], ~
			\Pr_{g \sim N(0,I_d)} \Big[  \max_{x \in K} \langle g, x \rangle > (1 + \alpha) w_G(K) \Big]
		\bigg\}
		 \leq \exp \bigg( - \frac{\alpha^2}{4\pi} \bigg).
	\]
\end{lemma}

\paragraph{MGF of Sub-Gaussian Chaos of Order 2}
We review the concentration of quadratic forms of the type 
\[
	\sum_{i, j=1}^n a_{ij} X_i X_j = X^T A X
\]
where $A = (a_{ij})$ is an $n \times n$ matrix of coefficients, and $X = (X_1, \dots, X_n)$ is a random vector 
with independent coordinates. Such a quadratic form is known as a chaos (of order 2) in probability theory. 

When $X_i$'s are sub-Gaussian random variables (e.g., Gaussian or Rademacher), the quadratic form $X^T A X$ 
is sub-exponential. The following upper bound is well known, and can be used to derive a Bernstein-type exponential 
concentration results (e.g., Hanson-Wright inequality) for $X^T  A X$. Its proof is based on standard techniques such as decoupling 
and comparison to Gaussian chaos. We omit the proof here and refer the interested readers to \cite[Sections 6.1 \& 6.2]{vershynin2018high} for more details. 

\begin{lemma}[MGF of sub-Gaussian chaos of order 2]\label{lemma:hanson_wright}
	Let $X = (X_1, \ldots, X_n) \in \RR^n$ be a random vector with independent sub-Gaussian coordinates with sub-Gaussian parameter $v$, 
	and let $A$ be an $n \times n$ matrix with zero diagonal. 
	Then $X^T A X$ is sub-exponential with parameters $( c_1 \|A\|_F^2 v, ~ c_2 \| A \|_{op} )$ for some absolute constants $c_1, c_2 > 0$, i.e., 
	\[
		\bbE \exp \big( \lambda X^T A X \big) \leq \exp \Big( \frac{ \lambda^2 }{2} c_1 \| A \|_F^2 v \Big),
			\qquad\text{for all }\lambda \text{ s.t. }|\lambda| \leq \frac{1}{c_2 \| A \|_{op}}.
	\]
\end{lemma}

Observe that for any function $f: H_n \to \RR$, its degree-2 projection, $\proj_2 (f)$, is a multilinear quadratic form on $H_n$. 
That is, there exists some matrix $A$ with zero diagonal\footnote{More precisely, $A_{ij} = \frac{1}{2} \bbE_{X \sim \mu(H_n)} [ X_i X_j f(X) ]$ 
for $i, j \in [n]$ such that $i \neq j$. } such that $\proj_2(f)(x) = x^T A x$ for all $x \in H_n$. 
Therefore, the random variable $\proj_2(f)(X)$ derived from the uniform random vector $X \sim \mu(H_n)$ is sub-exponential by 
Lemma \ref{lemma:hanson_wright}. We formally state this observation in the following lemma to use later in the proof of Theorem \ref{thm:main.1}; 
see Appendix \ref{sec:proof_lemma} for its proof.

\begin{lemma}\label{lemma:subexp_tail}
	Let $X$ be a random vector uniformly distributed over $H_n$. For any function $f: H_n \to \RR$, the derived random variable 
	$\proj_2(f)(X)$ is sub-exponential with parameters $(c_1 M_f^2, ~ c_2 M_f )$ where $M_f = \| \proj_2 f \|_2 / \sqrt{2}$, and 
	$c_1, c_2>0$ are the same absolute constants that appear in Lemma \ref{lemma:hanson_wright}. That is,
	\[
		\bbE_{X \sim \mu(H_n)} \exp \big( \lambda~ \proj_2(f)(X) \big)
			\leq  \exp \bigg( \frac{\lambda^2}{2} c_1 M_f^2 \bigg),
			\qquad\text{for all }\lambda \text{ s.t. }|\lambda| \leq \frac{1}{c_2 M_f }.
	\]
\end{lemma}

\paragraph{Maximal Inequalities}
The following simple maximal inequality is well known, and it is asymptotically sharp if the random variables are i.i.d. Gaussian. 
Its proof can be found in Appendix \ref{sec:proof_lemma}.
\begin{lemma}\label{lem:maximal}
	Let $X_1, \dots, X_N$ be sub-exponential random variables with parameters $(v, c)$. Then
	\[
		\bbE\Big[ \max_{i \in [N] } X_i \Big] \leq \max \big\{ \sqrt{2 v \log N}, ~ 2c \log N \big\}.
	\]
\end{lemma}

\section{Three Notions of Approximation}\label{sec:notion_approx}
Recall that we want to approximate the positive semidefinite cone $\bS_+^n$ with a convex cone $\cK \supseteq \bS_+^n$ 
so that the feasible set $\base{\cK} = \big( \cK \cap H \big) - \frac{1}{n}I_n$ (cf. Remark \ref{rem:BHK}) 
well approximates $\base{\bS_+^n}$. In Section \ref{sec:defn_sets}, we introduce three notions of approximation for sets. 
In Section \ref{sec:defn_cones}, we extend these notions to cones to assess the quality of $\cK$ as an approximation of $\bS_+^n$.

Specifically, we first define a natural notion of $\epsilon$-approximation for sets that contain the origin (Definition \ref{defn:approx}). 
Then, we additionally describe two auxiliary notions of approximation for the convenience of our analysis, namely, the average $\epsilon$-approximation 
(Definition \ref{defn:average_approx}) and the dual-average $\epsilon$-approximation (Definition \ref{defn:average_approx_dual}). 
These two auxiliary notions can be obtained by relaxing a quantifier in the definition of $\epsilon$-approximation. These relaxed notions are 
closely related, but incomparable to each other. They will be respectively used in Section \ref{sec:results.1} and Section \ref{sec:result.extended} 
to prove the hardness of approximating $\bS_+^n$ with a small number of $k \times k$ PSD constraints.

\subsection{Notions of Approximation for Sets}\label{sec:defn_sets}
To begin with, we define the notion of $\epsilon$-approximation for sets containing the origin.
\begin{definition}[$\epsilon$-approximation]\label{defn:approx}
	Let $P$ be a set containing $0$. For $\epsilon > 0$, a set $S$ is an \emph{$\epsilon$-approximation} of $P$ 
	if $P \subseteq S \subseteq ( 1 + \epsilon) P$. 
	Given two sets $P, S$ that contain $0$, we let
	\[
		\epsbasic(P, S) := \inf\{ \epsilon > 0: S \text{ is an }\epsilon \text{-approximation of }P  \}.
	\]
\end{definition}
\noindent
This is a natural notion to quantify how tightly a set $P$ containing $0$ can be approximated by another set $S \supseteq P$. 
Recall the definition of the support function $h_S(x) := \sup_{z \in S} \langle x, z \rangle$, cf. \eqref{eqn:support}. 
We observe that if $S$ is an $\epsilon$-approximation of $P$, then 
\begin{equation}\label{eq:implication}
	h_P(x) \leq h_S(x) \leq (1+\epsilon) h_P(x) \quad\text{for all } x.
\end{equation}
That is, if $S$ is an $\epsilon$-approximation of $P$, then for every direction in the ambient space, 
the distance from the supporting hyperplane of $S$ to the origin is at most $(1+\epsilon)$ times 
the distance from the supporting hyperplane of $P$ to the origin.
Moreover, when $P$ and $S$ are convex, the converse is also true.

Next, we define a more lenient notion of approximation by relaxing the quantifier `for all $x$' in \eqref{eq:implication} 
by taking average over random direction $x$. To this end, recall the notion of Gaussian width from Definition \ref{defn:gaussian_width} 
that $w_G(S) = \bbE_G   \left[ h_S(G) \right]$ for any nonempty bounded set $S \subset \RR^d$, where $G$ is a standard Gaussian. 
\begin{definition}[average $\epsilon$-approximation]\label{defn:average_approx}
	Let $P$ be a set containing $0$. For $\epsilon > 0$, a set $S$ is an \emph{average $\epsilon$-approximation} of $P$, 
	or \emph{$\epsilon$-approximation of $P$ in the average sense}, if $P \subseteq S$ and $w_G(S) \leq (1+\epsilon) w_G(P)$. 
	Given two sets $P, S$ that contain $0$, we let
	\[
		\epsavg(P, S) := \inf\{ \epsilon > 0: S \text{ is an average }\epsilon \text{-approximation of }P  \}.
	\]
\end{definition}
\noindent
By definition, $S$ is an average $\epsilon$-approximation of $P$ if and only if $\bbE_{G} [ h_S(G) - h_P(G) ] 
\leq \epsilon \cdot \bbE_{G}\big[ h_P(G) \big]$ where $G$ is a standard Gaussian random matrix in $\bS^n$.

Note that average $\epsilon$-approximation is a weaker notion than $\epsilon$-approximation because 
$\epsbasic(P, S) \geq \epsavg(P, S)$. That is, for a fixed $\epsilon > 0$, if $S$ is an $\epsilon$-approximation of $P$, 
then $S$ is also an average $\epsilon$-approximation of $P$. As a matter of fact, average $\epsilon$-approximation is 
a strictly weaker notion because there exists a pair of sets $(P,S)$ such that $\epsbasic(P, S) > \epsavg(P, S)$, i.e., 
there exists some $\epsilon > 0$ for which $S$ is not an $\epsilon$-approximation of $P$ whereas 
$S$ is an average $\epsilon$-approximation of $P$. We illustrate this point with the following two examples.

\begin{example}\label{exmp:1}
Let $P = \{(x,y) \in \RR^2: x^2 + y^2 \leq 1 \}$ and $S = \{ (x,y) \in \RR^2: x^2/4 + y^2 \leq 1 \}$. 
Then $\epsbasic(P, S) = 1$. On the other hand, $\epsavg(P, S) = \frac{4}{\pi} E(3/4) - 1 \approx 0.54196$ 
where $E(m)$ is the complete elliptic integral of the second kind with parameter $m = k^2$. 
The value of $\epsavg(P, S)$ can be computed by observing that $w_G(P) = \bbE_{G \in N(0, I_2)} \| g \|_2$ and 
$w_G(S) = \bbE_{g \in N(0, I_2)} \| g \|_2 \cdot \frac{1}{2\pi} \int_0^{2\pi} \sqrt{ 4 \cos^2 \theta + \sin^2 \theta } d\theta 
= \frac{4}{\pi} E(3/4) \bbE_{g \in N(0, I_2)} \| g \|_2$.
\end{example}

\begin{example}\label{exmp:2}
	Let $P = \{ x \in \RR^n: \| x \|_{2} \leq 1 \}\}$ and $S = \{ x \in \RR^n: \| x \|_{1} \leq \sqrt{n} \}\}$ where $\| \cdot \|_{p}$ 
	denotes the $\ell_p$-norm. Then $\epsbasic(P, S) = \sqrt{n} - 1$. On the other hand, $\epsavg(P, S)= f(n)$ where $f(n)$ 
	is a function of $n$ such that $f(n) \approx \sqrt{2 \log n}$ for sufficiently large $n$. 
	It is because $w_G(P) = \bbE_{g \sim N(0, I_n)} \| g \|_2 \approx \sqrt{n}$ and $w_G(S) 
	= \sqrt{n} \cdot \bbE_{g \sim N(0, I_n)} \max_{i \in [n]}  | g_i |\approx \sqrt{2n \log n}$.
\end{example}

In the two examples above, we observed there exists $\epsilon > 0$ such that $S$ is an $\epsilon$-approximation of $P$ 
in the average sense, while it is not an $\epsilon$-approximation. This happens because $h_S(G) - h_P(G)$ is small on average, 
but the difference can be potentially large for some $G$. In other words, $S$ approximates $P$ well on average, 
but poorly for certain `bad' directions in the ambient space, as illustrated in Figure \ref{fig:approx}. 
Nevertheless, the set of `bad' directions might have only a small measure as in Example \ref{exmp:2}, and the notion of 
$\epsilon$-approximation as in Definition \ref{defn:approx} can be overly conservative. That is why we additionally consider 
the notion of average $\epsilon$-approximation, which is more lenient with the shape of the approximating set $S$. 


\begin{figure}	
	\centering
	\begin{subfigure}[b]{0.33\textwidth}
		\centering
		\begin{tikzpicture}
		\draw [pattern=north east lines, pattern color=black!15] (0,0) circle (1.5cm);
		\draw [color=blue]
			(2.5,0) -- (1.25, 2.15) -- (-1.25, 2.15) -- (-2.5, 0) -- (-1.25, -2.15) -- (1.25, -2.15) -- cycle;
		\draw [dashed] (0,0) circle (2.5cm);
		
		\node (A1) at (0.5, -1) {$P$};
		\node (A2) at (1.25, -1.5) {\color{blue}$S$};
		\node (A3) at (2.4, -2) {$(1+\epsilon) P$};
		
		\draw[<->] (0, 0) -- (0.76,1.285);
		\draw[<->] (0.76,1.29) -- (1.25, 2.15) ;
		\node (A01) at (0.1, 0.8) {\footnotesize$1$};
		\node (A02) at (0.7, 1.7) {\footnotesize$\epsilon$};
	
		\end{tikzpicture}
		\caption{$S$ is an $\epsilon$-approximation of $P$}
	\end{subfigure}
	\qquad\qquad
	\begin{subfigure}[b]{0.5\textwidth}
		\centering
		\begin{tikzpicture}
		\draw [pattern=north east lines, pattern color=black!15] (0,0) circle (1.5cm);
		\draw [blue] plot [smooth cycle]
			coordinates {(5,1.5) (1.7, 1.1) (1, 1.8) (0.5, 1.7) (-0.2, 2.1) (-0.9, 1.4) (-1.7, 1.2) (-1.8, 0.3) (-2.3, -0.2) (-1.7, -0.5) (-1.8, -1.1) (-1.5, -1.3) (-1.3, -1.8) (-0.4, -1.7) (0.1, -2.2) (0.6, -1.7) (1.4, -1.3) (2.4, -1.5) (2.2, -0.5)};
	
		\node (B1) at (1.1, -0.25) {$P$};
		\node (B2) at (1.9, -0.35) {\color{blue}$S$};

		\draw[<->] (0, 0) -- (1.44, 0.432);
		\draw[<->] (1.44, 0.432) -- (5, 1.5) ;
	
		\end{tikzpicture}
		\caption{$S$ `poorly' approximates $P$ only in some directions}
	\end{subfigure}
	\caption{Cartoons illustrating the difference between $\epsilon$-approximation and average $\epsilon$-approximation.}
	\label{fig:approx}
\end{figure}
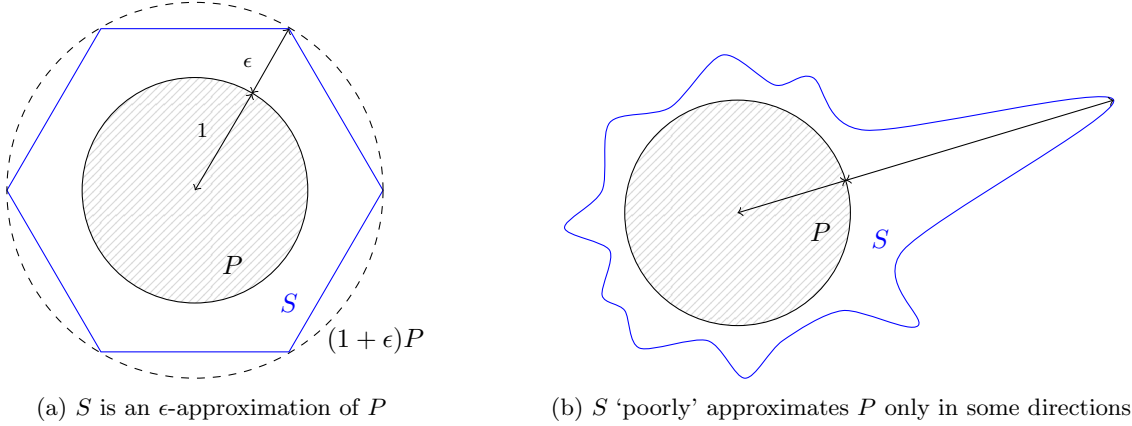

One drawback of evaluating the quality of approximation with the notion of average $\epsilon$-approximation is that it only 
measures the difference averaged over an ensemble of random objectives. Thus, we cannot control the gap $h_S(x) - h_P(x)$  
for any specific $x$, however, we can still establish a probabilistic upper bound on $h_S(G) - h_P(G)$ when $G$ is randomly drawn 
from the standard Gaussian distribution. 

\begin{lemma}\label{lem:prob_control}
	Let $S$ be an average $\epsilon$-approximation of $P$ for some $\epsilon > 0$. 
	Then 
	for all $ \tau > 0$,
	\[
		\Pr_{G \sim \text{std Gaussian}} \Big[ h_S(G) - h_P(G) > \tau \Big] 
			\leq	\epsilon	\frac{w_G(P)}{ \tau }.
	\]
\end{lemma}
\noindent
Lemma \ref{lem:prob_control} operationally means that if $S$ is an average $\epsilon$-approximation of $P$ for small $\epsilon$, 
then $h_S(x) - h_P(x)$ can be large only for $x$ in a set that has small measure. 
In particular, the probability upper bound converges to $0$ as $\epsilon \to 0$. That is, $h_S(x) - h_P(x)$ converges 
to $0$ for all $x$ (but those in a set of measure-zero) as $\epsilon \to 0$.

\begin{proof}[Proof of Lemma \ref{lem:prob_control}]
	Note that $h_S(G) - h_P(G) \geq 0$ for all $g$ because $P \subseteq S$. 
	The conclusion follows from Markov's inequality and the observation that 
	$w_G(S) - w_G(P) \leq \epsilon \cdot w_G(P)$.
\end{proof}

Lastly, we revisit Definition \ref{defn:approx} to introduce an alternative relaxation of $\epsilon$-approximation, 
namely, the `dual' version of average $\epsilon$-approximation. Recall from \eqref{eqn:minkowski} that 
the gauge function of $S$ is defined as $p_S(x) := \inf \{ \lambda \in \RR : \lambda > 0  \text{ and } x \in \lambda S \}$. 
Observe that $P \subseteq S \subseteq ( 1 + \epsilon) P$ if and only if $\frac{1}{1+\epsilon}p_P(x) \leq p_S(x) \leq p_P(x)$ 
for all $x$.
When $P$ and $S$ are closed convex sets, $p_P(x) = h_{P^{\circ}}(x)$ and $p_S(x) = h_{S^{\circ}}(x)$ by Lemma \ref{lem:mink_supp}. 
Therefore, $S$ is an $\epsilon$-approximation of $P$ if and only if $\frac{1}{1+\epsilon}h_{P^{\circ}}(x) \leq p_{S^{\circ}}(x) 
\leq p_{P^{\circ}}(x)$ for all $x$.
As before, we ease the condition ``$\frac{1}{1+\epsilon}h_{P^{\circ}}(x) \leq p_{S^{\circ}}(x)$ for all $x$'' by averaging over $x$ 
to reach at the following definition.

\begin{definition}[dual-average $\epsilon$-approximation]\label{defn:average_approx_dual}
	Let $P$ be a set containing $0$. For $\epsilon > 0$, a set $S$ is a \emph{dual-average $\epsilon$-approximation} 
	of $P$, or \emph{$\epsilon$-approximation of $P$ in the dual-average sense}, 
	if $P \subseteq S$ and $w_G(S^{\circ}) \geq \frac{1}{1+\epsilon} w_G(P^{\circ})$. 
	Given two sets $P, S$ that contain $0$, we define
	\[
		\epsdavg(P, S) := \inf\{ \epsilon > 0: S \text{ is a dual-average }\epsilon \text{-approximation of }P \}.
	\]
\end{definition}

Note that dual-average $\epsilon$-approximation is also a weaker notion than $\epsilon$-approximation. That is, for a fixed $\epsilon > 0$, 
if $S$ is an $\epsilon$-approximation of $P$, then $S$ is also a dual-average $\epsilon$-approximation of $P$. In Section \ref{sec:results.1}, 
we use the notion of dual-average $\epsilon$-approximation as a technical tool to prove the hardness of $k$-PSD approximations of $\bS_+^n$.

The notion of dual-average $\epsilon$-approximation is closely related to the notion of average $\epsilon$-approximation; they are dual to 
each other. However, they are not equivalent notions of approximation, i.e., there exist convex sets $P, S$ such that $S$ is a good approximation 
of $P$ in the average sense, but not in the dual average sense. The opposite is also possible. See the next remark and Example \ref{exmp:avg_davg}.

\begin{remark}
	For $\epsilon > 0$, $S$ is an average $\epsilon$-approximation of $P$ if and only if $P^{\circ}$ is a dual-average $\epsilon$-approximation 
	of $S^{\circ}$. In other words, $\epsavg(P, S) = \epsdavg(S^{\circ}, P^{\circ})$. In this sense, the notion of dual-average $\epsilon$-approximation 
	is the dual of the notion of average $\epsilon$-approximation.
\end{remark}


\begin{example}[Ball, Needle, and Pancake]\label{exmp:avg_davg}
	Consider a $d$-dimensional unit $\ell_2$-ball and a `needle' obtained by taking the convex hull of the union of 
	the ball and two points that are located on the opposite side of the origin at distance $d$. The polar of this `needle' 
	is the `pancake' obtained by intersecting the unit ball with a slab of thickness $2/d$ along its equator. 
	These three sets are illustrated in Figure \ref{fig:avg_davg}. 
	We observe that the Gaussian width of the ball, the needle, and the pancake are approximately $\sqrt{d - 1/2}, d\sqrt{2/\pi}$, 
	and $\sqrt{d-3/2}$, respectively. Thus, the ball is a good approximation of the pancake in the average sense, but not in the dual-average sense. 
	Likewise, the needle is a good approximation of the ball in the dual-average sense, but not in the average sense.
\end{example}


\begin{figure}
	\centering
	\begin{tikzpicture}
	\draw [pattern=north west lines, pattern color=black!50] (-4,2) circle (1.5cm);
	\path [pattern=north east lines, pattern color=blue!10]
		(-4.75, 3.3) -- (-7,2) -- (-4.75, 0.7) arc[start angle=240, end angle=300, radius=1.5] -- (-3.25, 0.7) -- (-1,2) -- (-3.25, 3.3) arc [start angle=60, end angle=120, radius=1.5] -- cycle;
	\draw[blue] (-4.75, 3.3) -- (-7,2) -- (-4.75, 0.7);
	\draw[blue] (-3.25, 3.3) -- (-1,2) -- (-3.25, 0.7);
	
	\fill (-7,2) circle (1.5pt);
	\fill (-1,2) circle (1.5pt);
	
	\draw [pattern=north west lines, pattern color=black!10] (4,2) circle (1.5cm);
	\draw[blue] (3.5, 0.25) -- (3.5, 3.75);
	\draw[blue] (4.5, 0.25) -- (4.5, 3.75);
	\path [pattern=north east lines, pattern color=blue!50]
		(4.5, 0.6) -- (4.5, 3.4) arc[start angle=70, end angle=110, radius=1.5] -- (3.5, 3.4) -- (3.5, 0.6) arc [start angle=250, end angle=290, radius=1.5] -- cycle;
	
	\draw (-7,-0.15) -- (-7, 0.15);
	\draw (-4,-0.15) -- (-4, 0.15);
	\draw (-1,-0.15) -- (-1, 0.15);
	\draw[<->] (-7,0) -- (-4,0);
	\draw[<->] (-4,0) -- (-1,0) ;
	\node (A1) at (-5.5, -0.4) {$d$};
	\node (A2) at (-2.5, -0.4) {$d$};
	\draw (3.5,-0.15) -- (3.5, 0.15);
	\draw (4,-0.15) -- (4, 0.15);
	\draw (4.5,-0.15) -- (4.5, 0.15);
	\draw[<->] (3.5,0) -- (4,0);
	\draw[<->] (4,0) -- (4.5,0) ;
	\node (B1) at (3.65, -0.4) {$1/d$};
	\node (B2) at (4.35, -0.4) {$1/d$};

	\node (A_ball) at (-4, 4) {Ball};
	\node (A_needle) at (-1.5, 3) {Needle};
	\node (B_ball) at (6, 3) {Ball};
	\node (B_pancake) at (4, 4) {Pancake};	
	
	\end{tikzpicture}
	\caption{The sets described in Example \ref{exmp:avg_davg}. }
	\label{fig:avg_davg}
\end{figure}
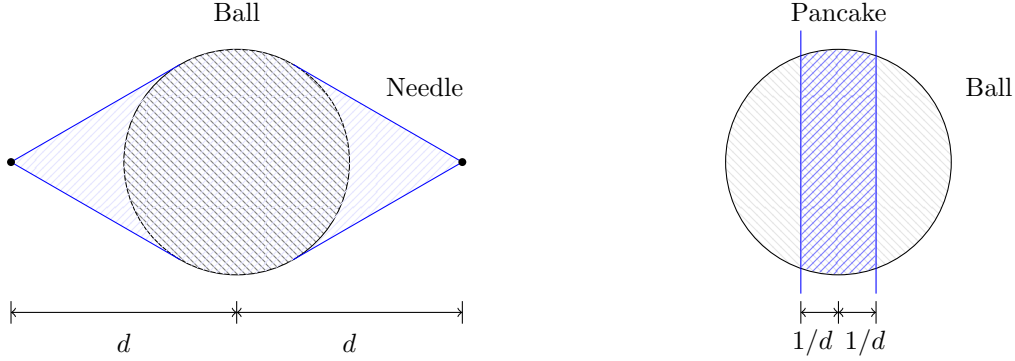

\subsection{Notions of Approximation for Cones}\label{sec:defn_cones}
Recall that our primary motivation for introducing the notions of approximation is 
to quantify the optimality gap that arises from a conic programming relaxation of the problem in \eqref{eqn:sdp}. 
Suppose that we are to relax 
the problem \eqref{eqn:sdp} by replacing the PSD cone $\bS_+^n$ with a larger cone $\cK \supseteq \bS_+^n$. 
Letting $P = \{ X \in \bS_+^n: ~ \langle A_i, X \rangle = b_i, ~ i = 1, \ldots, m \}$ and $S = \{ X \in \cK: ~ \langle A_i, X \rangle = b_i, ~ i = 1, \ldots, m \}$ 
denote the feasible sets of the original and the relaxed problems, we can see that $S \supseteq P$ and there arises 
an increase  in the optimal value, $\Gamma_{P, S}(C) := h_S(C) - h_P(C)$, as a result of the relaxation.

We extend the notions of approximation for sets, defined in Section \ref{sec:defn_sets}, to the notions for cones 
by fixing a certain affine constraint. Recall that for a cone $\cK$, we let $\base{\cK} := \big( \cK \cap H \big) - \frac{1}{n} I_n 
= \big\{ X - \frac{1}{n} I_n \in \bS^n: X \in \cK \cap H \big\}$ where $H = \{ X \in \bS^n: \tr~ X = 1 \}$ and $I_n$ denotes 
the $n \times n$ identity matrix. Note that $\base{\cK}$ is the feasible set of the problem \eqref{eqn:sdp}, translated 
by $-\frac{1}{n}I_n$, when the affine constraint in \eqref{eqn:sdp} is the unit trace constraint.
We define the notions of approximation for cones as follows.

\begin{definition}[$\epsilon$-approximation for cones in $\bS^n$]\label{defn:approx_cones}
	A cone $\cK \subseteq \bS^n$ is an $\epsilon$-approximation (average $\epsilon$-approximation / dual-average 
	$\epsilon$-approximation, resp.) of $\bS_+^n$ if $\base{\cK}$ is an $\epsilon$-approximation (average 
	$\epsilon$-approximation / dual-average $\epsilon$-approximation, resp.) of $\base{\bS_+^n}$. 
	Also, we let 
	\[
		\epsbasic( \bS_+^n, \cK ) := \epsbasic( \base{\bS_+^n}, \base{\cK} )
	\]
	 and define 
	$\epsavg( \bS_+^n, \cK )$ and $\epsdavg( \bS_+^n, \cK )$ in a similar manner.
\end{definition}


\begin{remark}\label{rem:width_Sn}
For later use, we remark here that $w_G\big( \base{\bS_+^n}\big) \leq \sqrt{2n}$ and that $\lim_{n \to \infty} 
\frac{w_G( \base{\bS_+^n})}{\sqrt{2n}} = 1$ because $\bS_+^n \cap H = \conv\{ vv^T: v \in \bbS^{n-1} \}$, 
cf. Lemma \ref{lem:largest_GOE} and Remark \ref{rem:tracy_widom}. 
\end{remark}

\section{$k$-PSD Approximations of $\bS_+^n$}\label{sec:results.1}
One option to relax the PSD constraint $X \in \bS_+^n$ in \eqref{eqn:sdp} is to enforce the PSD constraints 
only on the smaller $k \times k$ principal submatrices of $X$, which leads to the following relaxation:
\begin{equation}\label{eqn:sparse_sdp}
	\begin{aligned}
		\text{maximize}	\quad&\langle C, X\rangle\\
		\text{subject to}\quad&\langle A_i, X \rangle = b_i, \quad i = 1, \ldots, m,\\
						&k \times k \text{ principal submatrices of } X \in \bS_+^k.
	\end{aligned}
\end{equation}
Note that the PSD cone $\bS_+^n$ in \eqref{eqn:sdp} is replaced with a relaxed cone that is defined using ($k \times k$)-sized 
PSD constraints, and \eqref{eqn:sparse_sdp} can be solved more efficiently when $k \ll n$. 
For example, $k=1$ yields a linear programming (LP) approximation and $k=2$ produces a second-order cone programming (SOCP) 
approximation of the original SDP \cite{ahmadi2019dsos}.

In this section, we consider a scheme to approximate $\bS_+^n$ by enforcing $k \times k$ PSD constraints on particular subspaces. 
To be precise, we choose a fixed set of $k$-dimensional subspaces in $\RR^n$ and define a cone of $n \times n$ symmetric matrices 
that are PSD when restricted to these subspaces. The cone associated with \eqref{eqn:sparse_sdp} is an example of this construction 
that is obtained by imposing PSD constraints on the ${n \choose k}$ subspaces of $k$-sparse vectors in $\RR^n$, and will be referred 
to as the sparse $k$-PSD approximation of $\bS_+^n$.

In Section \ref{sec:restr_kPSD}, we formalize the definition of the $k$-PSD approximation and prove a lower bound on the number 
of $k \times k$ PSD constraints required. We show that when $k$ is much smaller than $n$, it is necessary to impose 
PSD constraints on at least exponentially many subspaces to produce a cone that approximates $\bS_+^n$ well.
In Section \ref{sec:coord_kPSD}, we discuss the sparse $k$-PSD approximation in more detail. 


\subsection{Lower Bound for $k$-PSD Approximations of $\bS_+^n$}\label{sec:restr_kPSD}
We recall the definition of the $k$-PSD approximation of $\bS_+^n$ from Definition \ref{def:restr_kPSD}.
\begin{definition}[$k$-PSD approximation of $\bS_+^n$; restatement of Definition \ref{def:restr_kPSD}]\label{def:gen_kPSD}
	Let $\cV = \{ V_1, \dots, V_N \}$ be a set of $k$-dimensional subspaces of $\RR^n$. 
	The $k$-PSD approximation of $\bS_+^n$ induced by $\cV$ is the convex cone
	\[
		\Snk( \cV ) := \big\{ X \in \bS^n: v^T X v \geq 0, ~\forall v \in V_i, ~\forall i = 1, \dots, N\}.
	\]
\end{definition}
\noindent
Note that $\Snk( \cV ) \supseteq \bS_+^n$ is the set of $n \times n$ symmetric matrices whose associated quadratic forms are 
positive semidefinite when restricted to $V_1 \cup \dots \cup V_N$. Thus, if $U_i \in \RR^{n \times k}$ is 
a matrix whose columns form a basis of $V_i$, 
then $\Snk( \cV ) = \big\{ X \in \bS^n: U_i^T X U_i \in \bS_+^k, ~\forall i = 1, \dots, N\}$.

Our first main theorem presents an upper bound on the Gaussian width of the base of the dual cone of 
$\Snk( \cV )$ as a function of $k$ and $N = | \cV |$.
 
\begin{theorem}\label{thm:main.0}
	Let $n$, $1 \leq k \leq n$ be positive integers and $\cV = \{ V_1, \dots, V_N \}$ be any set of $k$-dimensional subspaces of $\RR^n$. 
	Then 
	\[
		w_G\Big( \based{ \Snk( \cV ) } \Big)	
			\leq 	\sqrt{2k} + \sqrt{2 \log N }.
	\]
\end{theorem}

Recall that $w_G\big( \based{\bS_+^n}\big) = w_G\big( \base{\bS_+^n}\big) \approx \sqrt{2n}$, cf. Remark \ref{rem:width_Sn}. Comparing 
the upper bound in Theorem \ref{thm:main.0} against $\sqrt{2n}$, we can contrast the size of $\based{ \Snk( \cV ) }$ relative to $\based{\bS_+^n}$. 
For example, when $k$ and $N$ are small, $\sqrt{2k} + \sqrt{2 \log N} \ll \sqrt{2n}$, and we can intuitively see that the dual of the cone $\Snk( \cV )$ is 
much smaller than the original PSD cone $\bS_+^n$. Therefore, the primal cone $\Snk( \cV )$ is too big to well approximate $\bS_+^n$ in such a case.

\begin{remark}\label{rem:union}
	Note that the upper bound in Theorem \ref{thm:main.0} holds regardless of the subspaces $V_1, \dots, V_N$ in $\cV$, 
	i.e., it is oblivious to the configuration of the subspaces. That is, this upper bound is valid even for the ``best'' possible configuration 
	of subspaces to imitate the expressive power of the full-sized PSD cone. We also note that this upper bound could conceivably be too conservative, 
	especially when $N$ is large, because it implicitly hinges on the union bound (through the use of Lemma \ref{lem:maximal}).
\end{remark}

\begin{proof}[Proof of Theorem \ref{thm:main.0}]
First of all, due to the translation invariance of the Gaussian width, we have 
\[
	w_G\Big( \based{ \Snk( \cV ) } \Big)
		= w_G\bigg( \Snk( \cV )^* \cap H - \frac{1}{n} I_n \bigg)
		= w_G\Big( \Snk( \cV )^* \cap H \Big).
\]

Next, we let $U_i \in \RR^{n \times k}$ be a matrix whose columns form an orthonormal basis of $V_i$ for each $V_i \in \cV$. 
We observe that $\Snk( \cV )^* = \cl \cone \big( \bigcup_{i \in [N]} \{ U_i Z U_i^T: Z \in \bS_+^k \} \big)$ 
because $(\bS_+^k)^* = \bS_+^k$ and $(C_1 \cap C_2)^* = \cl \cone( C_1^* \cup C_2^*)$, cf. Section \ref{sec:dual}. 
Thus, it follows that $\Snk( \cV )^* \cap H = \conv \Big( \bigcup_{i \in [N]} \{ U_i vv^T U_i^T: v \in \bbS^{k-1} \} \Big)$.
Note that $\langle G, U_i vv^T U_i^T \rangle = \langle U_i^T G U_i, v v^T \rangle$, and therefore,
\begin{align*}
	w_G\Big( \Snk( \cV )^* \cap H\Big)
		&= \bbE_{G} \bigg[ \sup_{i \in [N] \atop v \in \bbS^{n-1}} \big\langle U_i^T G U_i, vv^T \big\rangle \bigg]
			= \bbE_{G} \bigg[ \sup_{i \in [N]}  \lambda_1 ( U_i^T G U_i ) \bigg]\\
		&\leq \sup_{i \in [N]} \bbE_G \big[ \lambda_1 ( U_i^T G U_i ) \big]
			+ \bbE_{G} \bigg[ \sup_{i \in [N]} \Big(  \lambda_1 ( U_i^T G U_i ) - \bbE_G \big[ \lambda_1 ( U_i^T G U_i ) \big] \Big) \bigg].
\end{align*}

Note that for every $i \in [N]$, the random matrix $U_i^T G U_i \in \bS^k$ has the standard Gaussian distribution in $\bS^k$. 
By Lemma \ref{lem:largest_GOE}, $\bbE_{G} \big[ \lambda_1 ( U_i^T G U_i ) \big] \leq \sqrt{2k}$. 
Moreover, the function $G \mapsto \lambda_1( U_i^T G U_i)$ is $1$-Lipschitz, and therefore, the random variable 
$\lambda_1 ( U_i^T G U_i ) - \bbE_G \big[ \lambda_1 ( U_i^T G U_i )\big]$ is sub-Gaussian with sub-Gaussian parameter $1$ 
by Lemma \ref{lem:gauss_conc}. Lemma \ref{lem:maximal} implies that
$\bbE_{G} \Big[ \sup_{i \in [N]} \big(  \lambda_1 ( U_i^T G U_i ) - \bbE_G \big[ \lambda_1 ( U_i^T G U_i ) \big] \big) \Big] \leq \sqrt{2 \log N}$.

\end{proof}

Now we discuss how Theorem \ref{thm:main.0} implies the hardness of approximating $\bS_+^n$ with a small number of 
$k \times k$ PSD constraints. In the next corollary, we show that if $N = |\cV|$ is below a certain threshold determined by $n, k, \epsilon$, 
then $\Snk(\cV)$ cannot be a dual-average $\epsilon$-approximation of $\bS_+^n$. Thus, it cannot be an $\epsilon$-approximation of $\bS_+^n$, either. 

\begin{corollary}\label{coro:main.0}
	Let $n, k$ be positive integers such that $1 \leq k \leq n$, and $\epsilon > 0$. 
	If $ \Snk ( \cV )$ is a dual-average $\epsilon$-approximation of $\bS_+^{n}$, then $|\cV| \geq \exp \big( n \cdot \varphi(n, k, \epsilon)  \big)$ where
	\begin{equation*}
		\varphi(n, k, \epsilon) = \bigg[~ \frac{1}{1+\epsilon} \frac{w_G \big( \base{\bS_+^{n}} \big)}{\sqrt{2 n}} - \sqrt{\frac{k}{n}} ~ \bigg]_+ ^2.
	\end{equation*}
\end{corollary}

\begin{proof}[Proof of Corollary \ref{coro:main.0}]
	Suppose that $ \Snk ( \cV )$ is a dual-average $\epsilon$-approximation of $\bS_+^{n}$. Then 
	by definition of the dual-average approximation (see Definitions \ref{defn:average_approx_dual} and \ref{defn:approx_cones}),
	\begin{equation}\label{eqn:assumption}
		w_G \Big( \base{ \Snk ( \cV ) }^{\circ} \Big) 
			\geq \frac{1}{1+\epsilon} w_G \Big( \base{ \bS_+^{n} }^{\circ} \Big).
	\end{equation}
	By Lemma \ref{lem:base}, we have $\base{ \Snk( \cV) }^{\circ} = -n \based{ \Snk( \cV )} $ 
	and $\base{ \bS_+^{n} \cap H }^{\circ} = -n \based{ \bS_+^{n} } = -n \base{ \bS_+^{n} }$ because $\bS_+^n$ is self-dual. 
	Thus, Theorem \ref{thm:main.0}, combined with the inequality \eqref{eqn:assumption}, implies
	\[
		 \frac{1}{1+\epsilon} w_G \Big( \base{ \bS_+^{n} } \Big)
		 	\leq w_G\Big( \base{ \Snk ( \cV )^* } \Big)	
			\leq 	\sqrt{2k} + \sqrt{2 \log |\cV|}.
	\]
	Note that this inequality holds if and only if 
	\[
		\sqrt{ \log |\cV| } \geq  \frac{1}{\sqrt{2}(1+\epsilon)} w_G \Big( \base{ \bS_+^{n} } \Big) - \sqrt{k},
	\]
	which is again equivalent to 
	\[
		|\cV| \geq \exp \bigg[ ~\frac{w_G \big( \base{\bS_+^{n} } \big) }{ \sqrt{2} (1+\epsilon)} - \sqrt{k} ~\bigg]_+ ^2
			= \exp \big( n \cdot \varphi(n,k,\epsilon) \big).
	\]
\end{proof}

\begin{remark}\label{rem:genkPSD_asymp}
	Recall from Remark \ref{rem:tracy_widom} that $\lim_{n \to \infty} w_G \big( \base{\bS_+^{n}} \big) / \sqrt{2n} = 1$. 
	With $k = \lfloor \delta n \rfloor$ for $0 < \delta < 1$, 
	\[
		\lim_{n \to \infty} \varphi \big(n, \lfloor \delta n \rfloor, \epsilon \big) = \bigg[~ \frac{1}{1+\epsilon} - \sqrt{\delta} ~\bigg]_+ ^2.
	\]
	That is, when $n$ is sufficiently large, $|\cV| \geq \exp \big( n \big[ 1/(1+\epsilon) - \sqrt{ \delta }\big]_+^2 \big)$ is necessary 
	for the cone $\Snk( \cV ) $ to be a dual-average $\epsilon$-approximation of $\bS_+^n$.
\end{remark}

As discussed in Remark \ref{rem:union}, our lower bound in Corollary \ref{coro:main.0} can be conservative due to the union bound. 
In fact, we do not know whether our lower bound is tight. Thus, it is possible that even if $N \geq \exp\big( n \cdot \varphi(n,k,\epsilon) \big)$, 
there does not exist any $\cV$ such that $|\cV| = N$ and $\Snk(\cV)$ is a dual-average $\epsilon$-approximation of $\bS_+^n$.


\subsection{Example: the Sparse $k$-PSD Approximation of $\bS_+^n$}\label{sec:coord_kPSD}
In this section, we consider the sparse $k$-PSD approximation, which is a concrete example of the $k$-PSD approximation of $\bS_+^n$ 
(Definition \ref{def:gen_kPSD}) discussed in the previous section. 

\begin{definition}[Sparse $k$-PSD approximation of $\bS_+^n$]\label{def:kPSD}
	Given positive integers $n$ and $1 \leq k \leq n$, the sparse $k$-PSD approximation of $\bS_+^n$ is the set
	\[
		\Snk := \big\{ X \in \bS^n: X_{I} \succeq 0, ~~\forall I \subset [n] \text{ with } |I| \leq k \big\}.
	\]
\end{definition}
\noindent
We observe that the sparse $k$-PSD approximation is an instance of the $k$-PSD approximation $\Snk(\cV)$ 
such that $\cV = \{ V_I:  I \in [n] \text{ with } |I| = k \}\}$ where $V_I = \{ v \in \RR^n: v_i = 0, ~\forall i \not\in I \}$.
Note that $|\cV| = {n \choose k}$.

In Section \ref{sec:k_coordinate_generic}, we examine the implications of Corollary \ref{coro:main.0}
for the sparse $k$-PSD approximation of $\bS_+^n$. In Section \ref{eqn:eps_approx}, we provide a more refined 
analysis that is tailored to $\Snk$, based on properties that are specific to $\Snk$. It turns out that we can 
derive stronger hardness results from the tailored approach.

\subsubsection{A Weak Bound Using Corollary \ref{coro:main.0}}\label{sec:k_coordinate_generic}
First of all, we inspect what the lower bound obtained in Section \ref{sec:restr_kPSD} implies for the sparse $k$-PSD 
approximation of $\bS_+^n$. According to the contrapositive of Corollary \ref{coro:main.0}, when $n$ and $\epsilon > 0$ are fixed, 
$\Snk$ cannot be a dual-average $\epsilon$-approximation of $\bS_+^{n}$ if $k$ satisfies the following inequality:
\begin{equation}\label{eqn:coordinate}
	{n \choose k} < 
		\exp \Bigg( n \cdot \bigg[~ \frac{1}{1+\epsilon} \frac{w_G \big( \base{\bS_+^{n}} \big)}{\sqrt{2 n}} - \sqrt{\frac{k}{n}} ~ \bigg]_+ ^2  \Bigg).
\end{equation}

Let's assume $k = \delta n$ for some $0 < \delta < 1$ and $n$ tends to infinity. By Stirling's approximation,
\[
	\log {n \choose k} = \big( 1 + o_n(1) \big) H_2 \bigg( \frac{k}{n} \bigg) n,
\]
where $H_2(p) = -p \log p - (1-p) \log (1-p)$ is the binary entropy function defined for $p \in [0,1]$. With this asymptotic approximation 
and the observation that $w_G \big( \base{\bS_+^{n}} \big) / \sqrt{2 n} \leq 1$, we take logarithm of both sides of 
\eqref{eqn:coordinate} to obtain the inequality (in the limit $n \to \infty$),
\begin{equation}\label{eqn:critical}
	H_2 ( \delta )
		<  \bigg[~ \frac{1}{1+\epsilon} - \sqrt{\delta}  ~ \bigg]_+ ^2.
\end{equation}

Given $\epsilon \geq 0$, let $g_{\epsilon}(\delta) := \big[~ \frac{1}{1+\epsilon} - \sqrt{\delta}  ~ \big]_+ ^2 - H_2(\delta)$.
Note that $g_{\epsilon}$ is strictly convex on the interval $\delta \in [0, 1]$ and $g_{\epsilon}(0) > 0$. 
Moreover, if $\epsilon > 0$, then $g_{\epsilon}( 1 / (1+\epsilon)^2 ) < 0$. By the intermediate value theorem, 
there exists a unique $0 < \delta^*(\epsilon) < 1 / (1+\epsilon)^2$ such that $g_{\epsilon}(\delta^*(\epsilon)) = 0$ 
and $ g_{\epsilon}(\delta) > 0$ for all $0 \leq \delta < \delta^*(\epsilon)$.
As a result, if $k/n < \delta^*(\epsilon)$, then $\Snk$ cannot be a dual-average $\epsilon$-approximation of $\bS_+^n$. 
The expressions on both sides of Eq. \eqref{eqn:critical} are illustrated in Figure \ref{fig:generic_kPSD_ftns} for 
a few values of $\epsilon$; the plot of $\delta^*(\epsilon)$ vs $\epsilon$ is depicted in Figure \ref{fig:generic_kPSD_delta}.

\begin{figure}[h]
	\centering
	\begin{subfigure}[b]{0.48\textwidth}
		\centering
		\includegraphics[width=\textwidth]{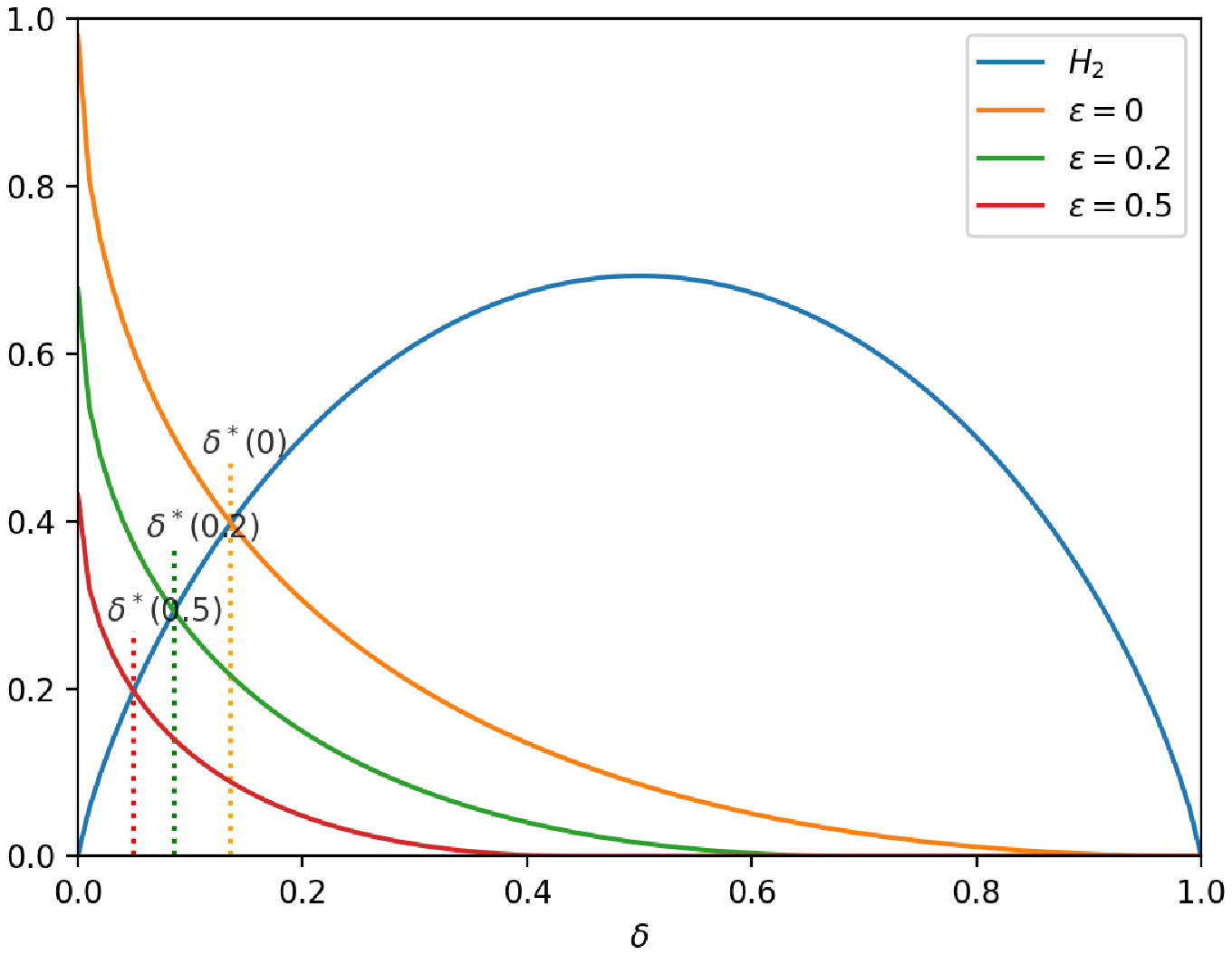}
		\caption{Plot of the expressions in Eq. \eqref{eqn:critical}: $H_2(\delta)$ (entropy) vs 
			$[~ 1/(1+\epsilon) - \sqrt{\delta}  ~ ]_+ ^2$ for $\epsilon = 0$, $0.2$, and $0.5$. 
			The location of $\delta^*(\epsilon)$ are also annotated.}
		\label{fig:generic_kPSD_ftns}
	\end{subfigure}
	\hfill
	\begin{subfigure}[b]{0.48\textwidth}
		\centering
		\includegraphics[width=\textwidth]{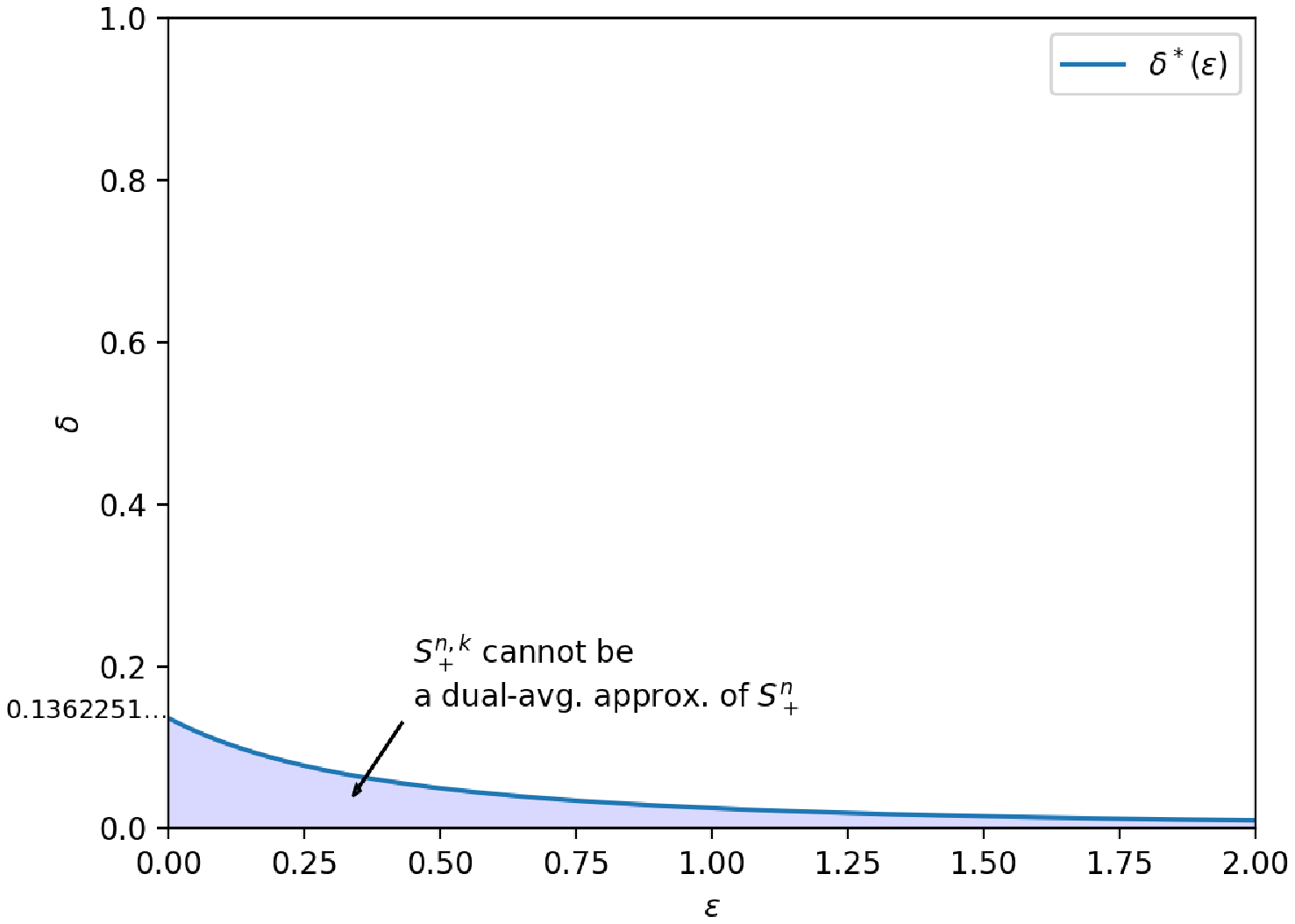}
		\caption{Plot of $\delta^*(\epsilon)$ vs $\epsilon$. 
			For a fixed $\epsilon> 0$, if $k/n$ is contained in the blue region, 
			$\Snk$ cannot be a dual-average $\epsilon$-approximation of $\bS_+^n$.}
		\label{fig:generic_kPSD_delta}
	\end{subfigure}
	\caption{Illustration of the hardness results obtained by applying Corollary \ref{coro:main.0} to the sparse $k$-PSD.}
	\label{fig:naive}
\end{figure}

Recall the definition of $\epsdavg(P, S) = \inf\{ \epsilon > 0: S \text{ is a dual-average }\epsilon 
\text{ approximation of }P \}$, which indicates the `best possible' (i.e., the smallest) $\epsilon > 0$ for which 
$S$ is a dual-average $\epsilon$-approximation of $P$. For fixed $n$ and $k$, the preceding discussion leads to 
a lower bound on $\epsdavg(\bS_+^n, \Snk)$ as
\begin{equation}\label{eqn:edavg_lower.1}
	\epsdavg(\bS_+^n, \Snk)	\geq \sup \Bigg\{ \epsilon > 0: 
		H_2 \Big( \frac{k}{n} \Big) <  \bigg[~ \frac{1}{1+\epsilon} - \sqrt{\frac{k}{n}}  ~ \bigg]_+ ^2 \Bigg\}
		=: \xi(k/n).
\end{equation}

On the one hand, we can already see from the above discussion that for any fixed $\epsilon > 0$, $\Snk$ with $k = o_n(n)$ 
cannot be an $\epsilon$-approximation of $\bS_+^n$ (in the dual-average sense). That is, $k$ must scale linearly with respect to $n$ 
for $\Snk$ to be a good approximation of $\bS_+^n$. 
On the other hand, the lower bound on $k$ from the discussion above -- $k/n \geq \delta^*(\epsilon)$ -- becomes uninformative 
once $k$ increases beyond a certain threshold because $\delta^*(\epsilon) < \delta^*(0) \approx 0.137$ for all $\epsilon > 0$. 
%
In other words, if $k/n > \delta^*(0)$, then we can only get a trivial lower bound $\epsdavg(\bS_+^n, \Snk) > -\infty$, 
and do not know whether $\Snk$ approximates $\bS_+^n$ well or not.

We remark that this is possibly due to the conservative nature of inequality \eqref{eqn:coordinate}, which is inherited from Corollary \ref{coro:main.0}. 
Recall that the cardinality lower bound from Corollary \ref{coro:main.0} is oblivious to the configuration of the subspaces $V_1, \dots, V_N$ in $\cV$. 
That is, it is valid even for the ``best'' possible configuration of subspaces to imitate the expressive power of the full-sized PSD cone.
Nevertheless, the subspaces of $k$-sparse vectors have overlaps, and some of them could be redundant. Thus, the general lower bound 
from Corollary \ref{coro:main.0} can be excessively conservative to apply to the sparse $k$-PSD approximation of $\bS_+^n$. 

Indeed, we can acquire a tighter lower bound for $\epsdavg( \bS_+^n, \Snk)$ by using the knowledge about the subspaces of $\Snk$. 
This is the topic that will be discussed in Section \ref{eqn:eps_approx}.

\subsubsection{A More Refined Analysis Tailored to $\Snk$}\label{eqn:eps_approx}
In this section, we derive lower bounds on $\epsbasic( \bS_+^n, \Snk)$ and $\epsdavg( \bS_+^n, \Snk)$ with an analysis that exploits 
specific properties of $\Snk$. More precisely, we construct a matrix on the boundary of $ B_H(\Snk)$ to argue a lower bound on 
$\epsbasic( \bS_+^n, \Snk)$, and characterize $\epsdavg( \bS_+^n, \Snk)$ by observing that the Gaussian width of $ B_H^*(\Snk)$ is 
the expectation of the largest $k$-sparse eigenvalue of a standard Gaussian random matrix. The resulting lower bounds imply stronger 
hardness results for approximating $\bS_+^n$ with $\Snk$ than those discussed in Section \ref{sec:k_coordinate_generic}.
%

\paragraph{Hardness of $\epsilon$-approximation}
First of all, we discuss how hard it is to approximate $\bS_+^n$ with $\Snk$ in the $\epsilon$-approximation 
sense (see Definition \ref{defn:approx}) when $k$ is small. For that purpose, we consider a specific matrix on the line segment connecting 
$\frac{1}{n} \bOne_n \bOne_n^T $ and $\frac{1}{n} I_n$ where $\bOne_n \in \RR^{n \times 1}$ denotes the $n \times 1$ column matrix with all entries equal to $1$. 
Specifically, we construct a matrix $M \in B_H(\Snk)$ that is far away from $\base{\bS_+^n}$, and prove a lower bound for $\epsilon > 0$ 
as a necessary condition for $M \in (1 + \epsilon) \cdot \base{\bS_+^n}$. 

\begin{proposition}\label{prop:e-approx}
	If $\Snk$ is an $\epsilon$-approximation of $\bS_+^n$, then $k > \frac{n-1}{1+\epsilon}$.
\end{proposition}
\begin{proof}
	Let $P_1(n) := \bOne_n \bOne_n^T / n$ and $P_2(n) := I_n - P_1(n)$. Note that $P_1(n)$ and $P_2(n)$ are projection matrices. 
	For $a, b \in \RR$, we define 
	\[
		G(a,b;n) := a P_1(n) + b P_2(n).
	\] 
	It is easy to verify that the eigenvalues of $G(a,b;n)$ are $a$ with multiplicity $1$, and $b$ with multiplicity $n-1$.

	Next, recall from Definition \ref{def:kPSD} that $G(a,b;n) \in \Snk$ if and only if $G(a,b;n)_{[k]} \succeq 0$. 
	Observe that $G(a,b;n)_{[k]} = \frac{ka + (n-k)b}{n} P_1(k) + b P_2(k) \succeq 0$ if and only if $ka + (n-k)b \geq 0$ 
	and $b \geq 0$. 
	Letting $a = \frac{k-n}{n(k-1)}$ and $b = \frac{k}{n(k-1)}$, we observe that (1) $G(a,b;n) \in \Snk$ because 
	$ka + (n-k)b = 0$ and $b \geq 0$; and (2) $G(a,b;n) \in H$ because $\tr ~ G(a,b;n) = a + b (n-1) = 1$. 
	Next, we can also verify that $G(a,b;n) - \frac{1}{n} I_n \in (1+\epsilon) \cdot \base{\bS_+^n}$ if and only if $\epsilon \geq \frac{n-k}{k-1}$. 
	It is because $G(a,b;n) + \frac{\epsilon}{n} I_n = G\big(a + \frac{\epsilon}{n}, b+\frac{\epsilon}{n}; n \big) \in \bS_+^n$ if and only if 
	$a + \frac{\epsilon}{n} \geq 0$. Rewriting $\epsilon \geq \frac{n-k}{k-1} $ as a condition for $k$ in terms of $\epsilon$, 
	we obtain $k \geq \frac{n-1}{1+ \epsilon} + 1$. 
\end{proof}

Alternatively, when $k$ is fixed, Proposition \ref{prop:e-approx} implies that
\begin{equation}\label{eqn:edavg_lower.vanilla}
	\epsbasic(\bS_+^n, \Snk)	\geq \frac{n-k}{k-1} \geq \frac{1 - k/n}{k/n} =: \zeta(k/n).
\end{equation}

\paragraph{Hardness of dual average $\epsilon$-approximation}
Next, we re-examine how well $\Snk$ can approximate $\bS_+^n$ in the dual-average sense 
(Definition \ref{defn:average_approx_dual}) to find a better lower bound on $\epsdavg(\bS_+^n, \Snk)$. 
We use the duality between $\Snk$ and its dual cone, $ (\Snk)^* = \cone\{ vv^T: v \in \RR^n \text{ with } \| v \|_0 \leq k \}$, 
which is the cone of matrices that have factor width at most $k$ \cite{boman2005factor}.

Observe that $\bS_+^n \cap H = \{ X \in \bS_+^n: \tr(X) = 1 \} = \conv \{ vv^T: x \in \RR^n, ~\|v\|_2 = 1 \}$. For any $G \in \bS^n$, 
$\max_{X \in \bS_+^n \cap H} \langle G, X \rangle = \lambda_1 (G)$ and thus, $w_G\big( \bS_+^n \cap H \big)$ 
is equal to the expectation of the largest eigenvalue of a random matrix that has the standard Gaussian distribution in $\bS^n$ 
(Definition \ref{def:std_Gauss}). 
Likewise, $(\Snk)^* \cap H = \conv \{ vv^T: x \in \RR^n, ~\|v\|_2 = 1, ~\|v\|_0 \leq k  \}$, and 
$\max_{X \in (\Snk)^* \cap H} \langle G, X \rangle$ is the largest $k$-sparse eigenvalue of $G$. 
Based on these observations, we show an asymptotic upper bound on the ratio $w_G\big( B_H^*( \Snk ) \big) 
/ w_G( B_H^*(\bS_+^n) )$ in Proposition \ref{prop:upper_affine} that subsequently leads to a tighter lower bound on 
$\epsdavg(\bS_+^n, \Snk)$ in \eqref{eqn:edavg_lower.2}. 

\begin{proposition}\label{prop:upper_affine}
	Fix $0 < \delta < 1$ and let $k = \lfloor \delta n \rfloor$. Then
	\begin{equation}
		\lim_{n \to \infty} \frac{ w_G\big( B_H^* (\Snk ) \big) }{w_G\big( B_H^* (\bS_+^n) \big)} 	
			\leq 	\bigg( \int_0^{\delta} Q_{\chi^2}(1-s) ds \bigg)^{1/2},
	\end{equation}
	where $Q_{\chi^2}$ denotes the quantile function\footnote{That is, $Q_{\chi^2}(s) := \inf\{ x \in \RR : F_{\chi^2}(x) \geq s \}$ for $0 < s \leq 1$ 
	where $F_{\chi^2}$ be the cumulative distribution function of the $\chi^2$-distribution with one degree of freedom.} of the $\chi^2$-distribution with one degree of freedom. 
	Moreover,
	\begin{equation}\label{eqn:integral}
		\int_0^{\delta} Q_{\chi^2}(1-s) ds
			= \delta + \sqrt{\frac{2}{\pi}} \Phi^{-1} \bigg( 1 - \frac{\delta}{2}\bigg) \exp \bigg( - \frac{1}{2} \bigg[ \Phi^{-1} \Big( 1 - \frac{\delta}{2}\Big) \bigg]^2 \bigg)
	\end{equation}
	where $\Phi(x)$ is the cumulative distribution function of the standard normal distribution. 
\end{proposition}

Before we prove Proposition \ref{prop:upper_affine}, we note that it implies the following lower bound in the asymptotic limit $n \to \infty$:  
\begin{equation}\label{eqn:edavg_lower.2}
	\epsdavg(\bS_+^n, \Snk)	\geq \bigg( \int_0^{\delta} Q_{\chi^2}(1-s) ds \bigg)^{-1/2} - 1 =: \psi(k/n).
\end{equation}
See Figure \ref{fig:summary} (left) in Section \ref{sec:intro} to compare the three lower bounds, $\xi$ (Corollary \ref{coro:main.0} and \eqref{eqn:edavg_lower.1}), 
$\zeta$ (Proposition \ref{prop:e-approx} and \eqref{eqn:edavg_lower.vanilla}), and $\psi$ (Proposition \ref{prop:upper_affine} and \eqref{eqn:edavg_lower.2}). 
We make two remarks: one on the advantage of tailored analysis for $\Snk$; and the other on comparing the rate of convergence for $\zeta$ vs $\psi$.
\begin{itemize}
	\item
	(Generic vs tailored) 
	The lower bound $\psi$ gives a sharper lower bound than $\xi$. In particular, $\psi(\delta) > 0$ for all $0 < \delta < 1$ and $\psi$ 
	gracefully converges to $0$ as $k/n \to 1$, whereas $\xi(\delta) = 0 $ for all $\delta \geq \delta^*(0)$. 
	\item
	($\epsilon$-approx. vs dual-avg. $\epsilon$-approx.) 
	We can see from the expression in \eqref{eqn:integral} that $\psi(1-\delta) = \Theta_{\delta}(\delta^3)$ as $\delta \to 0$. 
	This sharply contrasts with $\varphi(1-\delta) = \Theta_{\delta}(\delta)$. 
	That is, $\Snk$ gets harder to approximate $\bS_+^n$ in both senses as $k$ diminishes from $n$, but at a much slower rate in the dual-average sense.
	
\end{itemize}


\begin{proof}[Proof of Proposition \ref{prop:upper_affine}]
Fix $k \in \{0, 1, \dots, n\}$. Let $T = \{ u \in \RR^n: \|u\|_2 \leq 1, \|u\|_0 \leq k \}$ and observe that 
\[
	w_G\big( B_H^* (\Snk ) \big) 
		= w_G\big( (\Snk )^* \cap H \big)
		= \bbE_G \bigg[ \sup_{u \in T} \inner{G}{ uu^T} \bigg].
\]
We consider a Gaussian process $(X_u)_{u \in T}$ such that $X_u = u^T G u + \gamma$ with $G$ being 
standard Gaussian in $\bS^n$ and $\gamma \sim N(0,1)$ independent of $G$. It is easy to verify that 
\[
	\bbE_G \bigg[ \sup_{u \in T} \inner{G}{ uu^T} \bigg]
		= \bbE_{G, \gamma} \bigg[ \sup_{u \in T} \big\{ u^T G u + \gamma \big\} \bigg]
		=   \bbE_{G, \gamma} \bigg[  \sup_{u \in T} X_u \bigg].
\]

Next, we introduce an instrumental Gaussian process $(Y_u)_{u \in T}$ such that $Y_u = g^Tu$ with $g \sim N(0, 2 I_n)$. 
It is easy to check that for all $u, v \in T$, (1) $\bbE X_u = \bbE Y_u = 0$; (2) $\bbE X_u^2 = \bbE Y_u^2 = 2$; 
and (3) $\bbE X_u X_v - \bbE Y_u Y_v = (1 - u^T v )^2 \geq 0$. Now we can apply Slepian's lemma 
(Lemma \ref{lem:slepian}) to obtain $\bbE_{G, \gamma} \big[ \sup_{u \in T} X_u \big] 
\leq \bbE_{g \sim N(0, 2I_n)} \big[ \sup_{u \in T} Y_u \big]$. 
Then it follows that 
\[
	w_G\big( B_H^* (\Snk ) \big) 
		=   \bbE_{G, \gamma} \bigg[  \sup_{u \in T} X_u \bigg]
		\leq \bbE_{g \sim N(0, 2I_n)} \big[ \sup_{u \in T} Y_u \big]
		= \bbE_{g \sim N(0, 2I_n)} \sup_{u \in \RR^n \atop \|u\|_2\leq 1, \|u\|_0 \leq k} g^T u.
\]
Therefore,
\[
	\frac{1}{\sqrt{2n}} w_G\big( B_H^* (\Snk ) \big)  
		\leq \bbE_{g \sim N(0, 2I_n) } \Bigg[ \frac{\| g \|_2}{ \sqrt{2n} } \sup_{u \in \RR^n \atop \|u\|_2\leq 1, \|u\|_0 \leq k} \frac{ g^T u } { \| g \|_2 } \Bigg].
\]
Note that when $g \sim N(0, 2I_n)$, $\frac{\| g \|_2}{ \sqrt{2n}} \to 1$ in probability as $n \to \infty$. 
Thus, it suffices to identify the limit of $\sup_{u \in \RR^n \atop \|u\|_2\leq 1, \|u\|_0 \leq k} \frac{ g^T u } { \| g \|_2 }$ (in probability)
to compute the expectation on the right-hand side. 

Given $x \in \RR^n$, we let $(x_i^2)^{\downarrow}$ denote the $i$-th largest element in the set $\{ x_1^2, x_2^2, \dots, x_n^2 \}$. 
Observe that 
\[
	\sup_{u \in \RR^n \atop \|u\|_2\leq 1, \|u\|_0 \leq k} \frac{g^T u }{\|g\|_2}
		= \frac{1}{\|g\|_2} \frac{ \sum_{i=1}^k (g_i^2)^{\downarrow} }{ \sqrt{ \sum_{i=1}^k (g_i^2)^{\downarrow} }} 
		= \bigg( \frac{1}{\|g\|_2^2} \sum_{i=1}^k (g_i^2)^{\downarrow} \bigg)^{1/2}
\]
and that $(g_1^2)^{\downarrow} \geq (g_2^2)^{\downarrow} \geq \dots \geq (g_n^2)^{\downarrow}$ are $\chi^2$ order statistics of degree 1, 
multiplied by a factor of 2. It is well known from literature on extreme order statistics (e.g., \cite[Theorem 2.7]{o2016eigenvectors}) that 
for any fixed $0 < \delta < 1$,
\[
	\frac{1}{\| g \|_2^2} \sum_{i=1}^{\lfloor \delta n \rfloor} (g_i^2)^{\downarrow} 
		\longrightarrow
		\int_0^{\delta} Q_{\chi^2}(1-s) ds		\qquad\text{in probability as }n \to \infty.
\]

Combining these observations and the well-known fact that $\lim_{n \to \infty} \frac{w_G( B_H^* (\bS_+^n) )} {\sqrt{2n}} = 1$, 
cf. Remark \ref{rem:width_Sn}, we obtain the desired inequality: 
\[
	\lim_{n \to \infty} \frac{ w_G\big( B_H^* (\Snk ) \big) }{w_G\big( B_H^* (\bS_+^n) \big)} 	
		= \lim_{n \to \infty} \frac{\sqrt{2n}}{w_G\big( B_H^* (\bS_+^n) \big)}  \lim_{n \to \infty} \frac{ w_G\big( B_H^* (\Snk ) \big) }{\sqrt{2n}} 
		\leq 	\bigg( \int_0^{\delta} Q_{\chi^2}(1-s) ds \bigg)^{1/2}.
\]
We conclude the proof by computing the integral in the upper bound. An explicit formula for the integral is well known; see \cite[Remark 2.8]{o2016eigenvectors}, for example.
\begin{align*}
	\int_0^{\delta} Q_{\chi^2}(1-s) ds
		&= 2 \int_{ \Phi^{-1}( 1 - \frac{\delta}{2})}^{\infty} s^2 \Phi' (s) ds
		= \delta + \sqrt{\frac{2}{\pi}} \Phi^{-1} \bigg( 1 - \frac{\delta}{2}\bigg) \exp \bigg( - \frac{1}{2} \bigg[ \Phi^{-1} \Big( 1 - \frac{\delta}{2}\Big) \bigg]^2 \bigg).
\end{align*}
\end{proof}

\paragraph{Hardness of average $\epsilon$-approximation}
As a matter of fact, we can derive the following corollary from Proposition \ref{prop:upper_affine} by applying Urysohn's inequality 
(Lemma \ref{lem:urysohn}), thereby obtaining an asymptotic lower bound on $\epsavg(\bS_+^n, \Snk)$ (see Definition \ref{defn:average_approx}).

\begin{corollary}\label{prop:avg_e-approx}
	Fix $0 < \delta < 1$ and let $k = \lfloor \delta n \rfloor$. Then
	\[
		\lim_{n \to \infty} \frac{ w_G\big( B_H (\Snk ) \big) }{w_G\big( B_H (\bS_+^n) \big)} 	
			\geq \frac{1}{4} \bigg( \int_0^{\delta} Q_{\chi^2}(1-s) ds \bigg)^{-1/2}.
	\]
\end{corollary}

Corollary \ref{prop:avg_e-approx} implies that $\epsavg(\bS_+^n, \Snk) \geq \frac{1}{4} \big( \int_0^{\delta} Q_{\chi^2}(1-s) ds \big)^{-1/2} - 1$. 
Note that this lower bound is more conservative than the lower bound for $\epsdavg(\bS_+^n, \Snk)$ in \eqref{eqn:edavg_lower.2}, 
due to the additional multiplier $1/4$ that arises from the use of Urysohn's inequality. 
It might be possible to derive a better lower bound for $\epsavg(\bS_+^n, \Snk)$, which is beyond the scope of this paper.
	
\begin{proof}[Proof of Corollary \ref{prop:avg_e-approx}]
	By Lemma \ref{lem:base}, we observe that $B_H( \Snk ) = - \frac{1}{n} B_H^*( \Snk )^{\circ}$. 
	It follows from Lemma \ref{lem:urysohn} that $w\big( B_H( \Snk ) \big) \geq \frac{1}{n \cdot w\big( B_H^*( \Snk ) \big)}$, 
	and therefore, 
	\[
		w_G\big( B_H( \Snk ) \big) \geq \frac{\kappa_d^2}{n \cdot w_G\big( B_H^*( \Snk ) \big)}
	\]
	where $d = {n+1 \choose 2} - 1$ is the dimension of $H$. 
	Since $\kappa_d^2 \geq d - \frac{1}{2}$, we obtain for any $0 < \delta < 1$,
	\[
		\lim_{n \to \infty} \frac{w_G \big( B_H( \Snk) \big)}{\sqrt{2n}}  
			\geq \lim_{n \to \infty} \frac{\sqrt{2n}}{w_G \big( B_H^*( \Snk) \big)} \frac{1}{2n^2} \bigg\{ {n+1 \choose 2 } - \frac{3}{2} 	\bigg\}
			= \frac{1}{4 f(\delta)}.
	\]
\end{proof}

\section{Approximate Extended Formulations of $\bS_+^n$}\label{sec:result.extended}
Now we further extend our discussion beyond the $k$-PSD approximation. 
Specifically, we consider an arbitrary approximation of $\bS_+^n$ through extended formulations. 
This defines a much broader class of approximations as we are allowed to introduce as many new variables 
as we want. However, even in this case, at least superpolynomially many $k \times k$ PSD constraints 
are required to approximate $\bS_+^n$ when $k \ll n$. 
In Section \ref{sec:main_thms}, we present our two main theorems about the extension complexity lower bounds 
that hold for any $\epsilon$-approximation of $B_H(\bS_+^n)$. 
Sections \ref{sec:proof_thm.1} and \ref{sec:proof_thm.2} are dedicated to the proof of the theorems. 

\subsection{Theorem Statements}\label{sec:main_thms}
Recall that $\base{\bS_+^n} = \bS_+^n \cap H - \frac{1}{n} I_n$. In this section, we present two main theorems on the hardness 
of approximating $\base{\bS_+^n}$ with a small number of $k \times k$ PSD constraints. Our first theorem is about 
an $\bS_+^k$-extension complexity lower bound that holds for any $\epsilon$-approximation of $\base{\bS_+^n}$. 

\begin{theorem}\label{thm:main.1}
	There exists a constant $C > 0$ such that if $S$ is an $\epsilon$-approximation of $\base{\bS_+^n}$, then 
	\[
		\xc_{\bS_+^k}(S) \geq
			\exp \bigg( C \cdot \min \bigg\{ \sqrt{ \frac{ n }{ 1+\epsilon} }, ~ \frac{1}{1+\epsilon}~ \frac{n}{k} \bigg\} \bigg).
	\]
\end{theorem}

Theorem \ref{thm:main.1} suggests that at least $\Omega_n(\exp(\sqrt{n}))$ copies of $\bS_+^k$ are required
to approximate $\bS_+^n$ when $k = O_n(\sqrt{n})$. When $k = \Omega_n( \sqrt{n} )$, this extension complexity 
lower bound gracefully decreases to $1$ as $k$ increases to $n$. We remark that Theorem \ref{thm:main.1} holds 
for arbitrary $k$, and thus, extends the result of Fawzi \cite[Theorem 1]{fawzi2018polyhedral} beyond the special case of $k=1$. 
A more formal version of Theorem \ref{thm:main.1} and its proof are deferred until Section \ref{sec:proof_thm.1}.

Next, we consider the $\bS_+^k$-extension complexity of a set that is an average $\epsilon$-approximation of $\base{\bS_+^n}$.
\begin{theorem}\label{thm:main.2}
	There exists a constant $C > 0$ such that if $S$ is an average $\epsilon$-approximation of $\base{\bS_+^n}$, then 
	\[
		\xc_{\bS_+^k}(S) \geq
			\exp \bigg( C \cdot \min \bigg\{ \Big( \frac{ n }{ (1+\epsilon)^2} \Big)^{1/3} , ~ \frac{1}{1+\epsilon} \sqrt{\frac{n}{k}} \bigg\} \bigg).
	\]
\end{theorem}
\noindent
Theorem \ref{thm:main.2} is a stronger result than Theorem \ref{thm:main.1} because it provides an extension complexity 
lower bound for a broader range of sets that approximate $\base{\bS_+^n}$. Again, this result subsumes \cite[Theorem 2]{fawzi2018polyhedral} 
as a special case for $k=1$. Specifically, Theorem \ref{thm:main.2} states that even if we relax the notion of approximation, 
we still need at least superpolynomially many number of $k \times k$ PSD constraints to approximate $\bS_+^n$ when $k$ is small, 
namely, when $k$ is smaller than $\frac{1}{(1+\epsilon)^2}\frac{n}{\log^2 n}$.
A more formal version of Theorem  \ref{thm:main.2} and its proof can be found in Section \ref{sec:proof_thm.2}.

Theorem \ref{thm:main.1} and Theorem \ref{thm:main.2} imply that any set that well approximates $\base{\bS_+^n}$ must have 
$\bS_+^k$-extension complexity at least superpolynomially large in $n$ for all $k$ much smaller than $n$. Thus, we conclude that 
it is impossible to approximate $\base{\bS_+^n}$ using only polynomially many $k \times k$ PSD constraints, for \emph{any} construction 
of the approximating set. Note that these  are stronger hardness results than those discussed in Section \ref{sec:results.1}, which only apply 
to the $k$-PSD approximations. Lastly, we mention that we do not know whether our lower bounds are tight. Thus, it might be possible 
to achieve even stronger lower bounds by means of a more sophisticated analysis.

\subsection{Proof of Theorem \ref{thm:main.1}}\label{sec:proof_thm.1}
Let $c = \max \big\{ \sqrt{c_1 / 2 \log 3}, ~ \sqrt{2} c_2 \big\}$ denote an absolute constant with $c_1, c_2>0$ being 
the constants that appear in Lemma \ref{lemma:hanson_wright}. We state a full version of Theorem \ref{thm:main.1} as follows. 
\begin{theorem}\label{thm:main_full.1}
	If $S$ is an $\epsilon$-approximation of $\base{\bS_+^n}$, then for all positive integer $1 \leq k \leq n$,
	\begin{align*}
		\log \xc_{\bS_+^k}(S)
			\geq - \frac{ \alpha + \beta }{2} + \sqrt{ \Big( \frac{ \alpha - \beta }{2} \Big)^2 + \gamma}
	\end{align*}	where
	\[
		\alpha = 2k \log 3, \qquad\beta = \log \Big(  \frac{n^3}{8 \cdot c \log 3 } \Big) , \qquad \gamma = \frac{1}{2 e \cdot c} \frac{n-1}{(1+\epsilon)}.
	\] 
\end{theorem}

Now we discuss how Theorem \ref{thm:main.1} can be derived from Theorem \ref{thm:main_full.1}. 
Suppose that $n$ is sufficiently large, tending to infinity. 

\begin{itemize}
\item
When $k = o_n(\sqrt{ \frac{n}{1+\epsilon}})$, we observe that 
$\gamma \gg \max \{ \alpha^2, \beta^2 \}$, and therefore, $- \frac{ \alpha + \beta }{2} + \big\{ ( \frac{ \alpha - \beta }{2} )^2 + \gamma \big\}^{1/2} \approx \sqrt{\gamma}$.

\item
When $k = \omega_n(\sqrt{ \frac{n}{1+\epsilon}})$, $\alpha \gg \max \{ \beta, \sqrt{\gamma}  \}$. 
Thus, $ \big\{ ( \frac{ \alpha - \beta }{2} )^2 + \gamma \big\}^{1/2} \approx \frac{\alpha}{2} \big( 1 + \frac{4\gamma}{\alpha^2} \big)^{1/2} 
\approx \frac{\alpha}{2} \big( 1 + \frac{2\gamma}{\alpha^2}\big)$. As a result, 
$- \frac{ \alpha + \beta }{2} + \big\{ ( \frac{ \alpha - \beta }{2} )^2 + \gamma \big\}^{1/2} \approx \frac{\gamma}{\alpha}$.
\end{itemize}

In the rest of this section, we prove Theorem \ref{thm:main_full.1}. Our proof is based on similar arguments to those 
found in the proof of \cite[Theorem 1]{fawzi2018polyhedral}, but with appropriate adaptations. Indeed, our results 
can be seen as an extension of Fawzi's beyond the special case with $k=1$, which is made possible by introducing 
different notions of normalization, \eqref{eqn:normalization}, and decomposition of $\bS_+^k$-factors into sharp and flat components, \eqref{eqn:decomposition}.

\begin{proof}[Proof of Theorem \ref{thm:main_full.1}]
We begin with a rough sketch of the main ideas used in the proof. First, we consider the generalized slack matrix $s$ of the pair 
$\big( \base{\bS_+^n}, ~(1+\epsilon) \base{\bS_+^n} \big)$ restricted to the hypercube $H_n$. 
In light of the generalized Yannakakis theorem (Lemma \ref{lem:gen_yannakakis}), the $\bS_+^k$-extension complexity of $S$ 
is bounded from below by the $\bS_+^k$-rank of the slack matrix $s$, cf. \eqref{eqn:yannakakis}. 
Thus, it suffices to prove a lower bound for $\rank_{\bS_+^k}(s)$.

To this end, we express the slack matrix $s$ in two equivalent ways: 
one obtained from the knowledge about the extreme points of $\base{\bS_+^n}$, and the other obtained by assuming that 
$s$ admits a $\bS_+^k$-factorization having $N$ factors. Interpreting the extreme points of $\base{\bS_+^n}$ and $(1+\epsilon) \base{\bS_+^n}$ as formal 
variables, $x$ and $y$, we may view the two expressions of the slack matrix as bivariate polynomials. Next, we `smooth out' 
the two expressions with respect to one variable, $x$, by taking projection onto the harmonic subspace of degree 2; 
and then take expectation with respect to the other variable, $y$. Comparing the two resulting expressions, we derive 
a lower bound on the number of factors $N$, which implies a lower bound on the $\bS_+^k$-extension complexity of $S$.

\paragraph{Step 1. Slack Matrix and $\bS_+^k$-Factorization}
We consider the (generalized) slack operator associated to the pair $\big(\base{\bS_+^n}, ~(1+\epsilon)\base{\bS_+^n} \big)$. 
Let $\bbS^{n-1} = \{ x \in \RR^n: \| x \|_2 = 1\}$. Observe that the extreme points of $\base{\bS_+^n}$ are $\tilde{x}\tilde{x}^T - \frac{1}{n} I_n$ for 
$\tilde{x} \in \bbS^{n-1}$, and that $\big( (1+\epsilon)\base{\bS_+^n} \big)^{\circ} = - \frac{n}{1+\epsilon} \base{\bS_+^n}$. Thus, we are led to study 
the following infinite matrix:
\[
	(\tilde{x}, \tilde{y}) \in \bbS^{n-1} \times \bbS^{n-1} \mapsto 1 - \inner{ \tilde{x} \tilde{x}^T - \frac{1}{n} I_n } { -\frac{n}{1+\epsilon} \Big( \tilde{y} \tilde{y}^T - \frac{1}{n} I_n \Big) } 
		= \frac{n}{1+\epsilon} ( \tilde{x}^T \tilde{y})^2 + \frac{\epsilon}{1+\epsilon}.
\]

We consider the PSD rank ($\bS_+^k$-rank) of the finite submatrix restricted to $\tilde{x}, \tilde{y} \in \big\{ -\frac{1}{\sqrt{n}}, \frac{1}{\sqrt{n}} \big\}^n \subset \bbS^{n-1}$. 
Specifically, we consider the following matrix $s$ defined on the $n$-dimensional hypercube, 
with a proper reparametrization ($x = \sqrt{n} \tilde{x}$ and $y = \sqrt{n} \tilde{y}$):
\begin{equation}\label{eqn:slack_hypercube}
	s: (x,y) \in \{ -1, 1\}^n \times \{-1, 1\}^n \mapsto \frac{1}{1+\epsilon} \bigg( \frac{1}{n}(x^Ty)^2 + \epsilon \bigg).
\end{equation}
Assuming that we can write the matrix \eqref{eqn:slack_hypercube} as a sum of $N$ trace inner products of $\bS_+^k$ factors, we have
\begin{equation}\label{eqn:slack_two_repns_psd}
	 \frac{1}{1+\epsilon} \bigg( \frac{1}{n}(x^Ty)^2 + \epsilon \bigg) = s(x,y) = \sum_{i=1}^N \inner{ f_i(x) }{ g_i(y) },	\qquad \forall x, y \in H_n
\end{equation}
where $f_i, g_i: H_n \to \bS_+^k$ are some matrix-valued functions on $H_n$. 

In this proof, we use the two expressions of $s(x,y)$ in \eqref{eqn:slack_two_repns_psd} to derive a lower bound on $N$.
First, we fix $y \in H_n$ and `smooth out' the expressions on both sides of \eqref{eqn:slack_two_repns_psd} with respect to $x$ 
by taking projection onto the space of harmonic polynomials of degree $2$. Then we plug $x = y$ and consider the expectation of the smoothed functions with respect to $y \in H_n$. 

More precisely, for each fixed $y \in H_n$, we let $q_y(x) = (x^Ty)^2 - n$. Also, let $\mu$ denote the uniform probability measure on $H_n$. 
The inner product of any two functions $f, g: H_n \to \RR$ is defined as $\langle f, g \rangle_{\mu} = \bbE_{x \sim \mu} \big[ f(x) g(x) \big]$.
We observe that $ \langle f(x), q_y(x) \rangle_{\mu} = 2 ~ \proj_2 f(y)$. 

Taking the inner product of  both sides of \eqref{eqn:slack_two_repns_psd} with $q_y(x)$, we obtain
\begin{equation*}
	\bbE_{x \sim \mu}\bigg[ \frac{1}{(1+\epsilon) n} q_y(x)^2 + q_y(x) \bigg] = 
		\sum_{i=1}^N \inner{ \bbE_{x \sim \mu}\big[ q_y(x) f_i(x) \big] }{ g_i (y)}.
\end{equation*}
Subsequently, we get the following equation by taking expectation over $y \sim \mu$:
\begin{equation}\label{eqn:slack_two_repns_psd.a}
	\underbrace{\bbE_{y \sim \mu} \bbE_{x \sim \mu }\bigg[ \frac{1}{(1+\epsilon) n} q_y(x)^2 + q_y(x) \bigg]}_{=: LHS}
		 = 
		\underbrace{\bbE_{y \sim \mu } \sum_{i=1}^N \inner{\bbE_{x \sim \mu} \big[ q_y(x) f_i(x) \big] }{ g_i (y)}}_{=: RHS}.
\end{equation}

The rest of the proof is organized as follows. In Step 2, we compute the expectation on the left-hand side exactly. 
In Step 3, we derive an upper bound on the expectation on the right-hand side as a function of $N$. 
In the end, we obtain the desired lower bound on $N$ in Step 4 by comparing these two quantities.

\paragraph{Step 2. The Left-hand Side of \eqref{eqn:slack_two_repns_psd.a}.}
We evaluate the left-hand side of \eqref{eqn:slack_two_repns_psd.a} based on the following observations:
\begin{align*}
	\bbE_{x \sim \mu }\big[ (x^T y)^2 \big] 
		&= \bbE_{x \sim \mu} \Bigg[ \bigg( \sum_{i=1}^n x_i y_i \bigg)^2 \Bigg]
		= \sum_{i, j=1}^n y_i y_j \bbE_{x \sim \mu } [x_i x_j]
		= \sum_{i=1}^n \bbE_{x \sim \mu } [x_i^2] \\
		&= n,\\
	\bbE_{x \sim \mu}\Big[ (x^T y)^4 \Big] 
		&= \bbE_{x \sim \mu} \Bigg[ \bigg( \sum_{i=1}^n x_i y_i \bigg)^4 \Bigg]
		=  \sum_{i=1}^n \bbE_{x \sim \mu} [x_i^4] + 3 \sum_{i=1 \atop j \neq i}^n \bbE_{x \sim \mu} [x_i^2] \cdot \bbE_{x \sim \mu}[ x_j^2]\\
		&= n + 3 n(n-1).
\end{align*}
Therefore, $\bbE_{x \sim \mu}[ q_y(x) ] = \bbE_{x \sim \mu} \big[ (x^Ty)^2 - n \big] = 0$ and $\bbE_{x \sim \mu}[ q_y(x)^2 ] = \bbE_{x \sim \mu} \big[ (x^T y)^4 - 2n (x^T y)^2 + n^2 \big] = 2n(n-1)$. It follows that for any $y \in H_n$,
\begin{equation}\label{eqn:LHS}
	\bbE_{x \sim \mu}\bigg[ \frac{1}{(1+\epsilon) n} q_y(x)^2 + q_y(x) \bigg]
		= \frac{2}{1+\epsilon}(n-1).
\end{equation}
This does not depend on $y$, and therefore, $ LHS \text{ in }\eqref{eqn:slack_two_repns_psd.a} =  \frac{2}{1+\epsilon}(n-1)$.

\paragraph{Step 3. An Upper Bound for the Right-hand Side of \eqref{eqn:slack_two_repns_psd.a}.}
Now, we prove an upper bound on the right-hand side of \eqref{eqn:slack_two_repns_psd.a}, which has the form 
of an increasing function of $N$. This step is the most technical part of the proof, and is composed of four mini-steps. 

First of all, we claim that we may assume without loss of generality that the factor functions $f_i, g_i$ satisfy
\begin{equation}\label{eqn:normalization}
	\big\| \bbE_{x \sim \mu}[ f_i(x) ] \big\|_{op} = 1, ~~\forall i \in [N]
	\qquad\text{and}\qquad
	\sum_{i=1}^N \tr \big(g_i(y) \big) = 1, ~~\forall y \in H_n.
\end{equation}
Next, in Step 3-B, we decompose each $f_i$ into its sharp component $\fs_i$ and flat component $\ff_i$ 
with a fixed threshold $\Lambda \geq e$ whose value will be determined later in Step 4 of the proof; see \eqref{eqn:decomposition}.
Then due to linearity of expectation, we observe that $RHS = 2 \big[ \bbE_y \sum_{i=1}^N \proj_2 \fs_i(y) g_i(y) + \bbE_y \sum_{i=1}^N \proj_2 \ff_i(y) g_i(y) \big]$. 
Lastly, we prove upper bounds for the two terms separately in Step 3-C and Step 3-D.

The key idea is that for all $i \in [N]$, $\fs_i$ is supported only on a set of small measure due to the normalization, 
and $\| \proj_2 ( v^T \ff_i v ) \|_2 \leq e \log \Lambda$ for all $v \in \bbS^{k-1}$ due to hypercontractivity (Lemma \ref{lemma:harmonic}).

\subparagraph{Step 3-A: Normalization of Factor Functions $f_i, g_i$}
We claim that if $s(x,y)$ admits a $(\bS_+^k)^N$-factorization, then we may assume \eqref{eqn:normalization} 
without loss of generality. More precisely, we show it is possible to normalize arbitrary factor functions $\{ (\tilde{f}_i, \tilde{g}_i) \}_{i=1}^N$ 
to $\{ (f_i, g_i) \}_{i=1}^N$ so that the conditions in \eqref{eqn:normalization} are satisfied.

Suppose that $s(x,y)$, defined in \eqref{eqn:slack_hypercube}, admits a factorization $\{ (\tilde{f}_i, \tilde{g}_i) \}_{i=1}^N$ 
such that $\tilde{f}_i, \tilde{g}_i: H_n \to \bS_+^k$ and $s(x,y) = \sum_{i=1}^N \inner{ \tilde{f}_i(x) }{ \tilde{g}_i(y) }$ for all $x, y \in H_n$. 
For each $i \in [N]$, we can see that $\bbE_x[ \tilde{f}_i(x) ] \in \bS_+^k$, and therefore, we may define $W_i = \bbE_x[ \tilde{f}_i(x) ]^{1/2}$ 
to be the principal square root of $\bbE_x[ \tilde{f}_i(x) ]$. Let $W_i^{\dagger}$ denote the Moore-Penrose pseudoinverse of $W_i$.

Now for each $i \in [N]$, we let $f_i(x) = W_i^{\dagger} \tilde{f}_i(x) W_i^{\dagger}$ and $g_i(y) = W_i \tilde{g}_i(y) W_i$. 
It is easy to verify that $\inner{ f_i(x) }{ g_i(y) } = \tr \big( W_i^{\dagger} \tilde{f}_i(x) W_i^{\dagger} W_i \tilde{g}_i(y) W_i \big) 
= \tr \big( \tilde{f}_i(x) \tilde{g}_i(y) \big) = \big\langle \tilde{f}_i(x), ~ \tilde{g}_i(y) \big\rangle$ for all $x, y \in H_n$. 
Therefore, $\{ (f_i, g_i) \}_{i=1}^N$ also constitutes a valid $\bS_+^k$-factorization of $s(x,y)$. 

It remains to check if $\{ (f_i, g_i) \}_{i=1}^N$ satisfies \eqref{eqn:normalization}.
First, we can easily observe that
\[
	\bbE_{x \sim \mu }[ f_i(x) ]  = W_i^{\dagger} \bbE_{x \sim \mu}[ \tilde{f}_i(x) ] W_i^{\dagger} 
		= W_i^{\dagger} W_i^2 W_i^{\dagger} = \Pi_{\cR(W_i)}
\]
where $\cR(W_i)$ is the range of $W_i$ and $\Pi_{\cR(W_i)}$ is the projection matrix onto $\cR(W_i)$. 
Thus, $\big\| \bbE_{x \sim \mu }[ f_i(x) ] \big\|_{op} =  \| \Pi_{\cR(W_i)} \|_{op} = 1$. 
Next, we revisit \eqref{eqn:slack_two_repns_psd}, fix any $y \in H_n$, and take expectation with respect to $x \sim \mu$. 
On the left-hand side, we obtain $\bbE_{x \sim \mu}\big[ \frac{1}{1+\epsilon} \big( \frac{1}{n}(x^Ty)^2 + \epsilon \big) \big] = 1$ 
because $\bbE_{x \sim \mu }\big[ (x^T y)^2 \big]  = n$ (see Step 2). On the right-hand side, 
we have
\[
	\bbE_{x \sim \mu} \sum_{i=1}^N \inner{ f_i(x) }{ g_i(y) } 
		= \sum_{i=1}^N \inner{ \bbE_{x\sim\mu} [f_i(x)] }{ g_i(y) } 
		= \sum_{i=1}^N \inner{ \Pi_{\cR(W_i)} }{ g_i(y) }  
		= \sum_{i=1}^N \tr~ g_i(y)
\]
because $\Pi_{\cR(W_i)} g_i(y) = g_i(y)$ for all $y \in H_n$, by definition of $g_i$.
Therefore, $\sum_{i=1}^N \tr~ g_i(y) = 1$ for all $y \in H_n$.

\subparagraph{Step 3-B: Decomposition of $f_i$}
We decompose each $f_i$ into its `sharp' (spiky) component $\fs_i$ and the `flat' component $\ff_i$ using 
a fixed threshold $\Lambda$ whose value will be determined later in Step 4 of the proof. 
To be specific, for each $i \in [N]$, we define the component functions $\fs_i, \ff_i: H_n \to \bS_+^k$  as follows. 
Given $x \in H_n$, let $f_i(x) = \sum_{a=1}^k \lambda_a u_a u_a^T$ be the eigendecomposition of $f_i(x)$. 
Then we let 
\begin{equation}\label{eqn:decomposition}
	\fs_i(x) = \sum_{a=1}^k \lambda_a \Ind{\lambda_a > \Lambda}  u_a u_a^T,
	\qquad\text{and}\qquad
	\ff_i(x) = \sum_{a=1}^k \lambda_a \Ind{\lambda_a \leq \Lambda}  u_a u_a^T.
\end{equation}
Observe that $f_i = \fs_i + \ff_i$ and $\big\langle \fs_i(x),~ \ff_i(x) \big\rangle= \tr \big( \fs_i(x) \ff_i(x) \big) = 0$ for all $x \in H_n$. 
From now on, we may refer to $\fs_i$ ($\ff_i$, resp.) as the sharp component (flat component, resp.) of $f_i$.

By linearity of expectation, we can decompose the expression on the right-hand side of \eqref{eqn:slack_two_repns_psd.a} as follows:
\begin{equation}\label{eqn:sharp_flat}
	RHS \text{ in }\eqref{eqn:slack_two_repns_psd.a}
		= 
		\bbE_{y \sim \mu}  \sum_{i=1}^N \inner{\bbE_{x \sim \mu}[ q_y(x) \fs_i(x) ] }{ g_i (y)}
		+
		\bbE_{y \sim \mu}  \sum_{i=1}^N \inner{\bbE_{x \sim \mu}[ q_y(x) \ff_i(x) ] }{ g_i (y)}.
\end{equation}

\subparagraph{Step 3-C. Upper Bound on the Contribution of Sharp Components in \eqref{eqn:sharp_flat}}
In this paragraph, we argue that the first term on the right hand side of \eqref{eqn:sharp_flat} is bounded from above 
by $N \frac{k}{\Lambda} n^3 $. Our argument is based on the following three observations.
\begin{itemize}
	\item
		Let $\supp \fs_i = \{ x \in H_n: \fs_i(x) \neq 0 \}$. Then $| \supp \fs_i | < \frac{k}{\Lambda} ~2^n$ for all $i \in [N]$.
		It is because (i) $\tr~ \bbE_{x \sim \mu}[ \fs_i(x) ] \leq \tr~ \bbE_{x \sim \mu}[ f_i(x) ] \leq k \normop{ \bbE_{x \sim \mu}[ f_i(x) ] } = k$, 
		cf. \eqref{eqn:normalization}; (ii) $\tr~ \bbE_{x \sim \mu}[ \fs_i(x) ] = \frac{1}{2^n} \sum_{x \in H_n} \tr~ \fs_i(x)$; 
		and (iii) $\tr~ \fs_i(x) > \Lambda$ for all $x \in \supp \fs_i$ by definition of $\fs_i$.
		
	\item
		For each $i \in [N]$, $\big\langle \fs_i(x),~ g_i(y) \big\rangle \leq n$ for all $x, y \in H_n$. 
		This follows from Eq. \eqref{eqn:slack_two_repns_psd} because
		\[
			\inner{ \fs_i(x) }{ g_i(y) } \leq \sum_{i=1}^N \inner{ f_i(x) }{ g_i(y) } = \frac{1}{1+\epsilon} \bigg( \frac{1}{n}(x^Ty)^2 + \epsilon \bigg) \leq \frac{1}{1+\epsilon} \big( n + \epsilon \big) \leq n.
		\]
		
	\item
		Lastly, $|q_y(x)| \leq n(n-1)$ for all $x, y \in H_n$ because $q_y(x) = (x^T y)^2 - n \leq n(n-1)$ and $q_y(x) \geq -n$.
\end{itemize}
Combining the three observations above, we can see that for any $y \in H_n$,
\begin{align*}
	\sum_{i=1}^N \inner{\bbE_{x \sim \mu}[ q_y(x) \fs_i(x) ] }{ g_i (y)}
		&= \sum_{i=1}^N \bbE_{x \sim \mu} \Big[ q_y(x) \inner{ \fs_i(x) }{ g_i (y)} \Big]
		= \sum_{i=1}^N \sum_{x \in \supp \fs_i} \frac{1}{2^n} q_y(x) \inner{ \fs_i(x) }{ g_i (y)}\\
		&\leq \sum_{i=1}^N \frac{\big| \supp \fs_i \big|}{2^n} \Big( \max_{x,y \in H_n} \big| q_y(x) \big| \Big)
			\Big( \max_{x,y \in H_n} \inner{ \fs_i(x) }{ g_i (y)} \Big)
		\leq \frac{k}{\Lambda} n^2(n-1) N.
\end{align*}
Taking expectation with respect to $y \sim \mu$, we obtain
\begin{equation}\label{eqn:sharp.a}
	\bbE_{y \sim \mu}  \sum_{i=1}^N \inner{\bbE_{x \sim \mu}[ q_y(x) \fs_i(x) ] }{ g_i (y)} 
		\leq \frac{k}{\Lambda} n^2(n-1) N 
		\leq \frac{k}{\Lambda} n^3 N .
\end{equation}

\subparagraph{Step 3-D. Upper Bound on the Contribution of Flat Components in \eqref{eqn:sharp_flat}}
Next, we prove an upper bound for the second term on the right hand side of \eqref{eqn:sharp_flat}. Our proof is based on 
the concentration of the degree-$2$ harmonic components of bounded functions and the usual $\epsilon$-net argument. 

First, we reduce the matrix-valued function $\ff_i$'s to the supremum of multiple scalar-valued functions 
indexed over a finite set. Given $\enet > 0$, let $\cN$ be an $\enet$-net of $\bbS^{k-1}$ with the smallest 
possible cardinality. Note that $|\cN| \leq \big( 1 + \frac{2}{\enet} \big)^k$ by the well-known upper bound 
on the $\enet$-covering number of $\bbS^{k-1}$. 
Then
\begin{align*}
	\bbE_{y \sim \mu}  \Bigg[ \sum_{i=1}^N \inner{\bbE_{x \sim \mu}[ q_y(x) \ff_i(x) ] }{ g_i (y)} \Bigg]
		&\stackrel{(a)}{\leq} \bbE_{y \sim \mu} \Bigg[ \sum_{i=1}^N \normop{\bbE_{x \sim \mu}[ q_y(x) \ff_i(x) ] } \tr ~g_i (y) \Bigg]\\
		&\stackrel{(b)}{\leq} \bbE_{y \sim \mu} \bigg[ \max_{i \in [N]} \normop{\bbE_{x \sim \mu}[ q_y(x) \ff_i(x) ] } \bigg]\\
		&\stackrel{(c)}{\leq} \frac{1}{1 - 2\enet}\bbE_{y \sim \mu}  \bigg[ \max_{i \in [N] \atop v \in \cN}  \Big| v^T \bbE_{x \sim \mu}\big[ q_y(x) \ff_i(x) \big] v \Big| \bigg]\\
		&= \frac{1}{1 - 2\enet}\bbE_{y \sim \mu}  \bigg[ \max_{i \in [N] \atop v \in \cN} \Big| \bbE_{x \sim \mu}\big[ q_y(x) ~ v^T \ff_i(x) v \big] \Big| \bigg].
\end{align*}
In the above lines, (a) follows from Cauchy-Schwarz inequality; (b) is due to the normalization $\sum_{i=1}^N \tr~g_i(y) \equiv 1$; 
and (c) is obtained by the $\epsilon$-net argument, i.e., if $\cN$ is an $\enet$-net of $\bbS^{k-1}$, then for any $M \in \bS_+^k$, 
$\normop{ M } \leq \frac{1}{1 - 2\enet} \sup_{v \in \cN} v^T M v$. Now it remains to evaluate the expectation in the last line. 

Recall that $\bbE_{x \sim \mu}\big[ q_y(x) v^T \ff_i(x) v \big] = \big\langle  q_y(x), ~ v^T \ff_i(x) v \big\rangle_{\mu} = 2 \proj_2 \big( v^T \ff_i v\big) (y)$.
We observe that for each $(i, v) \in [N] \times \cN$, the derived random variable $\proj_2 \big( v^T \ff_i v\big) (y)$ is sub-exponential with parameters 
$(c_1 \| \proj_2 ( v^T \ff_i v ) \|_2^2/2, ~c_2 \| \proj_2 ( v^T \ff_i v ) \|_2 / \sqrt{2} )$, due to Lemma \ref{lemma:subexp_tail}. 
Here, $c_1, c_2>0$ are the same absolute constants that appear in Lemma \ref{lemma:hanson_wright}.

Next, we find a common upper bound on $\| \proj_2 ( v^T \ff_i v ) \|_2$ that holds for all $(i,v)$.
Note that for all $(i,v)$, $\bbE_{y \in \mu}[v^T \ff_i(y) v] \leq \bbE_{y \in \mu} \| \ff_i (y) \|_{op} \leq \bbE_{y \in \mu} \| f_i(y) \|_{op} = 1$ 
due to the normalization in \eqref{eqn:normalization}, and $0 \leq v^T \ff_i v \leq \Lambda$ by definition of $\ff_i$. 
Thus, we can apply Lemma \ref{lemma:harmonic} to get $\| \proj_2 ( v^T \ff_i v ) \|_2 \leq e \log \Lambda$ for all $(i,v)$, 
provided that we will choose the threshold $\Lambda \geq e$.

Now we can use a result on the expected maximum of $N |\cN|$ sub-exponential random variables (Lemma \ref{lem:maximal}) to obtain
\begin{align*}
	\bbE_{y \sim \mu}  \bigg[ \max_{i \in [N] \atop v \in \cN} \Big| \bbE_{x \sim \mu}\big[ q_y(x) ~ v^T \ff_i(x) v \big] \Big| \bigg]
		&= 2 \cdot\bbE_{y \sim \mu}  \bigg[ \max_{i \in [N] \atop v \in \cN} \big| \proj_2 \big( v^T \ff_i v\big) (y) \big| \bigg]\\
		&\leq 2 \cdot \frac{e \log \Lambda}{\sqrt{2}} \max \bigg\{ \sqrt{2 c_1 \log (N |\cN|)}, ~ 2 c_2 \log (N |\cN|) \bigg\}\\
		&\leq 2 e \log \Lambda \cdot \max \bigg\{ \sqrt{ c_1 \log (N |\cN|)}, ~ \sqrt{2} c_2 \log (N |\cN|) \bigg\}.
\end{align*}

Collecting the pieces in this step, we obtain the following upper bound:
\begin{equation}\label{eqn:flat.a}
	\begin{aligned}
	&\bbE_{y \sim \mu}  \sum_{i=1}^N \inner{\bbE_{x \in H_n}[ q_y(x) \ff_i(x) ] }{ g_i (y)} \\
	&\qquad	\leq \frac{2 e \log \Lambda}{1-2\enet} 
			\max \Bigg\{ \sqrt{ c_1 \log \bigg[ N \Big( 1 + \frac{2}{\enet} \Big)^k \bigg] }, ~ \sqrt{2} c_2 \log \bigg[ N \Big( 1 + \frac{2}{\enet} \Big)^k \bigg] \Bigg\}.
	\end{aligned}
\end{equation}

\paragraph{Step 4. Concluding the Proof}
Lastly, we revisit Eq. \eqref{eqn:slack_two_repns_psd.a} to conclude the proof. Recall that we obtained the value of the left-hand side in Step 2, 
cf. \eqref{eqn:LHS}, and derived an upper bound for the right-hand side in Step 3, cf. \eqref{eqn:sharp_flat}, \eqref{eqn:sharp.a}, and \eqref{eqn:flat.a}.
Putting these together, we have the following inequality that holds for any choice of parameters $\enet$, $\Lambda$ such that $0 < \enet < \frac{1}{2}$ and $\Lambda \geq e$:
\begin{equation}\label{eqn:conclusion}
	\frac{2}{1+\epsilon}(n-1)
		\leq \frac{k}{\Lambda} n^3 N 
			+ \frac{2 e \log \Lambda}{1-2\enet} 
			\max \Bigg\{ \sqrt{ c_1 \log \bigg[ N \Big( 1 + \frac{2}{\enet} \Big)^k \bigg] }, ~ \sqrt{2} c_2 \log \bigg[ N \Big( 1 + \frac{2}{\enet} \Big)^k \bigg] \Bigg\}.
\end{equation}

We choose $\enet = 1/4$ for simplicity because optimizing $\enet$ does not make much difference. 
Observe that $\log(9^k N) \geq 2 \log 3$ for all $k, N \geq 1$. Thus, $\sqrt{ c_1 \log (9^k N ) } \leq \sqrt{c_1 / 2 \log 3} \log(9^k N )$ for all $k, N \geq 1$.
Therefore, $\max \big\{ \sqrt{c_1 \log(9^k N)}, ~ \sqrt{2} c_2 \log(9^k N) \big\} \leq c \log(9^k N)$ where $c = \max\{ \sqrt{c_1 / 2 \log 3}, ~ \sqrt{2} c_2 \}$.

Then, we select $\Lambda$ that minimizes the right-hand side of \eqref{eqn:conclusion}.
It is easy to see that the upper bound is minimized (w.r.t. $\Lambda$) at $\Lambda^* = \frac{k n^3 N}{4 e c \log(9^k N )}$. 
Noticing that $\Lambda^* \leq \frac{k n^3 N}{4 e c \log(9^k )} $ (because $N \geq 1$), we get the following 
quadratic inequality in $\log N$ as a necessary condition for \eqref{eqn:conclusion}:
\begin{align}
	\frac{2}{1+\epsilon}(n-1)
		&\leq 4 e \cdot c \log(9^k N ) \big( 1 + \log \Lambda^* \big)		\nonumber\\
		&\leq 4 e \cdot c \big[ \log N + 2k \log 3 \big] \bigg[ \log N + \log \Big(  \frac{n^3}{8 \cdot c \log 3 } \Big) \bigg].	\label{eqn:quad.thm.2}
\end{align}

Letting $z = \log N \geq 0$, we note that \eqref{eqn:quad.thm.2} is a quadratic inequality of the form $(z+\alpha) (z+\beta) \geq \gamma$ where
\[
	\alpha = 2k \log 3, \qquad\beta = \log \Big(  \frac{n^3}{8 \cdot c \log 3 } \Big) , \qquad \gamma = \frac{1}{2 e \cdot c} \frac{n-1}{(1+\epsilon)}.
\] 
We want to solve this quadratic inequality with an implicit constraint $z \geq 0$ because $N \geq 1$. Observe that 
its discriminant $D = (\alpha-\beta)^2 + 4\gamma > 0$, regardless of $n, k, \epsilon$. Therefore, 
the set of solutions is given as
$\big\{ z \in \RR: ~(z+\alpha) (z+\beta) \geq \gamma,~ z \geq 0 \big\} = \big\{ z \in \RR: ~ z \geq \big[ \frac{-(\alpha+\beta) + \sqrt{D}}{2} \big]_+ \big\} $ 
where $[x]_+ = \max\{ x, 0 \}$.
\end{proof}

%

\subsection{Proof of Theorem \ref{thm:main.2}}\label{sec:proof_thm.2}
The following is a formal version of Theorem \ref{thm:main.2}, which will be proved later in this section. 
\begin{theorem}\label{thm:main_full.2}
	If $S$ is an average $\epsilon$-approximation of $\base{\bS_+^n}$, then for all positive integer $1 \leq k \leq n$,
	\[
		\log \xc_{\bS_+^k}(S)  \geq	\Big\{ \big( \alpha + \sqrt{ \alpha^2 + \beta^3 }\big)^{1/3} + \big( \alpha - \sqrt{ \alpha^2 + \beta^3 }\big)^{1/3} \Big\}^2
			- 2k \log 3
	\]
	where
	\[
		\alpha = \frac{\sqrt{n}}{22000 e (1+\epsilon)}
		\qquad\text{and}\qquad	
		\beta = \frac{1}{3} \bigg\{ \log \bigg( \frac{16(1+\epsilon) k^{1/2} n^{3/2} }{5 \sqrt{ 2 \log 3 }}  \bigg) - 2k \log 3 \bigg\}.
	\]
\end{theorem}

Now we discuss how Theorem \ref{thm:main.2} can be derived from Theorem \ref{thm:main_full.2}. 
For notational brevity in our derivation, we let 
$T_+ = \big( \alpha + \sqrt{ \alpha^2 + \beta^3 }\big)^{1/3}$ and $T_- = \big( \alpha - \sqrt{ \alpha^2 + \beta^3 }\big)^{1/3}$.
Suppose that $n$ is sufficiently large, tending to infinity. 
\begin{itemize}
\item
When $k = o_n \big( \big( \frac{n}{(1+\epsilon)^2} \big)^{1/3} \big)$, we can see that $\alpha^2 \gg | \beta |^3$ and thus, 
$\sqrt{ \alpha^2 + \beta^3 } \approx \alpha$. Therefore, $T_+ + T_- \approx ( 2\alpha )^{1/3}$, and in the end, 
$( T_+ + T_- )^2 - 2k \log 3 \approx ( 2 \alpha )^{2/3} \approx  C \cdot \frac{n^{1/3}}{(1+\epsilon)^{2/3}} $ for some constant $C$.

\item
When $k = \omega_n \big( \big( \frac{n}{(1+\epsilon)^2} \big)^{1/3} \big)$, note that $\beta < 0$ and $\alpha^2 \ll | \beta |^3$. 
Let $\gamma := \sqrt{ \alpha^2 + \beta^3 }$. We observe that $\gamma \approx |\beta|^{3/2} i$, and thus, $|\gamma| \approx |\beta|^{3/2} \gg \alpha$.
Then, we can see that $T_+ = ( \alpha + \gamma )^{1/3} \approx \gamma^{1/3} \big( 1 + \frac{\alpha}{3 \gamma} \big)$, and likewise, 
$T_- \approx \bar{\gamma}^{1/3} \big( 1 + \frac{\alpha}{3 \bar{\gamma}} \big)$ where $\bar{\gamma}$ is the complex conjugate of $\gamma$. 
Then it follows that 
\begin{align*}
	T_+ + T_-
		&\approx \gamma^{1/3} + \bar{\gamma}^{1/3} + \frac{\alpha}{3} \big( \frac{1}{\gamma^{2/3}} + \frac{1}{\bar{\gamma}^{2/3}} \big)
		\approx |\gamma|^{1/3} \big( e^{i \frac{\pi}{6}} + e^{-i \frac{\pi}{6}} \big) + \frac{\alpha}{3 |\gamma|^{2/3}} \big( e^{ -i \frac{\pi}{3}} + e^{i \frac{\pi}{3}} \big)\\
		&= \sqrt{3} |\gamma|^{1/3} \bigg( 1 + \frac{\alpha}{3\sqrt{3} |\gamma|} \bigg).
\end{align*}
Therefore, $( T_+ + T_- )^2 \approx 3 |\gamma|^{2/3} \big( 1 + \frac{2\alpha}{3\sqrt{3} |\gamma|}\big) = 3 |\gamma|^{2/3} + \frac{2}{\sqrt{3}} \frac{\alpha}{|\gamma|^{1/3}}$. 
Lastly, noticing that $ |\gamma|^{2/3} \approx |\beta| \approx \frac{2}{3}k \log 3$, 
we can conclude that $( T_+ + T_- )^2 - 2k \log 3 \approx  C \cdot \frac{1}{1+\epsilon}\sqrt{\frac{n}{k}} $ for some constant $C$.
\end{itemize}

\begin{proof}[Proof of Theorem \ref{thm:main_full.2}]
We follow a similar strategy to that of Theorem \ref{thm:main_full.1} with some modifications. 
Here, we assume $S$ is $\epsilon$-approximation of $\base{\bS_+^n}$ only in the average sense, and thus, 
$S$ can be arbitrarily shaped and $S \subseteq (1+\epsilon) \base{\bS_+^n}$ is not necessarily true. 
Instead, we define a set $Q$ -- to be precise, we let $Q = 10(1+\epsilon) \sqrt{n} \GG^{\circ}$ for $\GG$ to be defined 
in \eqref{eqn:nice_gaussian} -- that contains $S$ in an adaptive manner. Then we consider the generalized slack matrix 
of the pair $\big( \base{\bS_+^n}, ~Q \big)$. We express the slack matrix in two equivalent ways: 
one is obtained from the knowledge about the extreme points of $\base{\bS_+^n}$, and the other is obtained 
by assuming the existence of a $\bS_+^k$-factorization having $N$ factors. Interpreting the extreme points of 
$\base{\bS_+^n}$ and $Q^{\circ}$ as formal variables, $x$ and $G$, we may view the two expressions of 
the slack matrix as bivariate polynomials. As already done in the proof of Theorem \ref{thm:main_full.1}, 
we `smooth out' the two expressions with respect to one variable, $x$; and then take expectation with respect to 
the other variable, $G$. Comparing the two resulting expressions, we derive a lower bound on the number of factors $N$, 
which implies a lower bound on the $\bS_+^k$-extension complexity of $S$.

\paragraph{Step 1. Gaussian Surrogate for $S^{\circ}$ and the Associated Slack Matrix}
Let $\bS_0^n$ denote the set of $n \times n$ symmetric matrices with trace zero, endowed with the trace inner product. 
Let $\cN_0$ denote the standard Gaussian distribution associated to $\bS_0^n$, i.e., 
$G_0 \sim \cN_0$ if $G_0 = G - \frac{\tr G}{n} I_n$ where $G$ has the standard Gaussian distribution in $S^n$.
Then we define a set
\begin{equation}\label{eqn:nice_gaussian}
	\GG = \left\{ G \in \bS_0^n: \big| \inner{G}{X} \big| \leq 5\sqrt{2} w_G(S), ~\forall X \in S \right\}.
\end{equation}
The number $5\sqrt{2}$ is chosen for the convenience of our analysis, and has no special meaning. 
Observe that $w_G(S) \leq (1+\epsilon) \cdot w_G\big( \base{\bS_+^n} \big)  \leq (1+\epsilon) \sqrt{2n}$, cf. Remark \ref{rem:width_Sn}.



Then, we can see that 
\[
	-10 (1+\epsilon) \sqrt{n} \leq \inner{G}{X} \leq 10 (1+\epsilon) \sqrt{n},	\qquad \forall (X, G) \in S \times \GG .
\]
This implies that $\frac{1}{10(1+\epsilon) \sqrt{n}} \GG \subseteq S^{\circ}$, or equivalently, $S \subseteq 10(1+\epsilon) \sqrt{n} \GG^{\circ}$. 

Now we consider the slack operator associated to the pair $\big( \base{\bS_+^n}, ~ 10(1+\epsilon) \sqrt{n} \GG^{\circ} \big)$, 
treating $\frac{1}{10(1+\epsilon) \sqrt{n}} \GG$ as a surrogate for $S^{\circ}$. 
Specifically, we are led to study the following infinite matrix:
\[
	(\tilde{x}, G) \in \bbS^{n-1} \times \GG \mapsto 1 - \inner{ \tilde{x} \tilde{x}^T - \frac{1}{n} I_n } { \frac{1}{10(1+\epsilon)\sqrt{n}} G } 
		= 1 - \frac{1}{10(1+\epsilon)\sqrt{n}} \tilde{x}^T G \tilde{x}.
\]
We consider the PSD rank ($\bS_+^k$-rank) of the submatrix restricted to $\tilde{x} \in \big\{ -\frac{1}{\sqrt{n}}, \frac{1}{\sqrt{n}} \big\}^n 
\subset \bbS^{n-1}$, with a proper reparametrization ($x = \sqrt{n} \tilde{x}$), namely,
\begin{equation}\label{eqn:slack_gaussian}
	 s: (x, G) \in H_n \times \GG \mapsto 1 - \frac{1}{10(1+\epsilon) n\sqrt{n}} x^T G x.
\end{equation}
Assuming that we can write the matrix \eqref{eqn:slack_gaussian} as a sum of $N$ trace inner products of $\bS_+^k$ factors, we have
\begin{equation}\label{eqn:slack_two_repns_gaussian}
	1 - \frac{1}{10(1+\epsilon) n\sqrt{n}} x^T G x 
		= s(x,G) 
		= \sum_{i=1}^N \inner{ f_i(x) }{ g_i(G) },	\qquad \forall (x, G) \in H_n \times \GG
\end{equation}
where $f_i: H_n \to \bS_+^k$ and $g_i: \GG \to \bS_+^k$ are some matrix-valued functions.

Again, we `smooth out' the two expressions of $s(x,G)$ in \eqref{eqn:slack_two_repns_gaussian} and compare them to derive a lower bound for $N$. 
To be precise, for each fixed $G \in \GG$, we let $q_G(x) = - x^T G x$. 
Recall that we let $\mu$ denote the uniform probability measure on $H_n$, and observe that for any function $f: H_n \to \RR$, 
the inner product, $\langle f, q_G(x) \rangle_{\mu} =  \bbE_{x \sim \mu} [ f(x) q_G(x) ]$ is a centered Gaussian random variable.

Taking the inner product of both sides of \eqref{eqn:slack_two_repns_gaussian} with $q_G(x)$, we get
\begin{equation*}
	\bbE_{x \sim \mu}\bigg[ q_G(x) + \frac{1}{10(1+\epsilon) n\sqrt{n}} q_G(x)^2 \bigg] = 
		\sum_{i=1}^N \inner{\bbE_{x \sim \mu}[ q_G(x) \cdot f_i(x) ] }{ g_i (G)}.
\end{equation*}
Letting $\bbE_{G \sim \cN_0 | \GG}[ ~\cdot~ ]$ denote the conditional expectation 	
with respect to $G \sim \cN_0$ given $G \in \GG$, we can see that
\begin{equation}\label{eqn:slack_two_repns_gaussian.a}
	\underbrace{  \frac{1}{10(1+\epsilon) n\sqrt{n}} \cdot \bbE_{G \sim \cN_0 | \GG} \bbE_{x \sim \mu}\big[q_G(x)^2 \big]}_{=: LHS}
		 = 
		\underbrace{\bbE_{G \sim \cN_0 | \GG} \sum_{i=1}^N \inner{\bbE_{x \sim \mu}[ q_G(x) \cdot f_i(x) ] }{ g_i (G)}}_{=: RHS}.
\end{equation}

The rest of the proof is organized as follows. In Step 2, we prove a lower bound for the expectation on the left-hand side. 
In Step 3, we derive an upper bound on the expectation on the right-hand side as a function of $N$. 
In the end, we obtain the desired lower bound on $N$ in Step 4 by comparing these bounds.

\paragraph{Step 2. A Lower Bound for the Left-hand side of \eqref{eqn:slack_two_repns_gaussian.a}.}
We additionally define a set
\[
	\GGb  = \left\{ G \in \bS_0^n:  \bbE_{x \sim \mu}\big[q_G(x)^2 \big] \geq \frac{1}{5} n(n-1) \right\}.
\]
The constant $1/5$ is chosen for the convenience of analysis, and has no special meaning. By the law of total probability, we can see that
\begin{align*}
	\bbE_{G \sim \cN_0 | \GG} \bbE_{x \sim \mu}\big[q_G(x)^2 \big]
		&\geq \bbE_{G \sim \cN_0 | \GG \cap \GGb} \bbE_{x \sim \mu}\big[q_G(x)^2 \big] \cdot \Pr [ G \in \GG \cap \GGb ~|~ G \in \GG].
\end{align*}
Note that $\bbE_{G \sim \cN_0 | \GG \cap \GGb} \bbE_{x \sim \mu}\big[q_G(x)^2 \big] \geq \frac{1}{5} n(n-1)$ by definition of $\GGb$. 
Thus, it suffices to find a lower bound for the conditional probability, $ \Pr [ G \in \GG \cap \GGb ~|~ G \in \GG]$.

It is easy to see that
\[
	 \Pr [ G \in \GG \cap \GGb ~|~ G \in \GG] = \frac{ \Pr [ G \in \GG \cap \GGb ]}{ \Pr [ G \in \GG ]} \geq \frac{ \Pr[ G \in \GG ] - \Pr[ G \not\in \GGb ]}{ \Pr[ G \in \GG ]}.
\]
Observe that $\Pr[ G \in \GG ] \geq 1 - \exp\big( - \frac{(5\sqrt{2}-1)^2}{4\pi} \big) > 0.893$ by Lemma \ref{lemma:gaussian_conc}. 
Now it remains to show an upper bound for $\Pr[ G \not\in \GGb ]$.

We use standard concentration results for the chi-square distribution.
Note that if $G \sim \cN_0$, then $q_G(x) = - \tr~G - 2\sum_{i < j} G_{ij} x_i x_j = - 2\sum_{i < j} G_{ij} x_i x_j$, 
and therefore, $ \bbE_{x \sim \mu}\big[q_G(x)^2 \big] = 4 \sum_{i < j} G_{ij}^2$. 
Thus, we have $\bbE_{G \sim \cN_0} \bbE_{x \sim \mu} [q_G(x)^2 ] = 4 {n \choose 2} \frac{1}{2} = n(n-1) $.
Using an exponential inequality for chi-square distribution (e.g., \cite[Lemma 1]{laurent2000adaptive}), we obtain
$\Pr[ G \not\in \GGb ] \leq \exp \big( - \frac{2}{25} n(n-1)  \big) \leq 0.8522$ for all $n \geq 1$.

All in all, we obtain
\begin{equation}\label{eqn:left.gauss}
	LHS \text{ in }\eqref{eqn:slack_two_repns_gaussian.a}
		\geq \frac{1}{10(1+\epsilon) n\sqrt{n}} \cdot \frac{1}{5} n(n-1) \cdot \frac{ \Pr[ G \in \GG ] - \Pr[ G \not\in \GGb ]}{ \Pr[ G \in \GG ]}
		\geq \frac{\sqrt{n}}{2200 (1+\epsilon)}
\end{equation}
because $\frac{0.893 - 0.8522}{0.893} \geq 1/22$ and $n-1 \geq n/2$ for all $n \geq 1$.

\paragraph{Step 3. An Upper Bound for the Right-hand side of \eqref{eqn:slack_two_repns_gaussian.a}.}
Next, we prove an upper bound on the right-hand side of \eqref{eqn:slack_two_repns_gaussian.a}, which is a function of $N$. 
Note that for the same reason as discussed in Step 3-A of the proof of Theorem \ref{thm:main_full.1}, 
we may assume without loss of generality that the factor functions $f_i, g_i$ satisfy
\begin{equation}\label{eqn:normalization_gaussian}
	\normop{ \bbE_{x \sim \mu}[ f_i(x) ] } = 1, ~~\forall i \in [N]
	\qquad\text{and}\qquad
	\sum_{i=1}^n \tr \big(g_i(G) \big) = 1, ~~\forall G \in \GG.
\end{equation}
For each $i \in [N]$, we define the component functions $\fs_i, \ff_i: H_n \to \bS_+^k$ in the same way as in \eqref{eqn:decomposition}, 
using a fixed threshold $\Lambda$ whose value will be determined later in this proof, cf. Step 3-B of the proof of Theorem \ref{thm:main_full.1}.

By linearity of expectation, we can decompose the expression on the right-hand side of \eqref{eqn:slack_two_repns_gaussian.a} as
\begin{equation}\label{eqn:sharp_flat_gaussian}
	RHS \text{ in }\eqref{eqn:slack_two_repns_gaussian.a}
		= 
		\bbE_{G \sim \cN_0 | \GG}  \sum_{i=1}^N \inner{\bbE_{x \sim \mu}[ q_G(x) \cdot \fs_i(x) ] }{ g_i (G)}
		+
		\bbE_{G \sim \cN_0 | \GG}  \sum_{i=1}^N \inner{\bbE_{x \sim \mu}[ q_G(x) \cdot \ff_i(x) ] }{ g_i (G)}.
\end{equation}
In the two sub-steps below, we prove upper bounds for the two terms on the right hand side separately.

\subparagraph{Step 3-A. Upper Bound on the Contribution of Sharp Components in \eqref{eqn:sharp_flat_gaussian}}
Here we argue that the first term on the right hand side of \eqref{eqn:sharp_flat_gaussian} is bounded from above 
by $ \frac{16 (1+\epsilon) }{\Lambda} k n \sqrt{n} N$. Our argument is based on the following three observations.
\begin{itemize}
	\item
		Let $\supp \fs_i = \{ x \in H_n: \fs_i(x) \neq 0 \}$. Then $| \supp \fs_i | < \frac{k}{\Lambda} ~2^n$ for all $i \in [N]$, 
		cf. Step 3-C of the proof of Theorem \ref{thm:main_full.1}.
	
	\item
		Observe that $\langle \fs_i(x), ~ g_i(G) \rangle \leq \inner{ f_i (x) }{ g_i(G) } \leq s(x, G) \leq 2$ 
		for all $i \in [N]$ and for all $(x, G) \in H_n \times \GG$.
	
	\item
		$q_G(x) = - x^T G x = 8(1+\epsilon) n\sqrt{n} \big( s(x, G) - 1\big) \leq 8(1+\epsilon) n\sqrt{n} $ for all $(x, G) \in H_n \times \GG$.

\end{itemize}
Combining these observations, we can see that for every $G \in \GG$,
\begin{align*}
	\sum_{i=1}^N \inner{\bbE_{x \sim \mu}[ q_G(x) \fs_i(x) ] }{ g_i (G)}
		&= \sum_{i=1}^N \bbE_{x \sim \mu} \Big[ q_G(x) \inner{ \fs_i(x) }{ g_i (G)} \Big]
		\leq \sum_{i=1}^N \frac{\big| \supp \fs_i \big|}{2^n} \cdot 16(1+\epsilon) n \sqrt{n}\\
		&\leq  \frac{16 (1+\epsilon) }{\Lambda} k n \sqrt{n} N.
\end{align*}
This upper bound is independent of $G$, and thus, we get
\begin{equation}\label{eqn:sharp.gauss}
	\bbE_{G \sim \cN_0 | \GG}  \sum_{i=1}^N \inner{\bbE_{x \sim \mu}[ q_G(x) \cdot \fs_i(x) ] }{ g_i (G)}
		\leq  \frac{16 (1+\epsilon) }{\Lambda} k n \sqrt{n} N.
\end{equation}

\subparagraph{Step 3-B. Upper Bound on the Contribution of Flat Components in \eqref{eqn:sharp_flat_gaussian}}
Here we prove an upper bound for the second term in \eqref{eqn:sharp_flat_gaussian}. 
First of all, we observe that for every $G \in \GG$,
\[
	\sum_{i=1}^N \inner{\bbE_{x \sim \mu}[ q_G(x) \cdot \ff_i(x) ] }{ g_i (G)}
		\leq \sum_{i=1}^N \normop{ \bbE_{x \sim \mu}[ q_G(x) \cdot \ff_i(x) ] } \tr ~g_i (G)
		\leq \max_{i \in [N]}  \normop{ \bbE_{x \sim \mu}[ q_G(x) \cdot \ff_i(x) ] }
\]
due to Cauchy-Schwarz inequality and the normalization assumption that $ \sum_{i=1}^N \tr ~g_i (G) = 1, ~ \forall G \in \GG$. 

Given $\enet > 0$, let $\cN$ be an $\enet$-net of $\bbS^{k-1}$ with the smallest possible cardinality. 
It follows from the standard $\epsilon$-net argument that for each $i \in [N]$,
\[
	 \normop{ \bbE_{x \sim \mu}[ q_G(x) \cdot \ff_i(x) ] }
	 	= \sup_{v \in \bbS^{k-1}} v^T \bbE_{x \sim \mu} \big[ q_G(x) \cdot \ff_i(x) \big] v
		\leq \frac{1}{1-2 \enet} \max_{v \in \cN}  \bbE_{x \sim \mu} \big[ q_G(x) \cdot v^T \ff_i(x) v \big].
\]

Next, we observe that if $G \sim \cN_0$, then for every function $f: H_n \to \RR$, the derived random variable 
$\langle f, ~q_G(x) \rangle_{\mu}$ is a centered Gaussian random variable with variance
\begin{align*}
	\bbE_{G \sim \cN_0} \Big[ \langle f, ~q_G(x) \rangle_{\mu}^2 \Big]
		&=\bbE_{G \sim \cN_0} \Big[ \bbE_{x \sim \mu }\big[ f(x) \cdot x^TGx \big]^2 \Big]
		=\bbE_{G \sim \cN_0} \Bigg[ \bbE_{x \sim \mu }\bigg[  f(x) \cdot \Big( \tr~G + 2 \sum_{i < j} G_{ij} x_i x_j \Big)\bigg]^2 \Bigg]\\\
		&= 4 \sum_{i < j} \bbE_{G \sim \cN_0} [ G_{ij}^2] \cdot \bbE_{x \sim \mu} [ f(x) x_i x_j ]^2
		= 2 \sum_{i < j} \langle f(x), ~ x_i x_j \rangle_{\mu}^2
		= 2 \| \proj_2 f \|_2^2.
\end{align*}

Then we use Lemma \ref{lem:maximal} to obtain the following inequalities:
\begin{align*}
	\bbE_{G \sim \cN_0 | \GG}  \sum_{i=1}^N \inner{\bbE_{x \sim \mu}[ q_G(x) \cdot \ff_i(x) ] }{ g_i (G)}
		&\leq \frac{1}{\Pr[G \in \GG]} \bbE_{G \sim \cN_0 }  \sum_{i=1}^N \inner{\bbE_{x \sim \mu}[ q_G(x) \cdot \ff_i(x) ] }{ g_i (G)}	\nonumber\\
		&\leq \frac{1}{\Pr[G \in \GG]} \frac{1}{1-2 \enet} \bbE_{G \sim \cN_0 } \bigg[ \max_{i \in [N] \atop v \in \cN}  \bbE_{x \sim \mu}[ q_G(x) ~ v^T \ff_i(x) v ] \bigg]\\
		&\leq \frac{1}{\Pr[G \in \GG]} \frac{2}{1-2 \enet}  \bigg( \max_{i \in [N] \atop v \in \cN} \| \proj_2 \big( v^T \ff_i v \big) \|_2 \bigg) \sqrt{ \log ( N |\cN| )}.
\end{align*}
We have seen in Step 2 that $\Pr[ G \in \GG ] \geq 1 - \exp\big( - \frac{(5\sqrt{2}-1)^2}{4\pi} \big) \geq 4/5$. 
Also, Lemma \ref{lemma:harmonic} ensures that $\| \proj_2 ( v^T \ff_i v ) \|_2 \leq e \log \Lambda$ for all $(i,v)$, 
provided that we will choose the threshold $\Lambda \geq e$. Lastly, it is well known that $|\cN| \leq \big( 1 + \frac{2}{\enet} \big)^k$. 
In conclusion, we obtain
\begin{equation}\label{eqn:flat.gauss}
	\bbE_{G \sim \cN_0 | \GG}   \sum_{i=1}^N \inner{\bbE_{x \sim \mu}[ q_G(x) \cdot \ff_i(x) ] }{ g_i (G)}
		\leq \frac{5e \log \Lambda}{2(1- 2\enet)} \sqrt{\log \bigg[ N \Big( 1 + \frac{2}{\enet} \Big)^k \bigg] }.
\end{equation}

\paragraph{Step 4. Concluding the Proof}
Lastly, we revisit Eq. \eqref{eqn:slack_two_repns_gaussian.a} to conclude the proof. Recall that we obtained a lower bound for the left-hand side in Step 2, 
cf. \eqref{eqn:left.gauss}, and derived an upper bound for the right-hand side in Step 3, cf. \eqref{eqn:sharp_flat_gaussian}, \eqref{eqn:sharp.gauss}, and \eqref{eqn:flat.gauss}.
Putting these together, we obtain the following inequality that holds for any choice of parameters $\enet$, $\Lambda$ such that $0 < \enet < \frac{1}{2}$ and $\Lambda \geq e$:
\begin{equation}\label{eqn:conclusion.gauss}
	\frac{1}{2200(1+\epsilon)}\sqrt{n}
		\leq \frac{16 (1+\epsilon) }{\Lambda} k n \sqrt{n} N
			+ \frac{5e \log \Lambda}{2(1-2\enet)} \sqrt{\log \bigg[ N \Big( 1 + \frac{2}{\enet} \Big)^k \bigg] }.
\end{equation}

We choose $\enet = 1/4$ for simplicity because optimizing $\enet$ does not make much difference. 
Next, we find $\Lambda$ that minimizes the right-hand side of \eqref{eqn:conclusion.gauss}. 
It is easy to see that the upper bound is minimized (w.r.t. $\Lambda$) at $\Lambda^* = \frac{16(1+\epsilon) k n\sqrt{n} N}{5 e \sqrt{ \log(9^k N ) } }$. 
As a result, we get the following inequality from \eqref{eqn:conclusion.gauss} by choosing $\Lambda = \Lambda^*$ and noticing $N \geq 1$:
\begin{align}
	\frac{1}{11000 e (1+\epsilon)}\sqrt{n}
		&\leq \sqrt{ \log(9^k N) } \cdot \log \bigg( \frac{16(1+\epsilon) k n\sqrt{n} N }{5 \sqrt{ \log(9^k N ) }}  \bigg)	\nonumber\\
		&\leq \sqrt{ \log(9^k N) } \cdot \Bigg[ \log N + \log \bigg( \frac{16(1+\epsilon) k n\sqrt{n} }{5 \sqrt{ \log(9^k ) }}  \bigg) \Bigg].	\label{eqn:cubic.thm.2}
\end{align}

Letting $z = \sqrt{ \log (9^k N) }$, we can see that \eqref{eqn:cubic.thm.2}  is a cubic inequality of the form $z^3 + 3 \beta z \geq 2\alpha$ where
\[
	\alpha = \frac{\sqrt{n}}{22000 e (1+\epsilon)}
	\quad\text{and}\quad	
	\beta = \frac{1}{3} \log \bigg( \frac{16(1+\epsilon) k n\sqrt{n} }{5 \cdot 9^k \sqrt{ \log(9^k ) }}  \bigg).
\]
We want to solve the cubic inequality with an implicit constraint $z > 0$ because $\log(9^k N) > 0$ for all $k, N \geq 1$. 

Note that $\alpha > 0$ for all $\epsilon\geq 0$, $n\geq1$. Observe that the cubic equation $z^3 + 3\beta z - 2\alpha = 0$
always has a unique positive real root when $\alpha > 0$, regardless of the value of $\beta$. Letting $z_*$ denote the positive real root, 
we can see that $\{ z \in \RR: ~ z^3 + 3\beta z \geq 2\alpha, ~ z > 0 \} = \{ z \in \RR: ~ z \geq z_* \}$.
Indeed, we can explicitly write $z_*$ as $z_* = \big( \alpha + \sqrt{ \alpha^2 + \beta^3 }\big)^{1/3} + \big( \alpha - \sqrt{ \alpha^2 + \beta^3 }\big)^{1/3}$, 
due to the general cubic formula, commonly referred to as Cardano's formula. See Appendix \ref{sec:cubic_eq} for more details.

Consequently, we obtain the following lower bound for $N$ by solving \eqref{eqn:quad.thm.2}:
\begin{equation*}
	\log N \geq
		\Big\{ \big( \alpha + \sqrt{ \alpha^2 + \beta^3 }\big)^{1/3} + \big( \alpha - \sqrt{ \alpha^2 + \beta^3 }\big)^{1/3} \Big\}^2
			- 2k \log 3
\end{equation*}
because $\sqrt{ \log(9^k N )} \geq z_*$ if and only if $\log N \geq z_*^2 - 2 k \log 3$.

\end{proof}

\bibliographystyle{alpha}
\bibliography{psd_xc}

\appendix
\section{Proof of Some Lemmas from Section \ref{sec:background}}\label{sec:proof_lemma}
\subsection{Proof of Lemma \ref{lemma:harmonic}}
\begin{proof}
	Let $f = f_0 + f_1 + f_2 + \dots + f_n$ be the Fourier expansion of $f$. Then for $0 \leq \rho \leq 1$,
	\[
		\| \proj_2 f \|_2^2 = \| f_2 \|_2^2 = \frac{1}{\rho^4} \big( \rho^2 \| f_2 \|_2 \big)^2  \leq \frac{1}{\rho^4} \sum_{k=0}^n \rho^{2k} \| f_k \|_2^2 = \frac{1}{\rho^4} \| \Trho f \|_2^2.
	\]
	With $\rho = \sqrt{p-1}$ for $1 \leq p \leq 2$, we have $\| \Trho f \|_2 \leq \| f \|_p$ by hypercontractivity. Then it follows that
	\[
		\| \proj_2 f \|_2 \leq \frac{1}{\rho^2} \| \Trho f \|_2 \leq \frac{1}{p-1} \| f \|_p \leq \frac{1}{p-1} \Lambda^{p-1} 
	\]
	because $\| f \|_p = \bbE[ f^p ]^{\frac{1}{p}} \leq \Lambda^{\frac{p-1}{p}} \bbE[ f ]^{\frac{1}{p}} \leq \Lambda^{\frac{p-1}{p}} \leq \Lambda^{p-1}$. 
	If $\Lambda < e$, we choose $p = 2$ to get $\| \proj_2 f \|_2 \leq \Lambda$. 
	Otherwise, we choose $p = 1 + \frac{1}{\log \Lambda}$ to obtain $\| \proj_2 f \|_2 \leq e \log( \Lambda )$. 
\end{proof}

\subsection{Proof of Lemma \ref{lem:largest_GOE}}
\begin{proof}
	We consider a Gaussian process $(X_v)_{v \in \bbS^{n-1}}$ defined over $ \bbS^{n-1}$ such that 
	$X_v = v^T G v + \gamma$ with $G$ being standard Gaussian in $\bS^n$ and $\gamma \sim N(0,1)$ 
	independent of $G$. It is easy to verify that $\bbE \big[ \sup_{v \in \bbS^{n-1}} \langle v, Gv \rangle \big] 
	= \bbE_{G, \gamma}\big[ \sup_{v \in  \bbS^{n-1}} X_v \big]$.
	Now we introduce an auxiliary Gaussian process $(Y_v)_{v \in  \bbS^{n-1}}$ such that $Y_v = g^Tv$ with $g \sim N(0, 2 I_n)$. 
	Observe that for all $u, v \in  \bbS^{n-1}$,
	(1) $\bbE X_v = \bbE Y_v = 0$; 
	(2) $\bbE X_v^2 = \bbE Y_v^2 = 2$; and
	(3) $\bbE X_u X_v - \bbE Y_u Y_v = (1 - u^T v )^2 \geq 0$.
	Thus, we can apply Slepian's lemma (Lemma \ref{lem:slepian}) to obtain $\bbE_{G, \gamma} \big[ \sup_{v \in \bbS^{n-1}} X_v \big] 
	\leq \bbE_{g \sim N(0, 2I_n)} \big[ \sup_{v \in \bbS^{n-1}} Y_v \big] = \bbE_{g \sim N(0, 2I_n)} \| g \|_2 
	\leq \big( \bbE_{g \sim N(0, 2I_n)} \| g \|_2^2 \big)^{1/2} = \sqrt{2n}$. 
\end{proof}

\subsection{Proof of Lemma \ref{lemma:gaussian_conc}}
\begin{proof}
	Let $h_K(u) := \max_{x \in K} \inner{u}{x} = \| u \|_{K^{\circ}}$ denote the support function of $K$. 
	The function $h_K$ is $L$-Lipschitz with $L = \sup_{x \in K} \| x \|_2$, the diameter of $K$, 
	because for any $u, v \in \RR^d$,
	\[
		\big| h_K(u) - h_K(v) \big| = \big| \|u\|_{K^{\circ}} - \|v\|_{K^{\circ}} \big| \leq \| u - v \|_{K^{\circ}} 
			\leq \sup_{x \in K} \| x \|_2 \|u-v\|_2.
	\]
	Moreover, we can show that $\sup_{x \in K} \| x \|_2 \leq \sqrt{2\pi} w_G(K)$. 
	To see this, let $B(0,R)$ denote the Euclidean ball centered at $0$ with radius $R$. 
	It follows from \cite[Proposition 7.5.2-(e)]{vershynin2018high} that $\sup_{x, y \in K} \| x - y \|_2 \leq \sqrt{2\pi} w_G(K)$.
	Since $0 \in K$, this implies $K \subseteq B(0, \sqrt{2\pi}w_G(K) )$.
	Applying Lemma \ref{lem:gauss_conc} with $f=h_K$ and $\tau = \alpha w_G(K)$ completes the proof.
\end{proof}

\subsection{Proof of Lemma \ref{lemma:subexp_tail}}
\begin{proof}
	Let $A$ be a symmetric $n \times n$ matrix such that $A_{ii} = 0, ~\forall i$ and $A_{ij} = \frac{1}{2} \bbE_{Y \sim \mu(H_n)}[ Y_i Y_j f(Y) ]$ 
	for $i \neq j$. Then we observe that for all $X \in H_n$,
	\[
		\proj_2(f) (X)  
			= \sum_{i = 1 \atop j > i}^n X_i X_j \bbE_{Y \sim \mu(H_n)}[ Y_i Y_j f(Y) ]
			= X^T A X.
	\]
	Note that $X_i$ is sub-Gaussian with sub-Gaussian parameter $1$ for all $i$ because $\bbE[ e^{\lambda X_i} ] 
	= \frac{1}{2}( e^{\lambda} + e^{-\lambda} ) \leq e^{\frac{\lambda^2}{2}}$.
	To conclude the proof, we apply Lemma \ref{lemma:hanson_wright} and observe that $\| A \|_F^2 = \sum_{i = 1 \atop j \neq i}^n \big( \frac{1}{2} 
	\bbE_{X \sim \mu(H_n)}[ X_i X_j f(X) ] \big)^2 = \frac{1}{2} \| \proj_2 f \|_2^2$ and $\| A \|_{op} \leq \| A \|_F$. 
\end{proof}

\subsection{Proof of Lemma \ref{lem:maximal}}
\begin{proof}
	For any $\lambda \in (0, 1/c]$, 
	\begin{align*}
		\bbE\Big[ \max_{i \in [N] } X_i \Big]
			&= \frac{1}{\lambda} \bbE\bigg[ \log \exp \Big( \lambda  \max_{i \in [N]} X_i \Big)  \bigg]	
			\leq \frac{1}{\lambda} \log \bbE\bigg[  \exp \Big( \lambda  \max_{i \in [N]} X_i \Big)  \bigg]	\qquad\because\text{Jensen's inequality}\\
			&= \frac{1}{\lambda} \log \bbE\bigg[   \max_{i \in [N]} \exp \Big( \lambda X_i \Big)  \bigg]	
			\leq \frac{1}{\lambda} \log \bigg(\sum_{i=1}^{N} \bbE\Big[    \exp \big( \lambda X_i \big)  \Big] \bigg)\\	
			&\leq \frac{1}{\lambda} \log \bigg( \sum_{i=1}^{N} e^{\frac{\lambda^2 v}{2}} \bigg)	\qquad\because\text{sub-exponential}\\
			&= \frac{\log N}{\lambda} + \frac{\lambda v}{2}.
	\end{align*}
	It remains to choose $\lambda$ in the interval $(0, 1/c]$ to optimize the upper bound. 
	If $\sqrt{2 \log N / v } \leq 1/c$, then we choose $\lambda = \sqrt{ 2 \log N / v }$ to get $\bbE\big[ \max_{i \in [N] } X_i \big] \leq \sqrt{2 v \log N}$.
	On the other hand, if $\sqrt{2 \log N / v } \leq 1/c$, then we choose $\lambda = 1/c$ to get $\bbE\big[ \max_{i \in [N] } X_i \big] 
	\leq 2c \log N $ since $v/2c \leq \sqrt{2 \log N / v} \leq c \log N$.
\end{proof}

\section{More on Example \ref{exmp:avg_davg} (Ball, Needle, and Pancake)}
Let $B_2^d := \{ x \in \RR^d: \| x \|_2 \leq 1 \}$ denote the $d$-dimensional unit $\ell_2$-ball, and let $B = B_2^d$. Fix $0 < \delta < 1$, and 
let $N = \conv \big\{ B_2^d(0, 1) \cup \{ \pm \frac{1}{\delta} e_1 \} \big\}$ be the `needle' where $e_1 = (1, 0, \dots, 0) \in \RR^d$. 
Lastly, we define the `pancake' $P= \{ x \in B: -\delta \leq x_1 \leq \delta  \}$ where $x_1$ is the first coordinate of $x \in \RR^d$. 
Observe that $N$ and $P$ are the polars of each other, and $B$ is the polar of itself. 

First of all, $w_G(B) = \bbE_g \| g \|_2 = \kappa_d$ and it is known that 
$\sqrt{d - 1/2} \leq \kappa_d \leq \sqrt{d - d/(2d+1)}$, cf. the paragraph below Definition \ref{defn:gaussian_width}. 
Next, we can see that $w_G(N) \geq \frac{1}{\delta}\sqrt{2/\pi}$ because $ \{ \pm \frac{1}{\delta} e_1 \} \subseteq N$ 
and thus, $w_G(N) \geq w_G \big( \{ \pm \frac{1}{\delta} e_1 \} \big)  = \frac{1}{\delta} \bbE_{g \sim \cN(0,1)} |g| = \frac{1}{\delta}\sqrt{2/\pi}$. 
Lastly, observe that $w_G(P) \geq \kappa_{d-1} \geq \sqrt{d - 3/2}$ because $\{0\} \times B_2^{d-1}(0,1) \subseteq P$ and 
$w_G(P) \geq w_G \big( \{0\} \times B_2^{d-1}(0,1) \big) = w_G \big( B_2^{d-1}(0,1) \big) = \kappa_{d-1}$.
	
It follows that $B$ is an $\epsilon$-approximation of $P$ in the average sense for $\epsilon = \kappa_d / \kappa_{d-1} - 1 
\leq 3 / (2d - 3)$. Nevertheless, $B$ is not an $\epsilon'$-approximation of $P$ in the dual-average sense 
unless $\epsilon' \geq \frac{1}{\delta} \sqrt{2/\pi} / \kappa_d - 1\geq \frac{2}{ \delta \sqrt{  \pi(2d-1) } } - 1$, 
which can be made arbitrarily large by choosing small $\delta$. 
For example, if we choose $\delta \leq 1/\sqrt{\pi(2d-1)}$, then $\epsdavg(P,S) \geq 1$ 
whereas $\epsavg(P, S) \leq 3/(2d-3)$ regardless of $\delta$.

\section{Solving the Cubic Inequality $z^3 + \alpha z \geq \beta$ with $\beta > 0$ }\label{sec:cubic_eq}
Consider a cubic equation of the form $z^3 + \alpha z - \beta = 0$, which is commonly referred to as a depressed cubic.
Note that when $\beta > 0$, this cubic equation always has a positive real root. The other two roots can be either negative real roots 
(when $D \leq 0$), or a pair of complex conjugate roots (when $D > 0$), depending on the sign of its discriminant, 
$D = (\alpha/3)^3 + (\beta/2)^2$. 

Indeed, we can find the roots with a generic cubic formula, known as Cardano's formula. 
Let $i = \sqrt{-1}$ denote the imaginary unit, $\omega = \frac{-1 + \sqrt{3} i}{2}$ be a primitive 3rd of unity, and
\begin{equation}\label{eqn:ST}
	T_+ = \sqrt[3]{ \frac{\beta}{2} + \sqrt{ \Big( \frac{\beta}{2} \Big)^2 + \Big( \frac{\alpha}{3} \Big)^3 } }
	\quad
	\text{and}
	\quad
	T_- = \sqrt[3]{ \frac{\beta}{2} - \sqrt{ \Big( \frac{\beta}{2} \Big)^2 + \Big( \frac{\alpha}{3} \Big)^3 } }.
\end{equation}

\paragraph{Case 1: $D > 0$.}
When $D > 0$, the cubic equation $z^3 + \alpha z - \beta = 0$ with $\beta > 0$ has only one real root, $z^*= T_+ + T_-$, 
which turns out to be positive. Thus, the set of real solutions for the cubic inequality $z^3 + \alpha z \geq \beta$ is 
$\{ z \in \RR: z \geq T_+ + T_- \}$.

\paragraph{Case 2: $D \leq 0$.}
There are three real roots for the cubic equation $z^3 + \alpha z - \beta = 0$, which can be written as
\[
	z_1 = T_+ + T_-,	\qquad
	z_2 = \omega T_+ + \omega^2 T_-,	\qquad
	z_3 = \omega^2 T_+ + \omega T_-.
\]
One of these three real roots is positive, and the other two are negative.

Note that \eqref{eqn:ST} now involves complex roots, and the choice of branches might affect the order of the roots, $z_1, z_2, z_3$, 
however, the choice will not change the values of the roots. To avoid any ambiguity in our description, we choose the principal branch so that 
$\Arg{ \sqrt[m]{z} } \in (-\frac{\pi}{m}, \frac{\pi}{m}]$ for any complex number $z$ and any positive integer $m$. 

Observe that $T_+ = \sqrt[3]{ \beta/2 + \sqrt{|D|}i }$ and $\Arg{ T_+ } \in [0, \pi/3)$. Similarly, we can see that $\Arg{ T_- } \in (-\pi/3, 0]$. 
It follows that $T_+ + T_-$ is a positive real number, and thus, the largest real root. Thus, the set of real solutions for the cubic inequality 
$z^3 + \alpha z \geq \beta$ is $\{ z \in \RR: z \geq T_+ + T_- \}$.

\end{document}